\documentclass[10pt]{amsart}
\usepackage{pgf,tikz,pgfplots,color}
\pgfplotsset{compat=1.15}
\usepackage{mathrsfs} 
\usetikzlibrary{arrows}
\usepackage{fancyhdr}
\usepackage{lastpage} 

\usepackage[margin=1in]{geometry}
\usepackage{hyperref}
\hypersetup{
    colorlinks=true,
    linkcolor=blue,
    citecolor=brown,
    filecolor=magenta,
    urlcolor=blue,
}
\usepackage{amsmath}
\usepackage{amsfonts}
\usepackage{amssymb}
\usepackage{amsthm}
\usepackage{setspace}
\usepackage{inputenc}
\usepackage{comment}
\usepackage{mathtools}
\usepackage[shortlabels]{enumitem}

\usepackage{chngcntr}

\allowdisplaybreaks

\numberwithin{equation}{section} 

\newtheorem{thm}{Theorem}[section]
\newtheorem{cor}[thm]{Corollary} 
\newtheorem{prop}[thm]{Proposition}
\newtheorem{lem}[thm]{Lemma}

\theoremstyle{definition}
\newtheorem{defn}[thm]{Definition}

\newtheorem{note}[thm]{Notation}
\newtheorem{conv}[thm]{Convention}
\newtheorem{assu}[thm]{Assumption}

\newtheorem{innercustomgeneric}{\customgenericname}
\providecommand{\customgenericname}{}
\newcommand{\newcustomtheorem}[2]{%
  \newenvironment{#1}[1]
  {%
   \renewcommand\customgenericname{#2}%
   \renewcommand\theinnercustomgeneric{##1}%
   \innercustomgeneric
  }
  {\endinnercustomgeneric}
}

\newcustomtheorem{customprop}{Proposition}
\newcustomtheorem{customlemma}{Lemma}

\theoremstyle{remark}
\newtheorem{rem}[thm]{Remark}

\def\XXint#1#2#3{{\setbox0=\hbox{$#1{#2#3}{\int}$ }
\vcenter{\hbox{$#2#3$ }}\kern-.6\wd0}}

\newcommand{\B}{\mathbb{B}} 
\newcommand{\Z}{\mathbb{Z}}
\newcommand{\R}{\mathbb{R}} 
\newcommand{\C}{\mathbb{C}}

\newcommand{\N}{\mathbb{N}}
\newcommand{\1}{\mathbf{1}} 

\newcommand{\dbar}{\overline{\partial}}

\newcommand{\Bc}{\mathcal{B}}
\newcommand{\Co}{\mathscr{C}}
\newcommand{\Ds}{\mathscr{D}}

\newcommand{\Bs}{\mathscr{B}}

\newcommand{\Ec}{\mathcal{E}}
\newcommand{\Es}{\mathscr{E}}
\newcommand{\Fc}{\mathcal{F}}
\newcommand{\Fs}{\mathscr{F}}

\newcommand{\Hc}{\mathcal{H}}
\newcommand{\Tc}{\mathcal{T}}
\newcommand{\Kb}{\mathbb{K}}
\newcommand{\Kc}{\mathcal{K}}
\newcommand{\Ls}{\mathscr{L}}

\newcommand{\M}{\mathcal{M}}

\newcommand{\Ns}{\mathscr{N}}
\newcommand{\Pc}{\mathcal{P}}

\newcommand{\Rc}{\mathcal{R}}
\newcommand{\Sc}{\mathcal{S}}
\newcommand{\Ss}{\mathscr{S}}

\newcommand{\Uc}{\mathcal{U}}
\newcommand{\Vs}{\mathscr{V}}

\newcommand{\Xs}{\mathscr{X}}
\newcommand{\Ys}{\mathscr{Y}}
\newcommand{\Zs}{\mathscr{Z}}
\newcommand{\id}{\mathrm{id}}
\newcommand{\eps}{\varepsilon}
\newcommand{\bmo}{\mathrm{bmo}}
\newcommand{\BMO}{\mathrm{BMO}}
\newcommand{\dist}{\operatorname{dist}}
\newcommand{\supp}{\operatorname{supp}}
\newcommand{\essup}{\mathop{\operatorname{essup}}}
\newcommand{\loc}{\mathrm{loc}}

\newcommand{\re}{\operatorname{Re}}
\newcommand{\im}{\operatorname{Im}}
\newcommand{\Vol}{\operatorname{Vol}}

\makeatletter
\def\@tocline#1#2#3#4#5#6#7{\relax
  \ifnum #1>\c@tocdepth 
  \else
    \par \addpenalty\@secpenalty\addvspace{#2}%
    \begingroup \hyphenpenalty\@M
    \@ifempty{#4}{%
      \@tempdima\csname r@tocindent\number#1\endcsname\relax
    }{%
      \@tempdima#4\relax
    }%
    \parindent\z@ \leftskip#3\relax \advance\leftskip\@tempdima\relax
    \rightskip\@pnumwidth plus4em \parfillskip-\@pnumwidth
    #5\leavevmode\hskip-\@tempdima
      \ifcase #1
       \or\or \hskip 1em \or \hskip 2em \else \hskip 3em \fi%
      #6\nobreak\relax
    \hfill\hbox to\@pnumwidth{\@tocpagenum{#7}}\par
    \nobreak
    \endgroup
  \fi}
\makeatother

	\title[Universal $\dbar$ Solution on Strongly Pseudoconvex Domains]{A Universal $\overline\partial$ Solution Operator on Nonsmooth Strongly Pseudoconvex Domains}          

\author[]{Liding Yao} 
\address{Liding Yao, Department of Mathematical Sciences,
	Purdue University Fort Wayne, Fort Wayne, IN 46805} 
\email{yao388@purdue.edu}

\subjclass[2020]{32A26 (primary) 32T15, 32W05 and 46E35 (secondary)} 
\keywords{Cauchy-Riemann equations, integral representation, strongly pseudoconvex domains, negative Sobolev spaces}

\begin{document}

\begin{abstract}
    We construct a universal family of $\overline\partial$ homotopy operators on bounded $C^{1,1}$ domains that are locally strongly $\mathbb C$-linearly convexifiable. In particular this includes $C^2$ strongly pseudoconvex domains and $C^{1,1}$ strongly $\mathbb C$-linearly convex domains. The operators satisfy Sobolev estimates $\mathcal H_q:H^{s,p}\to H^{s+1/2,p}$ and H\"older-Zygmund estimates $\mathcal H_q:\mathscr C^s\to\mathscr C^{s+1/2}$ simultaneously for all $s\in\mathbb R$ and $1<p<\infty$. The construction uses a modification of the Leray-Koppelman formula and a refined directional decomposition of the commutator $[\overline\partial,\mathcal E]$.
\end{abstract}
\maketitle

\section{Introduction}





The goal of this paper is to prove the following:
\begin{thm}\label{Thm::MainThm}
    Let $\Omega\subset\C^n$ be a bounded locally strongly $\C$-linearly convexifiable domain with $C^{1,1}$ boundary.
    
    Then there are operators $\Hc_q:\Ss'(\Omega;\wedge^{0,q})\to\Ss'(\Omega;\wedge^{0,q-1})$ that map $(0,q)$-forms to $(0,q-1)$-forms with distributional coefficients for $1\le q\le n$ (we set $\Hc_{n+1}:=0$), such that
\begin{enumerate}[(i)]
    \item{\normalfont(Homotopy formula)}\label{Item::MainThm::Homo} $f=\dbar\Hc_q f+\Hc_{q+1}\dbar f$ for all $1\le q\le n$ and $(0,q)$-forms $f\in\Ss'(\Omega;\wedge^{0,q})$.
    \item{\normalfont{($\frac12$ estimates)}}\label{Item::MainThm::Sob} For every $1\le q\le n$, $s\in\R$ and $p\in(1,\infty)$, we have $\frac12$ estimate on Sobolev spaces $\Hc_q:H^{s,p}(\Omega;\wedge^{0,q})\to H^{s+\frac12,p}(\Omega;\wedge^{0,q-1})$ and on H\"older spaces $\Hc_q:\Co^s(\Omega;\wedge^{0,q})\to \Co^{s+\frac12}(\Omega;\wedge^{0,q-1})$.
    \item\label{Item::MainThm::BddP} The corresponding skew Bergman projection $\Pc f:=f-\Hc_1\dbar f$ is bounded $\Pc:H^{s,p}(\Omega)\to H^{s,p}(\Omega)$ and $\Pc:\Co^s(\Omega)\to\Co^s(\Omega)$ for all $s\in\R$ and $p\in(1,\infty)$. In particular $\Pc:L^p(\Omega)\to L^p(\Omega)$ for $p\in(1,\infty)$.
\end{enumerate}
\end{thm}
Here $\Ss'(\Omega)$ is the space of distributions on $\Omega$ which admit extensions to distributions on $\C^n\simeq\R^{2n}$. See Notation~\ref{Note::Space::Dist} and Lemma~\ref{Lem::Space::SsLim}~\ref{Item::Space::SsLim::Cup}. For the precise definitions of $H^{s,p}$ and $\Co^s$ see Definitions~\ref{Defn::Space::Sob} and \ref{Defn::Space::Hold}. For precise definitions of $\C$-linearly convexifiability, see Definition~\ref{Defn::WeiEst::Convex}. When $\Omega$ has $C^2$ boundary, the definition coincides with the classical strongly pseudoconvexity. In particular Theorem~\ref{Thm::MainThm} includes the case where $\Omega$ is $C^2$ strongly pseudoconvex or $\Omega$ is $C^{1,1}$ strongly $\C$-linearly convex.

Such solution operators $\Hc_q$ are called \textit{universal solution operators}, in the sense that we have one operator which has optimal regularity estimates on $H^{s,p}$ and  $\Co^s$ simultaneously for all $s\in\R$ and $1<p<\infty$. 
In previous studies the constructions are only known for smooth domains, one by Shi and the author \cite{ShiYaoCk} on strongly pseudoconvex domains and the other by the author \cite{YaoCXFinite} on convex domains of finite type.

The terminology \textit{universal} is analogous to that of Rychkov's universal extension operator for Lipschitz domains, $E:\Ss'(\Omega)\to\Ss'(\R^N)$ where $\Omega\subset\R^N$ is a bounded Lipschitz domain, such that $E$ is bounded in all Triebel-Lizorkin spaces $E:\Fs_{pr}^s(\Omega)\to \Fs_{pr}^s(\R^N)$ for all $0<p,r\le\infty$ and $s\in\R$.

We will provide a stronger result in Theorem~\ref{Thm::GenThm}, which unifies the Sobolev and H\"older estimates via Triebel-Lizorkin spaces and includes the $L^p$-$L^q$ estimates for $\Hc_q$. In particular we prove that $\Hc_q:\Fs_{pr}^s(\Omega;\wedge^{0,q})\to \Fs_{pr}^{s+1/2}(\Omega;\wedge^{0,q-1})$ for $p\in[1,\infty]$, $r\in(0,\infty]$ and $s\in\R$. On the contrary, for $0<p<1$, the counterexamples in Appendix~\ref{Section::CounterHardy} show that it is impossible to have a solution operator $\Hc_1$ such that $\Hc_1:H^{s,p}(\Omega;\wedge^{0,1})\to H^{s+\frac12,p}(\Omega)$ in general when $n\ge2$ (See Remark~\ref{Rmk::Space::RmkHspHardy} for the definition of fractional Hardy spaces). Thus the range $p\in[1,\infty]$ is optimal for the $\frac12$-estimates.

As an immediate corollary of Theorem~\ref{Thm::MainThm} we have the following result for the $\dbar$ equation:
\begin{cor}
    Let  $\Omega\subset\C^n$ be a bounded domain, which is either $C^2$ strongly pseudoconvex or $C^{1,1}$ strongly $\C$-linearly convex. Let $1\le q\le n$. For every $(0,q)$-form $f$ on $\Omega$ whose coefficients are extendable distributions such that $\dbar f=0$ in the sense of distributions, there is a $(0,q-1)$ form $u$ on $\Omega$ whose coefficients are also extendable distributions, such that $\dbar u=f$.
    
    Moreover, for every $s\in\R$  (possibly negative) and $1<p<\infty$ there are a $C_{s,p}>0$ and $C_s>0$, such that if the coefficients of $f$ are $H^{s,p}$, then we can choose $u$ whose coefficients are $H^{s+1/2,p}$ with estimate $\|u\|_{H^{s+1/2,p}(\Omega;\wedge^{0,q-1})}\le C_{s,p}\|f\|_{H^{s,p}(\Omega;\wedge^{0,q})}$; if the coefficients of $f$ are $\Co^s$, then we can choose $u$ whose coefficients are $\Co^{s+1/2}$ with estimate $\|u\|_{\Co^{s+1/2}(\Omega;\wedge^{0,q-1})}\le C_s\|f\|_{\Co^s(\Omega;\wedge^{0,q})}$.
\end{cor}

The main novelty of this paper is to construct universal homotopy operators under the fixed \(C^{1,1}\) boundary assumption. The operators are defined on all extendable distributions and satisfy the exact \(\frac12\)-estimates for all real regularity indices. In particular, no additional boundary smoothness is required as the regularity index varies.

To achieve this we further decompose the Bochner-Martinelli formula. In previous constructions only one Leray-Koppelman kernel is used for the homotopy operator, see \eqref{Eqn::Intro::PreviousHT} and \eqref{Eqn::Intro::LKFormula}. In our construction, we decompose the form along different directions, each associated with a different Leray-Koppelman form, see \eqref{Eqn::Intro::RefineRevInt} for a general idea. 

To be more precise, let us recall the previous constructions and their obstructions below.

\subsection{Previous obstructions and new modifications}
Pick an extension operator $\Ec$ that extends functions of $\Omega$ to  a bounded neighborhood $\Uc$ of $\overline\Omega$. Recall the associated Bochner-Martinelli formula that for every $1\le q\le n$ and $(0,q)$ form $f$,
\begin{equation}\label{Eqn::Intro::BMFormula}
        f(z)=\dbar_z\int_{\Uc} B_{q-1}(z,\cdot)\wedge \Ec f+\int_{\Uc} B_q(z,\cdot)\wedge \Ec \dbar f+\int_{\Uc\backslash\overline\Omega}B_q(z,\cdot)\wedge[\dbar,\Ec]f,\quad\text{for every }z\in\Omega.
\end{equation}
This formula is valid when $f$ has distributional coefficients as long as the extension $\Ec f$ is defined as a distribution supported in $\Uc$.

Here $B_q(z,\zeta)$ is the following $(n,n-1)$ form which has degree $(0,q)$ in $z$ and $(n,n-1-q)$ in $\zeta$:
\begin{equation}
    \label{Eqn::Intro::DefB}
    B(z,\zeta):=\frac{b\wedge(\dbar b)^{n-1}}{(2\pi i)^n|\zeta-z|^{2n}}=:\sum_{q=0}^{n-1}B_q(z,\zeta),\qquad\text{where }b(z,\zeta)=\sum_{j=1}^n(\bar\zeta_j-\bar z_j)d\zeta_j.
\end{equation}

The Leray-Koppelman formula (via commutators) takes homotopy on the commutator term through a kernel $K(z,\zeta)=K^Q(z,\zeta)$ for $(z,\zeta)\in\Omega\times(\Uc\backslash\Omega)$ associated to a Leray form $Q(z,\zeta)$ (see \eqref{Eqn::Intro::DefKQ} below): now for $g=[\dbar,\Ec]f$ that satisfies the vanishing condition $g|_\Omega=0$, we have
\begin{equation}\label{Eqn::Intro::LKFormula}
    \int_{\Uc\backslash\overline\Omega}B_{q}(z,\cdot)\wedge g=\dbar_z\int_{\Uc\backslash\overline\Omega} K_{q-1}^Q(z,\cdot)\wedge g-\int_{\Uc\backslash\overline\Omega}K_q^Q(z,\cdot)\wedge \dbar g,\quad\text{for every }z\in\Omega.
\end{equation}

To summarize \eqref{Eqn::Intro::DefB} and \eqref{Eqn::Intro::DefKQ}, for previous studies one used homotopy formulae $f=\dbar\mathbf H_qf+\mathbf H_{q+1}\dbar f$ for $q=1,\dots,n$ and $(0,q)$ form $f$, where $(\mathbf H_q)_{q=1}^n$ are given by
\begin{equation}\label{Eqn::Intro::PreviousHT}
    \mathbf H_qf(z)=\int_\Uc B_{q-1}(z,\cdot)\wedge\Ec f+\int_{\Uc\backslash\overline\Omega}K_{q-1}^Q(z,\cdot)\wedge[\dbar,\Ec]f.
\end{equation}

To be more precise, a \textit{Leray form} is a $(1,0)$ form $Q(z,\zeta)=\sum_{j=1}^nQ_j(z,\zeta)d\zeta_j$ that is holomorphic in $z\in\Omega$ and $Q(z,\zeta)\cdot(\zeta-z)\neq0$ for all $(z,\zeta)\in\Omega\times(\Uc\backslash\Omega)$. The associated Leray-Koppelman kernel $K^Q(z,\zeta)$ is
\begin{gather}
    \label{Eqn::Intro::DefKQ}
    K(z,\zeta)=K^Q(z,\zeta):=\frac{ b(z,\zeta)\wedge Q(z,\zeta)}{(2\pi i)^n}\wedge\sum_{k=1}^{n-1}\frac{(\dbar b(z,\zeta))^{n-1-k}\wedge(\dbar Q(z,\zeta))^{k-1}}{|z-\zeta|^{2(n-k)}(Q(z,\zeta)\cdot(\zeta-z))^k}=:\sum_{q=0}^{n-2}K^Q_q(z,\zeta).
\end{gather}
Here  $K^Q_q$ is the component that has degree $(0,q)$ in $z$ and $(n,n-2-q)$ in $\zeta$.

In the case where $\Omega$ is strongly convex with a defining function $\rho$ we can choose $Q(z,\zeta)=\partial\rho(\zeta)$.

When the domain $\Omega$ is smooth, $Q$ is allowed to be $C^\infty$ in $\zeta$, hence $K^Q(z,\cdot)$ is a bounded smooth function in $\Uc\backslash\Omega$ for every fixed $z\in\Omega$. In previous studies, both in \cite{ShiYaoCk} and \cite{YaoCXFinite}, for a distribution $g$ we were able to build a decomposition $g=\sum_{|\alpha|\le m}D^\alpha g_\alpha$ for a large enough $m$, such that $g_\alpha$ are $L^1$ and $g_\alpha|_\Omega=0$ as well. Therefore integrating by parts we have
\begin{equation}\label{Eqn::Intro::RevInt}
    \int_{\Uc\backslash\overline\Omega} K_{q-1}^Q(z,\cdot)\wedge g=\sum_{|\alpha|\le m}\int_{\Uc\backslash\overline\Omega} K_{q-1}^Q(z,\cdot)\wedge D^\alpha_\zeta g_\alpha=\sum_{|\alpha|\le m}(-1)^{|\alpha|}\int_{\Uc\backslash\overline\Omega}(D^\alpha_\zeta K_{q-1}^Q)(z,\cdot)\wedge g_\alpha.
\end{equation}
The condition $g_\alpha|_\Omega=0$ guarantees that all boundary integrals vanish. Since now $D^\alpha_\zeta K$ is a bounded function (for fixed $z$), the right hand side above is legitimate Lebesgue integral, giving the definedness of $\Hc f$.

However this method does \textbf{not} work for nonsmooth domains, since $Q(z,\zeta)$ cannot be smooth in $\zeta$. As a result $\{(D^\alpha_\zeta K_{q-1}^Q)(z,\cdot)\}_{|\alpha|\le m}\not\subset L^1$ for large enough $m$. In \cite{ShiYaoC2} $\Hc_q$ is only defined on $H^{s,p}$ for $s>1/p$ because we need the property $[\dbar,\Ec]f\in L^1(\Uc\backslash\overline\Omega)$ when $f\in H^{s,p}$ (see \cite[Remark~4.2~(iii)]{ShiYaoC2}). 

To resolve the problem, instead of just modifying \eqref{Eqn::Intro::LKFormula} we further decompose the commutator term in \eqref{Eqn::Intro::BMFormula}. 

Let us work on a local domain $V\cap\Omega$ where $V$ is a neighborhood of a boundary point. We can assume $b\Omega$ is strongly convex and thus we can use $\partial\rho(\zeta)$ for the Leray form where $\rho$ is any ($C^{1,1}$) defining function. We can pick a $\R$-linear basis $(v_1,\dots,v_{2n})$ for $\C^n$ such that $v_j$ points outward with respect to $b\Omega\cap V$. 

Notice that each $1\le j\le 2n$, we can pick a $C^{1,1}$ local defining function $\rho_j$ on $V\cap\Omega$ such that $D_{v_j}\rho_j\equiv1$. Although $K^{\partial\rho_j}$ cannot be a smooth in $\zeta$, it is nevertheless smooth along the direction $D_{v_j}$ in $\zeta$.

Next, instead of using $g=\sum_{|\alpha|\le m}D^\alpha g_\alpha$ in \eqref{Eqn::Intro::RevInt}, we are able to construct a refined decomposition $g=\sum_{j=1}^{2n}D_{v_j}^mg_{m,j}$ such that $g_{m,j}|_\Omega=0$, see Theorem~\ref{Thm::RefAtD}. It is important that the differentiations $D_{v_j}^m$ are all confined to a single direction $v_j$. In this way, for $g=[\dbar,\Ec]f$, by taking $m$ large enough $g_{m,j}$ can all be $L^1$ as well. Therefore,
\begin{align}\notag
    \int_{V\backslash\overline\Omega}B_q(z,\cdot)\wedge g&=\sum_{j=1}^{2n}\int_{V\backslash\overline\Omega}B_q(z,\cdot)\wedge D_{v_j}^mg_{m,j}
    \\\notag
    &=\sum_{j=1}^{2n}\bigg(\dbar_z\int_{V\backslash\overline\Omega}K_{q-1}^{\partial\rho_j}(z,\cdot)\wedge D_{v_j}^mg_{m,j}-\int_{V\backslash\overline\Omega}K_q^{\partial\rho_j}(z,\cdot)\wedge D_{v_j}^m\dbar g_{m,j}\bigg)
    \\\label{Eqn::Intro::RefineRevInt}
    &=(-1)^m\sum_{j=1}^{2n}\bigg(\dbar_z\int_{V\backslash\overline\Omega}(D_{v_j}^mK_{q-1}^{\partial\rho_j})(z,\cdot)\wedge g_{m,j}-\int_{V\backslash\overline\Omega}(D_{v_j}^mK_q^{\partial\rho_j})(z,\cdot)\wedge \dbar g_{m,j}\bigg).
\end{align}
In the construction, the decomposition map $\Sc^{m,j}g:= g_{m,j}$ here are all achieved by certain convolutions, which, in particular, commute with $\dbar$.

Since the integrands in the last row are $L^1$ now, we obtain the local definedness of $(\Hc_q)_q$ on distributions, which, roughly speaking are the following: for $(0,q)$ form $f$ on $\Omega$ such that $f|_{\Omega\backslash\overline V}\equiv0$,
\begin{equation*}
    \Hc_qf(z)=\int_VB_{q-1}(z,\cdot)\wedge\Ec f+\sum_{j=1}^{2n}\int_{V\backslash\overline\Omega}K_{q-1}^{\partial\rho_j}(z,\cdot)\wedge D_{v_j}^m\Sc^{m,j}[\dbar,\Ec]f,\qquad z\in V\cap\Omega.
\end{equation*}
See \eqref{Eqn::LocalHT::H'} for the precise formula.

Now \eqref{Eqn::Intro::RefineRevInt} is only a local construction.
Although it is possible to extend this construction globally, in the paper we choose to construct local operators and glue them together by partition of unity. 

The partition of unity argument is not able to produce the final homotopy formulae, but a \textit{parametrix formula} that has the form $$f=\dbar\Hc'_qf+\Hc'_{q+1}\dbar f+\Rc_qf.$$ Fortunately, the term $(\Rc_q)_{q=1}^n$ are all compact, and in fact gains derivatives in the corresponding function spaces. Roughly speaking we can take $\Hc_q=\Hc'_q\circ(\id-\Rc_q)^{-1}$ for the final operators.

We then use the standard parametrix argument to obtain the global homotopy operators while preserving the regularity estimates. See Section~\ref{Section::PM} for the precise construction.

\subsection{Historical remark} The study of the $\dbar$ equation on strongly pseudoconvex domains has a long history.
There are two major approaches on the solution theory, the $\dbar$-Neumann approach and the integral representation method. 

The \textit{$\dbar$-Neumann problems} defines the \textit{canonical solutions} $\dbar^*N$ that has minimal $L^2$ norm. The $\dbar$-Neumann problem was proposed in \cite{GarabedianSpencer1952} and developed by Morrey \cite{Morrey1958} and Kohn \cite{Kohn1963}. We refer \cite{ChenShawBook} for a detailed discussion.

On smooth strongly pseudoconvex domains, the $\frac12$-estimate were first achieved by Kohn in \cite{Kohn1963} which proves $\dbar^*N_q:H^{s,2}\to H^{s+1/2,2}$ for $s\ge0$. Later for $(0,1)$-forms, Greiner-Stein \cite{GreinerStein1977} proves the boundedness $\dbar^*N_1:H^{s,p}\to H^{s+1/2,p}$ for $s\ge0$ and $1<p<\infty$ and $\dbar^*N_1:\Co^s\to \Co^{s+1/2}$ for $s>0$. For $(0,q)$ forms the optimal Sobolev and H\"older bounds for $\dbar^*N_q$ were done by Chang \cite{Chang1989}. 

It is worth to point out that even for the smooth domains, the canonical solution is never definable on generic distributions with suitably large order, i.e. the boundary value problem is ill-posed, see Lemma~\ref{Lem::Space::IllPosed}.

We use the second approach called \textit{integral representations}, which yield non-canonical solutions but the expressions can be more explicit. We refer \cite{RangeSCVBook} and \cite{LiebMichelBook} for a general discussion. 

This method was introduced by Henkin \cite{Henkin1969} and Grauert \& Lieb \cite{GrauertLieb1971}, both gave $C^0\to C^0$ estimate on strongly pseudoconvex domains with $C^2$ boundaries. Later Kerzman \cite{Kerzman1971} gave the boundedness $L^p\to L^p$ and $L^\infty\to \Co^{1/2-}$ for $q=1$. \O vrelid \cite{Ovrelid1971} obtained estimates $L^p\to L^p$ and $C^k\to C^k$ for $q\ge1$. The sharp $L^\infty\to \Co^{1/2}$ estimate was first achieved by Henkin-Romanov \cite{HenkinRomanov1971} for $q=1$. For $q\ge1$ the $L^\infty\to \Co^{1/2}$ estimate was first done by Range-Siu \cite{RangeSiu1973} where the domains need to be smooth. When the domain is sufficiently smooth, Siu \cite{Siu1974} proved the existence of estimate $C^k\to \Co^{k+1/2}$ on closed $(0,1)$ forms with $b\Omega\in C^{k+4}$; later Lieb-Range \cite{LiebRange1980} proved this estimate on $\dbar$ closed $(0,q)$ forms with $b\Omega\in C^{k+2}$.

Recently Gong \cite{GongHolderSPsiCXC2} constructed homotopy operators with boundedness $\Co^s\to\Co^{s+1/2}$ for all $s>1$ while the strong pseudoconvex $\Omega$ is only $C^2$ smooth. Note that for such domain, the canonical solution operator $\dbar^*N_q$ cannot have $\Co^s$ boundedness for large enough $s$. Later Shi and the author \cite{ShiYaoC2} proved the Sobolev estimate $H^{s,p}\to H^{s+1/2,p}$ for all $1<p<\infty $ and $s>1/p$.

All the above studies on strongly pseudoconvex domains require at least $C^2$ boundary. For the exact $\frac12$ estimate (either on H\"older spaces or on Sobolev spaces) the weakest regularity assumption that is known to be possible is $C^{1,1}$ by Gong-Lanzani \cite{GongLanzaniCConvex} with the domain being strongly $\C$-linearly convex. They proved the H\"older estimate $\Co^s\to\Co^{s+1/2}$ for $s>1$. Note that their method does not yield $\frac12$-estimate on $C^{1,\alpha}$ domains when $\alpha<1$, see \cite[Section~4]{GongLanzaniCConvex} for a brief discussion.

Our Theorem~\ref{Thm::MainThm} implies all the preceding noncanonical existence estimates.

\smallskip

There are also studies on intersections of strongly pseudoconvex domains, where $\Omega=\Omega_1\cap\dots\cap\Omega_m$ has Lipschitz singularity on the intersected boundary. Range-Siu \cite{RangeSiu1973} constructed a solution operator $L^\infty\to\Co^{1/2-}$ on $(0,q)$ forms assuming that the intersections are transversal. When $m=2$ the transversality assumption can be removed, see Peters \cite{Peters1989}. For fixed $k\ge1$, Peters \cite{Peters1991} gave the construction which has the estimate $C^k\to \Co^{k+1/2-}$ on weakly transversal intersections of strongly pseudoconvex domains. See also \cite{MichelPerotti1990}. However, none of these results has sharp $\frac12$ estimate. It is open whether one can construct a solution operator with estimate $L^\infty\to\Co^{1/2}$ on transversal intersections of smooth strongly pseudoconvex domains.

When the domain is not pseudoconvex, it is still possible to construct $\dbar$ solution operators on $(0,q)$ forms with certain fixed $q$. The so-called \textit{$a_q$ condition} was introduced by H\"ormander \cite{Hormander1965}, which is also named as $Z(q)$ domains in e.g. \cite[Page~57]{FollandKohn1972}. In the case $q=1$ this is exactly the strongly pseudoconvexity. On smooth $a_q$ domains the Sobolev and H\"older $\frac12$ estimates on $\dbar$-closed $(0,q)$-forms are given by Greiner-Stein \cite{GreinerStein1977} for $q=1$ and by Beals-Greiner-Stanton \cite{BGS1987} for $q\ge1$. See also the recent work by Gong \cite{GongHolderAq}.

In constructing solution operators, most of the previous constructions used Bochner-Martinelli on $\Omega$ and Leray-Koppelman on boundary $b\Omega$, see e.g. \cite[Theorems~11.1.2 and 11.2.2]{ChenShawBook}: for $(0,q)$ form $f$ on $\overline\Omega$,
\begin{gather*}
    f=\dbar_z\int_\Omega B_{q-1}(z,\cdot)\wedge f+\int_\Omega B_q(z,\cdot)\wedge\dbar f+\int_{b\Omega}B_q(z,\cdot)\wedge f,\qquad z\in\Omega;
    \\
    \int_{b\Omega}B_q(z,\cdot)\wedge f=-\dbar_z\int_{b\Omega} K_{q-1}(z,\cdot)\wedge f-\int_{b\Omega}K_q(z,\cdot)\wedge\dbar f,\qquad z\in\Omega.
\end{gather*}

The usage of extension operator $\Ec$ for \eqref{Eqn::Intro::BMFormula} dates back to Lieb-Range \cite{LiebRange1980}. The commutator method was introduced by Peters \cite{Peters1991} and later  used by Michel \cite{Michel1991}, both of whom used Seeley's half-space extension \cite{Seeley} which works for smooth domains. Recently Gong \cite{GongHolderSPsiCXC2} further developed this method using $\Ec$ for Stein's extension \cite[Chapter VI]{SteinBook} which is defined on Lipschitz domains and extends $H^{s,p}$ and $\Co^s$ for $s>0$. In our case we choose $\Ec$ to be the Rychkov extension operator, which works on Lipschitz domains and extends $H^{s,p}$ and $\Co^s$ for \textbf{all} $s$ (including $s<0$), see \eqref{Eqn::Space::ExtOp}. The Rychkov's extension operator was first introduced to solve the $\dbar$-equation in \cite{ShiYaoC2}.

\subsection{Organization and notations}
The paper is organized as follows. 
In Section~\ref{Section::PM} we use spectral theory to provide a version of parametrix formulae. The construction of homotopy formulae $f=\dbar\Hc_qf+\Hc_{q+1}\dbar f$ for the total domain $\Omega$ is then reduced to the construction of parametrix formulae $f=\dbar\Hc'_qf+\Hc'_{q+1}\dbar f+\Rc_qf$ on local domain $\omega$ which share a piece of boundary to $\Omega$. In Section~\ref{Section::AntiDev} we construct the refined anti-derivative operators associated to the local domain $\omega$ by  modifying \cite[Section~4]{ShiYaoExt}. In Section~\ref{Section::WeiEst} we use the standard argument to prove weighted estimate of the Leray-Koppelman form on local domains which are $C^{1,1}$ strongly $\C$-linearly convex. In Section~\ref{Section::LocalHT} we construct the parametrix formulae $f=\dbar\Hc'_qf+\Hc'_{q+1}\dbar f+\Rc_qf$ using the anti-derivative operators from Section~\ref{Section::AntiDev}, and we prove the regularity estimates for $\Hc'_q,\Rc_q$ using the weighted estimates in Section~\ref{Section::WeiEst}. In Section~\ref{Section::ProofThm} we complete the proof of the main theorem by combining Proposition~\ref{Prop::PM::ConcretePM} and Theorem~\ref{Thm::LocalHT}. In Appendix~\ref{Section::Space} we review all the necessary tools from theory of function spaces, including the extension operator and the regularity estimate of the Bochner-Martinelli integral operator.

In the following we use $\N=\{0,1,2,\dots\}$ as the set of non-negative integers. 

On a complex coordinate system $(z_1=x_1+iy_1,\dots,z_n=x_n+iy_n)$,  $D^\alpha=D^\alpha_z$ denotes the total derivative $\frac{\partial^{|\alpha|}}{\partial x_1^{\alpha_1}\partial y_1^{\alpha_2}\dots\partial x_n^{\alpha_{2n-1}}\partial y_n^{\alpha_{2n}}} $ where $\alpha\in\N^{2n}$.

We use the notation $x \lesssim y$ to denote that $x \leq Cy$, where $C$ is a constant independent of $x,y$, and $x \approx y$ for ``$x \lesssim y$ and $y \lesssim x$''. We use $x\lesssim_\eps y$ to emphasize the dependence of $C$ on the parameter $\eps$.

For a function class $\Xs$ and a domain $U$, we use $\Xs(U)=\Xs(U;\C)$ as the space of complex-valued functions in $U$ that have regularity $\Xs$. We use $\Xs(U;\R)$ if the functions are restricted to being real-valued. We use $\Xs(U;\wedge^{p,q})$ for the space of (complex-valued) $(p,q)$-forms on $U$ that have regularity $\Xs$.

We use the term \textit{parametrix formulae} for which take the form $f=\dbar\Hc'_qf+\Hc'_{q+1}\dbar f+\Rc_qf$ for $q\ge1$ and $f=\Pc' f+\Hc'_1\dbar f+\Rc_0f$ for $q=0$, all of which are only defined on a neighborhood of a boundary point.

    For integration of bi-degree forms we use the following convention from \cite[Page~264]{ChenShawBook}. Note that this is different from \cite[Section~III.1.9]{RangeSCVBook}: 
    \begin{equation}\label{Eqn::Intro::IntConv}
        \int_x\big(u_{IJ}(x,y)dx^I\wedge dy^J\big)=\Big(\int_xu_{IJ}(x,y)dx^I\Big)dy^J.
    \end{equation}

For a bidegree form $T(z,\zeta)$ defined on the product space which has no $dz$ components, we use the decomposition $T=\sum_{q=1}^nT_q$ where $T_q$ has degree $(0,q)$ in $z$. We use $\Tc_pf(z)=\int T_q(z,\cdot)\wedge f$ to denote that $\Tc_p$ is an integral operator on $(0,p)$ forms. In this notation we indicate that $\Tc_p$ maps a $(0,p)$ form to a $(0,q)$ form. As special cases, 
\begin{itemize}
    \item $\Bc_qf(z)=\int B_{q-1}(z,\cdot)\wedge f$ (given in \eqref{Eqn::Space::DefBOp}) to denote the Bochner-Martinelli operator $1\le q\le n$;
    \item $\Kc_qg(z)=\int K_{q-2}(z,\cdot)\wedge g$ (given in \eqref{Eqn::WeiEst::DefKOp}) to denote the Leray-Koppelman operator $2\le q\le n$;
    \item $\Fc g(z)=\Fc_1g(z)=\int F(z,\cdot)\wedge g$ (see \eqref{Eqn::WeiEst::DefFOp}), where $F(z,\zeta)=F_0(z,\zeta)$ is the Cauchy-Fantappi\`e form.
\end{itemize}

\begin{conv}\label{Conv::Intro::MixForm}
    For a mixed degree form $f(\zeta)=\sum_{q=0}^nf_q(\zeta)$ where $f_q(\zeta)=\sum_{|I|=q}f_I(\zeta)d\bar \zeta^I$ has degree $(0,q)$, and for operators $(\Tc_q)_{q=0}^n$ where each $\Tc_q$ is defined on $(0,q)$ forms, we use $\Tc f$ to denote the mixed degree form $\sum_{q=0}^n\Tc_qf_q$.

\end{conv}
In particular $\Bc f=\sum_{q=1}^n\Bc_qf_q$, $\Kc f=\sum_{q=2}^n\Kc_qf_q$, $\Hc f=\sum_{q=1}^n\Hc_qf$, $\Rc f=\sum_{q=0}^n\Rc_qf$, $\Fc f=\Fc f_1$ and $\Pc f=\Pc f_0$. The homotopy formula can be written as $f=\Pc f+\dbar\Hc f+\Hc\dbar f$, and parametrix formula can be written as $f=\Pc'f+\dbar \Hc'f+\Hc'\dbar f+\Rc f$.

On $\R^N$ we use $\{\Sc^{m,k}:(m,k)=(0,0)\text{ or }1\le k\le N,\ m\ge0\}$ to denote a chosen family of convolution operators (constructed in Theorem~\ref{Thm::RefAtD}) which depend on a $\R$-linear basis $(v_1,\dots,v_N)$. There we have $f=\sum_{j=0}^N\Sc^{0,j}f=\Sc^{0,0}f+\sum_{j=1}^ND_{v_j}^m\Sc^{m,j}f$ for all $m\ge1$.

In most cases, we use $\Omega$ to denote a bounded Lipschitz domain, we use $\omega=\{(x',x_N):x_N>\sigma(x')\}$ to denote a special Lipschitz domain (see Definition~\ref{Defn::Space::SpecDom}). When $\omega\subset\R^N$ is fixed we use $\delta(x)=\min(1,\dist(x,b\omega))$ for $x\in\R^N$ to denote a cut-offed distance function.

\section{Parametrix Formulae}~\label{Section::PM}

Ultimately we need the homotopy formulae $f=\dbar\Hc_qf+\Hc_{q+1}\dbar f$ on $\Omega$. The concrete construction in Section~\ref{Section::LocalHT} is local, and gluing these operators produces a parametrix formula with a compact remainder. In this section we show how to correct this remainder while preserving the estimates on all the function spaces under consideration. There are many previous studies using this method; see for example \cite{LeitererParametrix,Laurent-ThiebautParametrix}. 

In our case, we also need to incorporate function spaces with negative indices. We first give a spectral decomposition on a chain of Banach spaces, then construct the correction on the finite-dimensional space, and finally apply the abstract construction to the local operators.



\begin{lem}\label{Lem::PM::RieszThm}
    Let $(\Xs^k)_{k\in\Z}$ be a sequence of Banach spaces such that $\Xs^k\subset \Xs^{k-1}$ are all compact embeddings with respect to their norm topologies for all $k\in\Z$. Let $\Xs^{-\infty}:=\bigcup_{k\in\Z}\Xs^k$ and $\Xs^{+\infty}:=\bigcap_{k\in\Z}\Xs^k$ be their inductive and inverse limits.

    Let $R:\Xs^{-\infty}\to \Xs^{-\infty}$ be a linear map such that $R|_{\Xs^k}:\Xs^k\to \Xs^{k+1}$ are bounded for all $k\in\Z$. Then there is a decomposition $\Xs^{-\infty}=\Ns\oplus \Ls$ of $R$-invariant subspaces, such that
    \begin{enumerate}[(i)]
        \item\label{Item::PM::RieszThm::N} 
        
        $\Ns=\bigcup_{l=1}^\infty\{f\in \Xs^{-\infty}:(\id-R)^lf=0\}\subset \Xs^{+\infty}$ is finite dimensional. 
        \item\label{Item::PM::RieszThm::Gamma} Let $\Gamma:\Xs^{-\infty}\twoheadrightarrow \Ls$ be the associated projection. Then $\id-\Gamma R:\Xs^{-\infty}\to \Xs^{-\infty}$ is invertible, and $(\id-\Gamma R)^{-1}:\Xs^k\to \Xs^k$ are bounded for all $k\in\Z$.
        \item\label{Item::PM::RieszThm::MoreBdd} Suppose further that there is a quasi-Banach space $\Ys$ such that $\Xs^M\subset\Ys\subset\Xs^{-M}$ for some $M\ge0$, where inclusions are both continuous embedding. Suppose $R:\Ys\to\Ys$ is bounded. Then $\id-\Gamma R$ is invertible on $\Ys$ and  $(\id-\Gamma R)^{-1}:\Ys\to \Ys$ is bounded.
    \end{enumerate}
\end{lem}
\begin{proof}
    Recall the Riesz Theorem of compact operators, where we use the result from \cite[Theorem~3.3]{KressSpectralCompact}:
    
    \medskip\textit{Let $X$ be a Banach space and $T:X\to X$ be a compact operator. Then for the eigenvalue $1$, the generalized eigenspace $N:=\bigcup_{l=1}^\infty\ker((\id_X-T)^l$) is finite dimensional. Moreover let $\tilde m\ge0$ be the smallest integer such that $N=\ker((\id_X-T)^{\tilde m})$, then $L:=\im ((\id_X-T)^{\tilde m})$ satisfies $X=N\oplus L$ and that $(\id-T)|_L:L\to L$ is invertible. Moreover $N=\ker((\id_X-T)^m)$, then $L:=\im ((\id_X-T)^m)$ for all $m\ge\tilde m$.}

    \medskip
    Now let us take $X=\Xs^k$ for each $k$. Since $\Xs^{k+1}\subset\Xs^k$ is compact embedding, $R|_{\Xs^k}:\Xs^k\to\Xs^k$ is a compact operator. Thus we can find $\tilde m_k\ge1$ such that there is a decomposition $\Xs^k=\Ns^k\oplus \Ls^k$, where $\Ns^k:=\ker(\id_{\Xs^k}-R|_{\Xs^k})^m$ is finite dimensional and $\Ls^k:=\im(\id_{\Xs^k}-R|_{\Xs^k})^m$ for all $m\ge\tilde m_k$. Let us temporarily denote the associated projections 
    $$\Theta_k:\Xs^k\twoheadrightarrow \Ns^k,\qquad\Gamma_k:\Xs^k\twoheadrightarrow \Ls^k,\qquad k\in\Z.$$

    Clearly $\Ns^{k+1}\subseteq \Ns^k$. For every $f\in \Ns^k$ we can write $f=f-(\id-R)^{\tilde m_k}f$. Since $\id-(\id-R)^{\tilde m_k}$ is a linear combination of $R,\dots,R^{\tilde m_k}$, all of which map $\Xs^k\to\Xs^{k+1}$, we get $f\in\Xs^{k+1}$ which means $f\in\Ns^{k+1}$. This is to say $\Ns^{k+1}=\Ns^k$. Take $k\to\pm\infty$ we conclude that $\Ns=\lim_{k\to+\infty}\Ns^k\subset\bigcap_{l\ge0}\Xs^l=\Xs^{+\infty}$.
    Therefore there is a common $\tilde m\le \dim\Ns$ such that $\Ns=\ker((\id_{\Xs^{-\infty}}-R)^m)$ for $m\ge\tilde m$. 
    By the Riesz Theorem again $\Ls^k=\im((\id_{\Xs^{k}}-R|_{\Xs^k})^m)$ for all such $m\ge\tilde m$. Let $k\to-\infty$ we get $\Ls=\im((\id_{\Xs^{-\infty}}-R)^m)$.    This concludes \ref{Item::PM::RieszThm::N}.

    Next we show that $\Gamma_{k+1}=\Gamma_k|_{\Xs^{k+1}}$. For every $f\in \Xs^{k+1}$ we have $f=n_k+s_k=n_{k+1}+s_{k+1}$ where $n_j=\Theta_jf$ and $s_j=\Gamma_jf$ for $j\in\{k,k+1\}$. Clearly $\Ls^k\ni s_{k+1}-s_k=n_k-n_{k+1}\in \Ns$, which means $s_{k+1}=s_k$. Therefore we can define $\Gamma=\varinjlim\Gamma_k$ as the direct limit with respect to the chain $\dots\hookrightarrow \Xs^{k+1}\hookrightarrow\Xs^k\hookrightarrow\dots$. Moreover $\Theta:=\id_{\Xs^{-\infty}}-\Gamma$ defines the projection $\Xs^{-\infty}\twoheadrightarrow \Ns$.
    
    Since $1$ is not a spectral point of $\Gamma_kR$, we see that $(\id-\Gamma R)^{-1}|_{\Xs^k}=(\id_{\Xs^k}-\Gamma_kR)^{-1}:\Xs^k\to \Xs^k$ are bounded for all $k$. Take $k\to-\infty$, we get $\id-\Gamma R$ is invertible on $\Xs^{-\infty}$, giving \ref{Item::PM::RieszThm::Gamma}.

    
    For \ref{Item::PM::RieszThm::MoreBdd}, now assume $M\ge0$ and $\Xs^M\subset\Ys\subset\Xs^{-M}$ be a quasi-Banach space. Recall that $\Theta:\Ys\subset\Xs^{-\infty}\twoheadrightarrow\Ns\subset\Ys$ and $\Gamma=\id-\Theta$. Therefore, $\Gamma|_\Ys:\Ys\twoheadrightarrow\Ys$ is also bounded. 

    Suppose $R:\Ys\to\Ys$ is bounded, thus $\Gamma R:\Ys\to\Ys$ is also bounded. We write $(\id-\Gamma R)^{-1}=\sum_{k=0}^{2M-1}(\Gamma R)^k+(\id-\Gamma R)^{-1}(\Gamma R)^{2M}$. By assumption $(\Gamma R)^k:\Ys\to\Ys$ is bounded and $(\Gamma R)^{2M}:\Ys\subset\Xs^{-M}\to\Xs^M$. By \ref{Item::PM::RieszThm::Gamma} $(\id-\Gamma R)^{-1}:\Xs^M\to\Xs^M\subset\Ys$ is bounded. Together we conclude that $(\id-\Gamma R)^{-1}:\Ys\to\Ys$ is bounded, proving \ref{Item::PM::RieszThm::MoreBdd}.
\end{proof}
\begin{rem}\label{Rmk::PM::Topo}
    By endowing with the canonical initial topology $\Xs^{+\infty}$ is a Fr\'echet space; by endowing with the canonical final topology $\Xs^{-\infty}$ is a Silva space. Specifically, for the following linear maps $T$:
    \begin{itemize}
        \item $T:\Xs^{-\infty}\to\Xs^{-\infty}$ is continuous if and only if, for every $M\ge1$ there is a $N\ge1$ such that the restriction $T:\Xs^{-M}\to\Xs^{-N}$ is bounded.
        \item $T:\Xs^{+\infty}\to\Xs^{+\infty}$  is continuous if and only if, for every $M\ge1$ there are $N\ge1$ and $C>0$ such that $\|Tf\|_{\Xs^M}\le C\|f\|_{\Xs^N}$ for all $f\in\Xs^{+\infty}$.
        \item $T:\Xs^{-\infty}\to\Xs^{+\infty}$  is continuous if and only if, for every $M,N\ge1$, $T:\Xs^{-M}\to\Xs^N$ is bounded.
    \end{itemize}
    We refer the reader to \cite{Komatsu67} for more details.
\end{rem}
\begin{rem}
    By a more careful analysis, one can show that the same result is true if we only assume that $R|_{\Xs^k}:\Xs^k\to\Xs^k$ is merely compact for all $k$. In our application, $R$ gains derivatives in Sobolev spaces, which is stronger than being compact.

    The only reason we consider $\Ys$ to be quasi-Banach is because the space $\Fs_{pr}^s$ in Definition~\ref{Defn::Space::TLSpace} is not normed but quasi-normed for $p<1$ or $r<1$. To prove Theorem~\ref{Thm::MainThm} we only need to consider Banach spaces.
\end{rem}

    

Lemma~\ref{Lem::PM::RieszThm} leads to the following parametrix formulae. For future applications we state it in a more general setting:

\begin{prop}[Parametrix formulae, abstract version]\label{Prop::PM::AbsPM}
    For $0\le q\le n$, let $(\Xs^k_q)_{k\in\Z} $ be a sequence of Banach spaces such that $\Xs_q^k\subset\Xs_q^{k-1}$ are  compact embeddings and have dense images for all $k\in\Z$. Denote by $\Xs^{-\infty}_q=\bigcup_{k\in\Z}\Xs^k_q$ and $\Xs^{+\infty}_q=\bigcap_{k\in\Z}\Xs^k_q$ as before. For $0\le q\le n$, let $R_q:\Xs^{-\infty}_q\to \Xs^{-\infty}_q$ be linear maps such that:
    \begin{enumerate}[(A)]
        \item\label{Item::PM::AbsPM::RCpt} $R_q:\Xs_q^k\to\Xs_q^{k+1}$ is bounded for all $k\in\Z$.
    \end{enumerate}
    Let $\Gamma_q:\Xs^{-\infty}_q\to \Xs^{-\infty}_q$ be the corresponding spectral projections from Lemma~\ref{Lem::PM::RieszThm}~\ref{Item::PM::RieszThm::Gamma} and let $\Ns_q=\im(\id_{\Xs^{-\infty}_q}-\Gamma_q)$ be from Lemma~\ref{Lem::PM::RieszThm}~\ref{Item::PM::RieszThm::N}. In particular $(\id-\Gamma_qR_q)^{-1}:\Xs_q^k\to\Xs_q^k$ are bounded for all $k\in\Z$.

    Let $P':\Xs^{-\infty}_0\to \Xs^{-\infty}_0$, $H'_q:\Xs^{-\infty}_q\to \Xs^{-\infty}_{q-1}$ and $D_{q-1}:\Xs^{-\infty}_{q-1}\to\Xs^{-\infty}_q$ for $1\le q\le n$ be continuous linear maps (where we set $H'_{n+1}=D_n=0$) in the sense of Remark~\ref{Rmk::PM::Topo} such that
    \begin{enumerate}[(A)]\setcounter{enumi}{1}
        
        \item\label{Item::PM::AbsPM::D} $D_qD_{q-1}=0$ for all $1\le q\le n$. 
        \item\label{Item::PM::AbsPM::ParaHT} $\id_{\Xs^{-\infty}_0}=P'+H_1'D_0+R_0$ and $\id_{\Xs^{-\infty}_q}=D_{q-1}H_q'+H_{q+1}'D_q+R_q$ for $1\le q\le n$.
        \item\label{Item::PM::AbsPM::DP} $D_0:\Xs_0^{+\infty}\to\Xs_1^{+\infty}$, $D_0P':\Xs^{-\infty}_0\to\Xs^{+\infty}_1$ and $H'_1:\Xs^{+\infty}_1\to\Xs^{+\infty}_0$ are all continuous in the sense of Remark~\ref{Rmk::PM::Topo}.
        \item\label{Item::PM::AbsPM::0Cohomo} If $q\ge1$ and $f_q\in\Ns_q$ satisfies $D_qf_q=0$, then there exists $u_{q-1}\in\Xs_{q-1}^{-\infty}$ such that $D_{q-1}u_{q-1}=f_q$.
    \end{enumerate}
    
    Then we have the following results.
    \begin{enumerate}[(i)]
        \item\label{Item::PM::AbsPM::DN} $D_q:\Ns_q\to\Ns_{q+1}$ for all $1\le q\le n-1$. 
        \item\label{Item::PM::AbsPM::TildeH} There exist linear operators $\widetilde H_1:\Ns_1\to\Xs_0^{+\infty}$ and $\widetilde H_q:\Ns_q\to\Ns_{q-1}$ for $2\le q\le n$, such that $\id_{\Ns_q}=D_{q-1}\widetilde H_q+\widetilde H_{q+1}D_q$ for all $1\le q\le n$ (still we set $\widetilde H_{n+1}=0$).
    \end{enumerate}
    
    Let us then define the operators $P,H_1,\dots,H_n$ by
    \begin{equation}\label{Eqn::PM::AbsPM::PH}
        H_q:=(\widetilde H_q(\id_{\Xs^{-\infty}_q}-\Gamma_q)R_q+H'_q)\circ(\id_{\Xs^{-\infty}_q}-\Gamma_q R_q)^{-1},\quad P:=\id_{\Xs^{-\infty}_0}-H_1D_0.
    \end{equation}

    Then the following results hold.
    \begin{enumerate}[(i)]\setcounter{enumi}{2}
        \item\label{Item::PM::AbsPM::HT} $P:\Xs^{-\infty}_0\to \Xs^{-\infty}_0$ and $H_q:\Xs^{-\infty}_q\to \Xs^{-\infty}_{q-1}$ are defined, and satisfy $\id_{\Xs^{-\infty}_0}=P+H_1D_0$ and $\id_{\Xs^{-\infty}_q}=D_{q-1}H_q+H_{q+1}D_q$ for $1\le q\le n$.
    \item\label{Item::PM::AbsPM::BddH} Let $1\le q\le n$. Suppose there are $M\in\Z_+$ and quasi-Banach spaces $\Xs^{M}_q\subseteq\Ys_q\subseteq \Xs^{-M}_q$ and  $\Xs^{+\infty}_{q-1}\subseteq\Zs_{q-1}\subseteq \Xs^{-M}_{q-1}$ such that $R_q:\Ys_q\to\Ys_q$ and $H'_q:\Ys_q\to\Zs_{q-1}$ are both bounded. Then $H_q:\Ys_q\to\Zs_{q-1}$ is also bounded.
    \item\label{Item::PM::AbsPM::BddP} Let $q=0$. Suppose there are $M\in\Z_+$ and quasi-Banach spaces $\Xs^M_0\subseteq\Ys_0\subseteq \Xs^{-M}_0$ and $\Xs^M_0\subseteq\Zs_0\subseteq \Xs^{-M}_0$ such that $R_0:\Ys_0\to\Ys_0$ and $P':\Ys_0\to\Zs_0$ are both bounded. Then $P:\Ys_0\to\Zs_0$ is also bounded.
    \end{enumerate}
\end{prop}
\begin{proof}
    It is more convenient to use $\id_q=\id_{\Xs^{-\infty}_q}$ and we denote $\Theta_q:=\id_q-\Gamma_q$.

    Firstly by \ref{Item::PM::AbsPM::ParaHT}, by setting $\id=\sum_{q=0}^n\id_q$, $H'=\sum_{q=1}^nH'_q$, $R=\sum_{q=0}^nR_q$ and $D=\sum_{q=0}^{n-1}D_q$, we have $\id=P'+H'D+DH'+R$. By \ref{Item::PM::AbsPM::D} $D^2=0$, therefore
    \begin{equation*}
        RD=D-P'D-H'D^2-DH'D=D-DH'D=D-DH'D-D^2H'=DR+DP'.
    \end{equation*}
    Decomposing its degrees we see that,
    \begin{equation}
        \label{Eqn::PM::AbsPM::RD} R_{q+1}D_q=D_qR_q\text{ on }\Xs^{-\infty}_q\text{ for }1\le q\le n;\quad\text{ and }R_1D_0-D_0R_0=D_0P'\text{ on }\Xs^{-\infty}_0.
    \end{equation}

    Recall from Lemma~\ref{Lem::PM::RieszThm}  $\Ns_q=\bigcup_{l=1}^\infty\ker((\id_q-R_q)^l)\subseteq\Xs^{+\infty}_q$ is finite dimensional. For $f\in\Ns_q$ we have $(\id_q-R_q)^mf=0$ for large enough $m\ge1$. When $q\ge1$, by \eqref{Eqn::PM::AbsPM::RD} we have $(\id_{q+1}-R_{q+1})^mD_qf=D_q(\id_q-R_q)^mf=0$, giving $D_qf\in\Ns_{q+1}$. This proves \ref{Item::PM::AbsPM::DN}. Note that for $\Ls_q=\im(\id_q-R_q)^m$ the same argument yields $D_q:\Ls_q\to\Ls_{q+1}$ for all $q\ge1$, hence (again we use $\Gamma_{n+1}=\Theta_{n+1}=0$, $D_n=0$, $\Xs_{n+1}^{-\infty}=0$)
    \begin{equation}\label{Eqn::PM::AbsPM::DGammaComm}
        D_q\Gamma_q=\Gamma_{q+1}D_q,\qquad D_q\Theta_q=\Theta_{q+1}D_q\qquad\text{on }\Xs_q^{-\infty}\to\Xs_{q+1}^{-\infty}\qquad\text{for all }q\ge1.
    \end{equation}
    Note that by \ref{Item::PM::AbsPM::RCpt}, definition of $D_q$, and Remark~\ref{Rmk::PM::Topo}, all maps in \eqref{Eqn::PM::AbsPM::DGammaComm} are continuous linear as $\Xs^{-\infty}\to\Xs^{-\infty}$.

    By assumption \ref{Item::PM::AbsPM::0Cohomo} there are linear maps $\widetilde H'_q:\Ns_q\cap\ker D_q\to\Xs_{q-1}^{-\infty}$ for $q\ge1$ such that $D_{q-1}\widetilde H'_qf_q=f_q$ for every $f_q\in\Ns_q\cap\ker D_q$. By \eqref{Eqn::PM::AbsPM::DGammaComm} we have
    $$f_q=\Theta_qf_q=\Theta_qD_{q-1}\widetilde H_q'f_q=D_{q-1}(\Theta_{q-1}\widetilde H_q')f_q,\qquad\text{for all }q\ge2\text{ and }f_q\in\Ns_q\cap\ker D_q.$$
    By composition we have $\Theta_{q-1}\widetilde H_q':\Ns_q\cap\ker D_q\to\Ns_{q-1}$ for $q\ge2$. Take $\Vs_q:=\Theta_q\widetilde H_{q+1}'(\Ns_{q+1}\cap\ker D_{q+1})$ for $1\le q\le n$ (we use $\Vs_n:=0$ and $\widetilde H'_{n+1}:=0$) we get the direct sum decomposition $$\Ns_q=\Vs_q\oplus(\Ns_q\cap\ker D_q)\quad\text{via}\quad \Ns_q\ni f_q=\Theta_q\widetilde H_{q+1}'D_qf_q+(f_q-\Theta_q\widetilde H_{q+1}'D_qf_q),\qquad 1\le q\le n.$$
    For $q\ge2$, set $\widetilde H_q:=(\Theta_{q-1}\widetilde H_q')\oplus0_{\Vs_q}$ we get $\widetilde H_q:\Ns_q\to\Ns_{q-1}$ and $\id_{\Ns_q}=D_{q-1}\widetilde H_q+\widetilde H_{q+1}D_q$ for all $q\ge2$.

    When $q=1$ it is likely that $D_0:\Ns_0\not\to\Ns_1$. Nevertheless, we claim that $\Ns_1\cap\ker D_1\subset D_0(\Xs_0^{+\infty})$ holds.
    
    Since $\Ns_1$ is finite dimensional, we can pick a basis $g_1,\dots,g_r$ for $\Ns_1\cap\ker D_1$ (where $r=\dim(\Ns_1\cap\ker D_1)$). Let $v_1,\dots,v_r\in\Xs_0^{-\infty}$ be such that $D_0v_l=g_l$ for $1\le l\le r$. Since $\Xs^{-\infty}_0=\bigcup_{k\le 0}\Xs^k_0$ there is a $M\ge0$ such that $v_1,\dots,v_r\in\Xs^{-M}_0$. Since $\Xs_0^k\hookrightarrow\Xs_0^{k-1}$ are dense embeddings for all $k\in\Z$, we see that $\Xs_0^{+\infty}\subset\Xs_0^{-M}$ is dense, so we can find sequences $(v^\mu_l)_{\mu=1}^\infty\subset\Xs_0^{+\infty}$ such that $\lim_{\mu\to\infty}v^\mu_l=v_l$ in $\Xs^{-M}_0$ for all $1\le l\le r$. Let $w^\mu_l:=v^\mu_l-H_1'(\id_1-\Gamma_1R_1)^{-1}\Gamma_1D_0v^\mu_l$. By \ref{Item::PM::AbsPM::DP} with compositions we have $w^\mu_l\in\Xs_0^{+\infty}$ for all $\mu\ge1$ and $1\le l\le r$. On the other hand, by \ref{Item::PM::AbsPM::ParaHT}, \eqref{Eqn::PM::AbsPM::RD} and \ref{Item::PM::AbsPM::D},
    \begin{align*}
        D_0w^\mu_l&=D_0v^\mu_l-D_0H_1'(\id_1-\Gamma_1R_1)^{-1}\Gamma_1D_0v^\mu_l=(\Theta_1+\Gamma_1)D_0v^\mu_l+(R_1-\id_1+H_2'D_1)(\id_1-\Gamma_1R_1)^{-1}\Gamma_1D_0v^\mu_l
        \\
        &=(\Theta_1+\Gamma_1)D_0v^\mu_l+(\id_1-\Gamma_1R_1)^{-1}\Gamma_1(R_1-\id_1)D_0v^\mu_l+H_2'(\id_2-\Gamma_2R_2)^{-1}\Gamma_2D_1D_0v^\mu_l=\Theta_1D_0v^\mu_l\in\Ns_1.
    \end{align*}
    Now $D_0w^\mu_l\in\Ns_1\cap\im D_0\subseteq\Ns_1\cap\ker D_1$ by \ref{Item::PM::AbsPM::D}.
    Note that $\Theta_1D_0v_l=\Theta_1g_l=g_l$, so by continuity $\lim_{\mu\to\infty}D_0w^\mu_l=g_l$ as well for $1\le l\le r$. Therefore for a large enough $\mu_0$,  $D_0w^{\mu_0}_1,\dots,D_0w^{\mu_0}_r$ are linearly independent and hence form a basis of $\Ns_1\cap\ker D_1$. 
    
    Now we obtain the claim $\Ns_1\cap\ker D_1\subset D_0(\Xs_0^{+\infty})$. Moreover $\widetilde H_1\big(D_0w^{\mu_0}_l\big):=w^{\mu_0}_l$ for $1\le l\le r$ and $\widetilde H_1|_{\Vs_1}:=0$ raise to a well-defined map $\widetilde H_1:\Ns_1\to\Xs_0^{+\infty}$ such that $\id_{\Ns_1\cap\ker D_1}=D_0\widetilde H_1$. Since we already have $\id_{\Vs_1}=\widetilde H_2D_1$, we get $\id_{\Ns_1}=D_0\widetilde H_1+\widetilde H_2D_1$ and complete the proof of \ref{Item::PM::AbsPM::TildeH}.

    Recall from Lemma~\ref{Lem::PM::RieszThm}~\ref{Item::PM::RieszThm::Gamma}, $\id_q-\Gamma_qR_q$ is invertible on $\Xs^{-\infty}_q$. We see that $P$ and $H_q$ in \eqref{Eqn::PM::AbsPM::PH} are all defined.
    Clearly $\id_0=P+H_1D_0$ since we define $P$ by this way. 
    
    Next we check $\id_q=D_{q-1}H_q+H_{q+1}D_q$. Note that by \eqref{Eqn::PM::AbsPM::DGammaComm} we have $D_q(\id_q-\Gamma_qR_q)=(\id_{q+1}-\Gamma_{q+1}R_{q+1})D_q$. Multiplying $(\id_{q+1}-\Gamma_{q+1}R_{q+1})^{-1}$ and $(\id_q-\Gamma_qR_q)^{-1}$ we get 
    \begin{equation*}
        (\id_{q+1}-\Gamma_{q+1}R_{q+1})^{-1}D_q=D_q(\id_q-\Gamma_qR_q)^{-1},\quad 1\le q\le n.
    \end{equation*}
    Therefore, for $1\le q\le n$,
    \begin{align*}
        \id_q&=(\id_q-\Gamma_q R_q)(\id_q-\Gamma_q R_q)^{-1}=(D_{q-1}H'_q+H_{q+1}'D_q+R_q-\Gamma_q R_q)(\id_q-\Gamma_q R_q)^{-1}
        \\
        &=\big(D_{q-1}H'_q+H_{q+1}'D_q+(D_{q-1}\widetilde H_q+\widetilde H_{q+1}D_q)\Theta_qR_q\big)(\id_q-\Gamma_q R_q)^{-1}
        \\
        &=D_{q-1}(H'_q+\widetilde H_q\Theta_qR_q)(\id_q-\Gamma_q R_q)^{-1}+(H'_{q+1}+\widetilde H_{q+1}\Theta_{q+1}R_{q+1})D_q(\id_q-\Gamma_qR_q)^{-1}
        \\
        &=D_{q-1}(H'_q+\widetilde H_q\Theta_qR_q)(\id_q-\Gamma_q R_q)^{-1}+(H'_{q+1}+\widetilde H_{q+1}\Theta_{q+1}R_{q+1})(\id_{q+1}-\Gamma_{q+1} R_{q+1})^{-1}D_q
        \\
        &=D_{q-1}H_q+H_{q+1}D_q.
    \end{align*}
    This proves \ref{Item::PM::AbsPM::HT}.

    When $R_q:\Ys_q\to\Ys_q$ and $H'_q:\Ys_q\to\Zs_{q-1}$, by Lemma~\ref{Lem::PM::RieszThm}~\ref{Item::PM::RieszThm::MoreBdd} $(\id_q-\Gamma_q R_q)^{-1}:\Ys_q\to\Ys_q$ is also bounded. The result \ref{Item::PM::AbsPM::BddH} then follows from \eqref{Eqn::PM::AbsPM::PH} with standard compositions.

    Finally we prove \ref{Item::PM::AbsPM::BddP} using the assumption $R_1D_0-D_0R_0=D_0P':\Xs^{-\infty}_0\to\Xs^{+\infty}_1$. 
    
    For convenience we denote $S_0:=D_0(\id_0-\Gamma_0R_0)-(\id_1-\Gamma_1R_1)D_0$.    
    Recall that $\Theta_j=\id_j-\Gamma_j$ and $\Theta_j:\Xs^{-\infty}_j\to\Ns_j\subset\Xs_j^{+\infty}$ for $j=0,1$, which means 
    \begin{equation*}
        S_0=\Gamma_1R_1D_0-D_0\Gamma_0R_0=R_1D_0-\Theta_1R_1D_0+D_0\Theta_0R_0-D_0R_0=D_0P'-\Theta_1R_1D_0+D_0\Theta_0R_0:\Xs^{-\infty}_0\to\Xs^{+\infty}_1.
    \end{equation*}

    On the other hand, by multiplying $(\id_1-\Gamma_1R_1)^{-1}$ and $(\id_0-\Gamma_0R_0)^{-1}$ to $S_0$ we have
    \begin{align*}
        (\id_1-\Gamma_1R_1)^{-1}D_0-D_0(\id_0-\Gamma_0R_0)^{-1}=(\id_1-\Gamma_1R_1)^{-1}S_0(\id_0-\Gamma_0R_0)^{-1}:\Xs^{-\infty}_0\to\Xs^{+\infty}_1.
    \end{align*}
    Therefore, using $\Theta_j:\Xs^{-\infty}_j\to\Xs^{+\infty}_j$ again we have
    \begin{align*}
        &P=\id_0-H_1D_0=(\id_0-\Gamma_0R_0)(\id_0-\Gamma_0R_0)^{-1}-(H'_1+\widetilde H_1\Theta_1R_1)(\id_1-\Gamma_1R_1)^{-1}D_0
        \\
        =&(\id_0-\Gamma_0R_0-H'_1D_0)(\id_0-\Gamma_0R_0)^{-1}-H'_1(\id_1-\Gamma_1R_1)^{-1}S_0(\id_0-\Gamma_0R_0)^{-1}-\widetilde H_1\Theta_1R_1(\id_1-\Gamma_1R_1)^{-1}D_0
        \\
        =&(P'+\Theta_0R_0)(\id_0-\Gamma_0R_0)^{-1}+\{\text{maps with boundedness }\Xs^{-\infty}_0\to\Xs^{+\infty}_0\}
        \\
        =&P'(\id_0-\Gamma_0R_0)^{-1}+\{\text{maps with boundedness }\Xs^{-\infty}_0\to\Xs^{+\infty}_0\}.
    \end{align*}

    By the assumptions in \ref{Item::PM::AbsPM::BddP}, $R_0:\Ys_0\to\Ys_0$ and $P':\Ys_0\to\Zs_0$ are all bounded. By Lemma~\ref{Lem::PM::RieszThm}~\ref{Item::PM::RieszThm::MoreBdd} $(\id_0-\Gamma_0R_0)^{-1}:\Ys_0\to\Ys_0$ is bounded. Therefore the boundedness $P:\Ys_0\to\Zs_0$ follows from the standard composition argument.
\end{proof}

\begin{rem}
    If one only wants $f=Pf+DHf+HDf$ for $f\in\Xs^{-M}$ where $M\ge1$ is a large but finite number, we do not need to use the spectral projections $\Gamma_q$. Suppose there are homotopy operators $\widetilde H$ such that $\widetilde H:\Xs^{2M}\to \Xs^M$ for large $M$. In this case we can take
    \begin{equation*}
        \textstyle
        H_q:=H'_q\cdot\sum_{j=0}^{3M-1}R_q^j+\widetilde H_q\cdot R_q^{3M},\qquad 1\le q\le n.
    \end{equation*}
    One can check that $H_q$ are defined on $\Xs^{-M}_q$ and satisfy $f_q=D_{q-1}H_qf_q+H_{q+1}D_qf_q$ for all $f_q\in\Xs^{-M}_q$ such that $D_qf_q\in\Xs^{-M}_{q+1}$.

    If $H',R$ and $\widetilde H$ have explicit formula, then so does $H$ in this construction, while $\Gamma$ in general can be rather abstract, as it is difficult to characterize the eigenspace of a generic integral operator.
\end{rem}

In application to the homotopy operators we apply Proposition~\ref{Prop::PM::AbsPM} with partition of unity. To incorporate Theorem~\ref{Thm::GenThm} later we use the Triebel-Lizorkin spaces.

\begin{prop}[Parametrix formulae via partition of unity]\label{Prop::PM::ConcretePM}
    Let $\Omega\subset\C^n$ be a bounded Lipschitz domain that is pseudoconvex. For every $\zeta\in b\Omega$ we have the following objects:
    \begin{itemize}
        \item A holomorphic map $\Phi_\zeta:U_\zeta\hookrightarrow\C^n$ which is biholomorphic onto its image such that $\Phi_\zeta(U_\zeta)\ni\zeta$.
        \item An open set $\omega_\zeta\subset\C^n$ such that $\omega_\zeta\cap U_\zeta=\Phi_\zeta^{-1}(\Omega)$.
        \item Operators $\Pc^\zeta:\Ss'(\omega_\zeta)\to\Ss'(U_\zeta\cap\omega_\zeta)$, $\Hc_q^\zeta:\Ss'(\omega_\zeta;\wedge^{0,q})\to\Ss'(U_\zeta\cap\omega_\zeta;\wedge^{0,q-1})$ for $1\le q\le n$, $\Rc_q^\zeta:\Ss'(\omega_\zeta;\wedge^{0,q})\to\Ss'(U_\zeta\cap\omega_\zeta;\wedge^{0,q})$ for $0\le q\le n$.
    \end{itemize}

    Suppose we have the following:
    \begin{enumerate}[(a)]
        \item{\normalfont (Local homotopy)}\label{Item::PM::ConcretePM::LocalHT} For every $0\le q\le n$ and $f\in\Ss'(\omega_\zeta;\wedge^{0,q})$, (still we set $\Hc_{n+1}^\zeta=0$.)
        \begin{equation}\label{Eqn::PM::ConcretePM::LocalHTAssum}
    \begin{aligned}
        f|_{U_\zeta\cap\omega_\zeta}&=\Pc^\zeta f+\Hc_1^\zeta\dbar f+\Rc_0^\zeta f,&\dbar\Pc^\zeta f&=0,&&&& \text{when }q=0;
        \\
        f|_{U_\zeta\cap\omega_\zeta}&=\dbar\Hc_q^\zeta f+\Hc_{q+1}^\zeta\dbar f+\Rc_q^\zeta f,&&&&&&\text{when }1\le q\le n.
    \end{aligned}
\end{equation}

        \item\label{Item::PM::ConcretePM::SobBdd} There is a $0<t_0<1$ such that $\Hc_q^\zeta:H^{s,2}\to H^{s+t_0,2}$ and $\Rc_q^\zeta:H^{s,2}\to H^{s+t_0,2}$ are bounded for all $s\in\R$ and $0\le q\le n$ (we use $\Hc_0^\zeta=0$).
        \item\label{Item::PM::ConcretePM::DPBdd}$[f\mapsto \chi_1\cdot \Pc^\zeta(\chi_2f)]:\Ss'(\omega_\zeta)\to\Co^\infty(\omega_\zeta)$ is continuous when $\chi_1,\chi_2\in\Co_c^\infty(U_\zeta)$ have disjoint supports.
        
    \end{enumerate}

Then there are $\Pc^\Omega:\Ss'(\Omega)\to\Ss'(\Omega)$ and $\Hc_q^\Omega:\Ss'(\Omega;\wedge^{0,q})\to \Ss'(\Omega;\wedge^{0,q-1})$ for $1\le q\le n$ such that 
    \begin{enumerate}[(i)]
        \item\label{Item::PM::ConcretePM::HT} We have homotopy formula
        \begin{align}\label{Eqn::PM::ConcretePM::HT0}
        f&=\Pc^\Omega f+\Hc_1^\Omega\dbar f,&\text{for all }f\in\Ss'(\Omega),&&q=0;
        \\\label{Eqn::PM::ConcretePM::HTq}
        f&=\dbar\Hc_q^\Omega f+\Hc_{q+1}^\Omega\dbar f,&\text{for all }f\in\Ss'(\Omega;\wedge^{0,q}),&&1\le q\le n.
    \end{align}
        \item\label{Item::PM::ConcretePM::GlobalBdd} Fix a $0\le q\le n$. Let $s_0,s_1\in\R$, $p_0,p_1,r_0,r_1\in(0,\infty]$ be such that $s_0+1-s_1\ge2n\max(0,\frac1{p_0}-\frac1{p_1})$, with $r_0\le r_1$ when $s_1=s_0+1$. Suppose for every $\zeta\in b\Omega$ (here we use $\Hc_0^\zeta=\Hc_{n+1}^\zeta=0$ as usual):
        \begin{enumerate}[label=\normalfont{(\thesection.\arabic*)}]\setcounter{enumii}{\value{equation}}
            \item\label{Item::PM::ConcretePM::GlobalBdd::R} $\Rc_q^\zeta:\Fs_{p_0r_0}^{s_0}(\omega_\zeta;\wedge^{0,q})\to \Fs_{p_0r_0}^{s_0}(U_\zeta\cap\omega_\zeta;\wedge^{0,q})$ is bounded;
            \item \label{Item::PM::ConcretePM::GlobalBdd::H'Cpt} $\Hc_{q+\epsilon}^\zeta:\Fs_{p_0r_0}^{s_0}(\omega_\zeta;\wedge^{0,q+\epsilon})\to \Fs_{p_0r_0}^{s_0}(U_\zeta\cap\omega_\zeta;\wedge^{0,q+\epsilon-1})$ are bounded for $\epsilon\in\{0,1\}$. 
            \item\label{Item::PM::ConcretePM::GlobalBdd::Bdd0}  If $q=0$: $\Rc_0^\zeta,\Pc^\zeta:\Fs_{p_0r_0}^{s_0}(\omega_\zeta)\to \Fs_{p_1r_1}^{s_1}(U_\zeta\cap\omega_\zeta)$ and $\Hc_1^\zeta:\Fs_{p_0r_0}^{s_0}(\omega_\zeta;\wedge^{0,1})\to \Fs_{p_1r_1}^{s_1}(U_\zeta\cap\omega_\zeta)$ are all bounded.
            \item\label{Item::PM::ConcretePM::GlobalBdd::Bddq} If $1\le q\le n$: $\Hc_q^\zeta:\Fs_{p_0r_0}^{s_0}(\omega_\zeta;\wedge^{0,q})\to \Fs_{p_1r_1}^{s_1}(U_\zeta\cap\omega_\zeta;\wedge^{0,q-1})$ is bounded.\setcounter{equation}{\value{enumii}}
        \end{enumerate}

        Then automatically we have the boundedness $\Pc^\Omega:\Fs_{p_0r_0}^{s_0}(\Omega)\to \Fs_{p_1r_1}^{s_1}(\Omega)$ if $q=0$, and $\Hc^\Omega_q:\Fs_{p_0r_0}^{s_0}(\Omega;\wedge^{0,q})\to \Fs_{p_1r_1}^{s_1}(\Omega;\wedge^{0,q-1})$ if $1\le q\le n$.

    \end{enumerate}
\end{prop}

\begin{proof} 

    By compactness of $b\Omega$ we can take finitely many $\zeta_1,\dots,\zeta_M$ such that $b\Omega\subset\bigcup_{\nu=1}^M\Phi_{\zeta_\nu}(U_{\zeta_\nu})$. Let $\chi_0\in C_c^\infty(\Omega)$ and $\chi_\nu\in C_c^\infty(\Phi_{\zeta_\nu}(U_{\zeta_\nu}))$ ($1\le \nu\le M$) be such that 
    $\sum_{\nu=0}^M\chi_\nu(z)\equiv1$ for all $z\in\Omega$. We take $\lambda_0\in C_c^\infty(\Omega)$ and $\lambda_\nu\in C_c^\infty(\Phi_{\zeta_\nu}(U_{\zeta_\nu}))$ for $1\le \nu\le M$, such that for each $0\le\nu\le M$, $\lambda_\nu\equiv1$ in an open neighborhood of $\supp\chi_\nu$. Therefore $\chi_\nu=\lambda_\nu\chi_\nu$ and $\supp\chi_\nu\cap\supp(\dbar\lambda_\nu)=\varnothing$ for $0\le \nu\le M$.

    We need to apply Proposition~\ref{Prop::PM::AbsPM}. Set 
    $$\Xs^k_q:=H^{kt_0,2}(\Omega;\wedge^{0,q}),\qquad 0\le q\le n,\qquad k\in\Z.$$ So $\Xs^k_q\hookrightarrow\Xs^{k-1}_q$ are compact embeddings with dense images for all $k$. By Lemma~\ref{Lem::Space::SsLim}, we have $\Xs^{-\infty}_q=\Ss'(\Omega;\wedge^{0,q})$ and $\Xs^{+\infty}_q=\Co^\infty(\Omega;\wedge^{0,q})$. Set $D_q:=\dbar$ on $(0,q)$-forms. Immediately $D_{q-1}:\Xs_{q-1}^{-\infty}\to\Xs_q^{-\infty}$ is continuous and \ref{Item::PM::AbsPM::D} holds. Since $\Omega$ is bounded pseudoconvex, by the standard H\"ormander's estimate \cite{Hormander1965}, for every $q\ge1$ and  every $\dbar$-closed $f\in\Co^\infty(\Omega;\wedge^{0,q})\subset L^2(\Omega;\wedge^{0,q})$ is of the form $\dbar u=f$ for some $u\in L^2(\Omega;\wedge^{0,q-1})\subset\Ss'(\Omega;\wedge^{0,q-1})$, which gives \ref{Item::PM::AbsPM::0Cohomo}.
    
    We define operators $P'$, $H'_q$ and $R_q$ on $\Omega$ (in the notations for Proposition~\ref{Prop::PM::AbsPM}) by the following: 
    \begin{align}
    \label{Eqn::PM::ConcretePM::P'}
        P'f:=&\sum_{\nu=1}^M\Phi_{\zeta_\nu*}\big((\lambda_\nu\circ \Phi_{\zeta_\nu})\cdot\Pc^{\zeta_\nu}[\Phi_{\zeta_\nu}^*(\chi_\nu f)]\big)=\sum_{\nu=1}^M\lambda_\nu\cdot\Phi_{\zeta_\nu*}\circ\Pc^{\zeta_\nu}\circ\Phi_{\zeta_\nu}^*[\chi_\nu f];
        \\
    \label{Eqn::PM::ConcretePM::H'}    H'_qf:=&\lambda_0\cdot\Bc_q[\chi_0f]+\sum_{\nu=1}^M\lambda_\nu\cdot\Phi_{\zeta_\nu*}\circ\Hc_q^{\zeta_\nu}\circ\Phi_{\zeta_\nu}^*[\chi_\nu f].
    \end{align}
    Here $\Bc_q=\Bc_q^\Omega$ are the Bochner-Martinelli integral operators from \eqref{Eqn::Space::DefBOp}. Recall (see Lemmas~\ref{Lem::Space::BMKernel} and \ref{Lem::Space::BMFormula}) the boundedness $\Bc_q:H^{s,2}_c(\Omega;\wedge^{0,q})\to H^{s+1,2}(\Omega;\wedge^{0,q-1})$ for all $s\in\R$ and the formula $g=\dbar\Bc_qg+\Bc_{q+1}\dbar g$ for $(0,q)$ forms $g$ with compact supports.

    By definition along with standard compositions we see that $H_q':H^{s,2}\to H^{s+t_0,2}$ and hence $P':H^{s,2}\to H^{s+t_0-1,2}$ are bounded for all $s$. In particular $H'_q:\Xs_q^{-\infty}\to \Xs_{q-1}^{-\infty}$ and $P':\Xs_0^{-\infty}\to \Xs_0^{-\infty}$ are continuous.

    We define $(R_q)_{q=0}^n$ by $R_0f:=f-P'f-H'_1\dbar f$ and $R_qf:=f-\dbar H_q'f-H_{q+1}'\dbar f$ for $1\le q\le n$. Therefore \ref{Item::PM::AbsPM::ParaHT} holds since we define $R_q$ by this way. 


    Note that by \eqref{Eqn::PM::ConcretePM::P'} and \eqref{Eqn::PM::ConcretePM::LocalHTAssum} we have
    \begin{equation*}
        \dbar P'f=\sum_{\nu=1}^M\dbar\lambda_\nu\wedge\Phi_{\zeta_\nu*}\Pc^{\zeta_\nu}[\Phi_{\zeta_\nu}^*(\chi_\nu f)]=\sum_{\nu=1}^M\Phi_{\zeta_\nu*}\circ\big((\Phi_{\zeta_\nu}^*\dbar\lambda_\nu)\wedge\Pc^{\zeta_\nu}\circ[(\Phi_{\zeta_\nu}^*\chi_\nu)\cdot(\Phi_{\zeta_\nu}^*f)]\big).
    \end{equation*}
    Since $\dbar\lambda_\nu$ and $\chi_\nu$ have disjoint supports, so do $\Phi_{\zeta_\nu}^*\dbar\lambda_\nu$ and $\Phi_{\zeta_\nu}^*\chi_\nu$. By the assumption \ref{Item::PM::ConcretePM::DPBdd} we get $\dbar P':\Ss'\to\Co^\infty$. This completes the verification of \ref{Item::PM::AbsPM::DP}.
    
    To see $R_q:H^{s,2}\to H^{s+t_0,2}$ is bounded, i.e. to verify \ref{Item::PM::AbsPM::RCpt} we compute the explicit expression.
    \begin{align*}
    \notag
        Rf=&f-P'f-\dbar H'f-H'\dbar f
        \\
        =&\lambda_0\chi_0f-\dbar(\lambda_0\Bc[\chi_0f])-\lambda_0\Bc[\chi_0\dbar f]
        \\
        &+\sum_{\nu=1}^M\Big\{\lambda_\nu\chi_\nu f-\lambda_\nu\Phi_{\zeta_\nu*}\Pc^{\zeta_\nu}\Phi_{\zeta_\nu}^*[\chi_\nu f]-\dbar(\lambda_\nu\Phi_{\zeta_\nu*}\Hc^{\zeta_\nu}\Phi_{\zeta_\nu}^*[\chi_\nu f])-\lambda_\nu\Phi_{\zeta_\nu*}\Hc^{\zeta_\nu}\Phi_{\zeta_\nu}^*[\chi_\nu\dbar f]\Big\}
        \\
        =&\lambda_0\cdot(\id-\dbar\Bc-\Bc\dbar)[\chi_0f]-\dbar\lambda_0\wedge\Bc[\chi_0f]+\lambda_0\cdot\Bc[\dbar\chi_0\wedge f]+\sum_{\nu=1}^M\lambda_\nu\cdot\Phi_{\zeta_\nu*}\Hc^{\zeta_\nu}\Phi_{\zeta_\nu}^*[\dbar\chi_\nu\wedge f]
        \\
        &+\sum_{\nu=1}^M\Big\{\lambda_\nu\cdot\Phi_{\zeta_\nu*}\big\{\id-\Pc^{\zeta_\nu}-\dbar \Hc^{\zeta_\nu}-\Hc^{\zeta_\nu}\dbar\big\}\Phi_{\zeta_\nu}^*[\chi_\nu f]-\dbar\lambda_\nu\wedge\Phi_{\zeta_\nu*}\Hc^{\zeta_\nu}\Phi_{\zeta_\nu}^*[\chi_\nu f]\Big\}
        \\
        =&\lambda_0\cdot\Bc[\dbar\chi_0\wedge f]-\dbar\lambda_0\wedge\Bc[\chi_0f]+\sum_{\nu=1}^M\lambda_\nu\cdot\Phi_{\zeta_\nu*}\circ\Rc^{\zeta_\nu}\circ\Phi_{\zeta_\nu}^*[\chi_\nu f]
        \\
        &+\sum_{\nu=1}^M\Big\{-\dbar\lambda_\nu\wedge\Phi_{\zeta_\nu*}\circ\Hc^{\zeta_\nu}\circ\Phi_{\zeta_\nu}^*[\chi_\nu f]+\lambda_\nu\cdot\Phi_{\zeta_\nu*}\circ\Hc^{\zeta_\nu}\circ\Phi_{\zeta_\nu}^*[\dbar\chi_\nu\wedge f]\Big\}.
    \end{align*}
    
    By separating the degrees we have for $0\le q\le n$, (here we use $\Hc^\zeta_0=\Hc^\zeta_{n+1}=0$)
    \begin{multline}\label{Eqn::PM::ConcretePM::Rq}
        R_qf=\lambda_0\cdot\Bc_{q+1}[\dbar\chi_0\wedge f]-\dbar\lambda_0\wedge\Bc_q[\chi_0f]
        \\
        +\sum_{\nu=1}^M\Big\{\lambda_\nu\cdot\Phi_{\zeta_\nu*}\circ\Rc^{\zeta_\nu}_q\circ\Phi_{\zeta_\nu}^*[\chi_\nu f]-\dbar\lambda_\nu\wedge\Phi_{\zeta_\nu*}\circ\Hc^{\zeta_\nu}_q\circ\Phi_{\zeta_\nu}^*[\chi_\nu f]+\lambda_\nu\cdot\Phi_{\zeta_\nu*}\circ\Hc^{\zeta_\nu}_{q+1}\circ\Phi_{\zeta_\nu}^*[\dbar\chi_\nu\wedge f]\Big\}.
    \end{multline}
    Recall from Lemma~\ref{Lem::Space::BMKernel} $f\mapsto \lambda_0\cdot\Bc_{q+1}[\dbar\chi_0\wedge f]+\dbar\lambda_0\wedge\Bc_q[\chi_0f]$ is bounded $H^{s,2}(\Omega)\to H^{s+1,2}(\Omega)$, which is in particular as a map $H^{s,2}\to H^{s+t_0,2}$.

    Note that for $\psi\in\{\chi_\nu,\dbar\chi_\nu,\lambda_\nu,\dbar\lambda_\nu\}$, by Lemma~\ref{Lem::Space::ExtDiffeo} we have, for every $p,r,s$,
    \begin{equation}\label{Eqn::PM::ConcretePM::CompBdd}
        [f\mapsto\Phi_{\zeta_\nu}^*(\psi\wedge f)]:\Fs_{pr}^s(\Omega;\wedge^{0,\bullet})\to\Fs_{pr}^s(\omega_{\zeta_\nu};\wedge^{0,\bullet}),\ [g\mapsto\psi\wedge\Phi_{\zeta_\nu*}g]:\Fs_{pr}^s(U_{\zeta_\nu}\cap\omega_{\zeta_\nu};\wedge^{0,\bullet})\to\Fs_{pr}^s(\Omega;\wedge^{0,\bullet}).
    \end{equation}
    Recall from Lemma~\ref{Lem::Space::SpaceLem} we have $H^{s,2}=\Fs_{22}^s$. Therefore from \eqref{Eqn::PM::ConcretePM::Rq} and assumption $\Hc^\zeta:H^{s,2}\to H^{s+t_0,2}$, we conclude that $R:H^{s,2}\to H^{s+t_0,2}$ is bounded, which verifies \ref{Item::PM::AbsPM::RCpt}. 
    
    By Proposition~\ref{Prop::PM::AbsPM}~\ref{Item::PM::AbsPM::HT} we get \eqref{Eqn::PM::ConcretePM::HT0} and \eqref{Eqn::PM::ConcretePM::HTq} with $\Pc^\Omega=P$ and $\Hc^\Omega_q=H_q$ via \eqref{Eqn::PM::AbsPM::PH}, giving the proof of \ref{Item::PM::ConcretePM::HT}.

    \medskip
    To prove \ref{Item::PM::ConcretePM::GlobalBdd}, we apply Proposition~\ref{Prop::PM::AbsPM}~\ref{Item::PM::AbsPM::BddP} and \ref{Item::PM::AbsPM::BddH} with $\Ys_q=\Fs_{p_0r_0}^{s_0}(\Omega;\wedge^{0,q})$ and $\Zs_q=\Fs_{p_1r_1}^{s_1}(\Omega;\wedge^{0,q})$. 
    
    By Lemma~\ref{Lem::Space::BMKernel}, for $\phi,\psi\in\{\chi_0,\dbar\chi_0,\lambda_0,\dbar\lambda_0\}$, the map $f\mapsto\phi\wedge\Bc[\psi\wedge f]$ is bounded $\Fs_{p_0r_0}^{s_0}(\Omega;\wedge^{0,\bullet})\to\Fs_{p_0r_0}^{s_0+1}(\Omega;\wedge^{0,\bullet})$, which in particular is bounded as the map $\Fs_{p_0r_0}^{s_0}\to \Fs_{p_0r_0}^{s_0}$. By Lemma~\ref{Lem::Space::TLEmbed} and the assumption $s_0+1-s_1\ge2n\max(0,\frac1{p_0}-\frac1{p_1})$, with $r_0\le r_1$ when $s_1=s_0+1$, we have embedding $\Fs_{p_0r_0}^{s_0+1}(\Omega)\hookrightarrow\Fs_{p_1r_1}^{s_1}(\Omega)$. Therefore
    \begin{equation}\label{Eqn::PM::ConcretePM::BMBdd}
        [f\mapsto\phi\wedge\Bc(\psi\wedge f)]:\Fs_{p_0r_0}^{s_0}(\Omega;\wedge^{0,\bullet})\to \Fs_{p_1r_1}^{s_1}(\Omega;\wedge^{0,\bullet})\text{ is bounded}.
    \end{equation}
    
    By assumption \ref{Item::PM::ConcretePM::GlobalBdd::H'Cpt} $\Hc^\zeta_q,\Hc^\zeta_{q+1}:\Fs_{p_0r_0}^{s_0}\to\Fs_{p_0r_0}^{s_0}$ are bounded. Applying \eqref{Eqn::PM::ConcretePM::Rq} with \ref{Item::PM::ConcretePM::GlobalBdd::R} and \eqref{Eqn::PM::ConcretePM::CompBdd} we conclude that $R_q:\Fs_{p_0r_0}^{s_0}\to\Fs_{p_0r_0}^{s_0}$ (i.e. $R_q:\Ys_q\to \Ys_q$) is bounded.

    When $q=0$, we apply the assumption \ref{Item::PM::ConcretePM::GlobalBdd::Bdd0} that $\Pc^\zeta,\Rc_0^\zeta,\Hc_1^\zeta:\Fs_{p_0r_0}^{s_0}\to\Fs_{p_1r_1}^{s_1}$ to \eqref{Eqn::PM::ConcretePM::P'} and \eqref{Eqn::PM::ConcretePM::Rq}. By \eqref{Eqn::PM::ConcretePM::CompBdd}, we conclude that $P',R_0:\Fs_{p_0r_0}^{s_0}\to\Fs_{p_1r_1}^{s_1}$ are both bounded, i.e. $P',R_0:\Ys_0\to\Zs_0$. By Proposition~\ref{Prop::PM::AbsPM}~\ref{Item::PM::AbsPM::BddP}, we get the boundedness $\Pc^\Omega:\Fs_{p_0r_0}^{s_0}\to\Fs_{p_1r_1}^{s_1}$.

    When $1\le q\le n$, we apply the assumption \ref{Item::PM::ConcretePM::GlobalBdd::Bddq} that $\Hc_q^\zeta:\Fs_{p_0r_0}^{s_0}\to\Fs_{p_1r_1}^{s_1}$ to \eqref{Eqn::PM::ConcretePM::H'}. By \eqref{Eqn::PM::ConcretePM::CompBdd} and \eqref{Eqn::PM::ConcretePM::BMBdd}, we conclude that $H'_q:\Fs_{p_0r_0}^{s_0}\to\Fs_{p_1r_1}^{s_1}$ is bounded. In other words $H'_q:\Ys_q\to\Zs_{q-1}$. By Proposition~\ref{Prop::PM::AbsPM}~\ref{Item::PM::AbsPM::BddH}, we get the boundedness $\Hc^\Omega_q:\Fs_{p_0r_0}^{s_0}\to\Fs_{p_1r_1}^{s_1}$.  This proves \ref{Item::PM::ConcretePM::GlobalBdd}.\qedhere

\end{proof}
\section{A New Decomposition for Reverse Integration by Parts}\label{Section::AntiDev}


Now $\omega=\{(z',x_n+iy_n):y_n>\sigma(z',x_n)\}$ where $\sigma$ is $C^{1,1}$. We see that $\rho(z):=\sigma(z',x_n)-y_n$ is a $C^{1,1}$ defining function. In \cite{ShiYaoCk} we have to take derivatives of $\rho$ in every directions, which creates difficulties to on defining $\Kc_q\circ[\dbar,\Ec]f$ when $f$ is a distribution with large index. Our goal in the section is to restrict these differentiations to the $\frac\partial{\partial y_n}$ direction.

To be more precise, we are going to give the following refinement of the anti-derivative operators from \cite{ShiYaoExt}.

\begin{thm}\label{Thm::RefAtD}
    Let $v=(v_k)_{k=1}^N$ be a collection of $N$ linearly independent vectors in $\R^N$. We denote the associated positive cone $V_v:=\{a_1v_1+\dots+a_Nv_N:a_1,\dots,a_N>0\}$.
    
    Then there are $\{\mu^{m,k}:(m,k)=(0,0)\text{ or }m\ge0,\ 1\le k\le N \}\subset\Ss'(\R^N)$ such that the convolution operators $\Sc^{m,k}f:=\mu^{m,k}\ast f$ are defined on $\Ss'(\R^N)$ and the following hold:
    \begin{enumerate}[(i)]
        \item\label{Item::RefAtD::Sum} For every $f\in\Ss'(\R^N)$, $f=\sum_{k=0}^N\Sc^{0,k}f$ and $\Sc^{0,k}f=D_{v_k}^m\Sc^{m,k}f$ hold for all $1\le k\le N$ and $m\ge1$. 
        \item\label{Item::RefAtD::Supp} $\supp\mu^{m,k}\subseteq\overline{V_v}$ for every $m,k$. In particular $\supp(\Sc^{m,k}f)\subseteq\overline{V_v+\supp f}$.
        \item\label{Item::RefAtD::Bdd} For $m\ge0$ we have $\Sc^{0,0},\Sc^{m,k}:\Fs_{pr}^s(\R^N)\to \Fs_{pr}^{s+m}(\R^N)$ for all $s\in\R$ and $p,r\in(0,\infty]$.

    \end{enumerate}

In particular $\Sc^{0,0},\Sc^{m,k}:H^{s,p}(\R^N)\to H^{s+m,p}(\R^N)$ and $\Sc^{0,0},\Sc^{m,k}:\Co^s(\R^N)\to\Co^{s+m}(\R^N)$ for every $s\in\R$ and $1<p<\infty$. 
\end{thm}
\begin{rem}
    We are going to construct $\mu^{m,k}$ by taking Schwartz functions $\phi_0,\psi_0,\tilde\psi^{m,k}_j,\phi_j$ for $j\ge1$, all supported in $V_v$, such that for all $m\ge0$ and $1\le k\le N$,
    \begin{enumerate}[(a)]
        \item $\tilde\psi^{m,k}_j(x)=2^{(j-1)N}\tilde\psi^{m,k}_1(2^{j-1}x)$, $\phi_j(x)=2^{(j-1)N}\phi_1(2^{j-1}x)$ for $j\ge2$.
        \item $\int x^\alpha\tilde\psi^{m,k}_j(x)dx=\int x^\alpha\phi_j(x)dx=0$ for all multiindex $\alpha\in\N^N$.
        \item $\mu^{0,0}=\psi_0\ast\phi_0$ and $\mu^{m,k}=\sum_{j=1}^\infty 2^{-jm}\tilde\psi^{m,k}_j\ast\phi_j$ for all $m\ge0$ and $1\le k\le N$.
    \end{enumerate}
\end{rem}



Let us review some ingredients used in \cite[Section~4]{ShiYaoExt}. We denote by $\dot\Ss(\R^N)$ the space of all Schwartz functions $f\in\Ss(\R^N)$ such that $\int x^\alpha f(x)dx=0$ for all multi-index $\alpha$, equivalently for all Schwartz $f$ such that its Fourier transform $\hat f(\xi):=\int f(x)e^{-2\pi ix\xi}dx$ satisfies $\lim_{\xi\to0}|\xi|^{-m}\hat f(\xi)=0$ for all $m$. Its dual space $\dot\Ss'(\R^N)$ is the quotient space $\dot\Ss'(\R^N)=\Ss'(\R^N)/\{\text{polynomial}\}$ (see \cite[Section~5.1.2]{TriebelTheoryOfFunctionSpacesI}).

For $\sigma\in\dot\Ss'(\R^N)$ and $f\in\dot\Ss(\R^N)$ we can talk about convolution $\sigma\ast f\in\dot\Ss'(\R^N)$. Indeed, if we set $f^\kappa(x):=f(-x)$ then for every $g\in\dot\Ss(\R^N)$ we have $f^\kappa\ast g\in\dot\Ss(\R^N)$ which still has all moment vanishing, hence $\langle \sigma\ast f,g\rangle_{\dot\Ss',\dot\Ss}=\langle\sigma,f^\kappa\ast g\rangle_{\dot\Ss',\dot\Ss}$ is defined.

 In our case we need to following condition so that $\sigma\ast f\in\dot\Ss$:
\begin{lem}[{\cite[Lemma~4.2]{ShiYaoExt}}]\label{Lem::AntiDev::S0Conv}
    Let $\eta_0\in\dot\Ss(\R^N)$. For $j\in\Z$ we denote $\eta_j(x):=2^{jN}\eta_0(2^jx)$. 
    
    Then for each $r\in\R$ the summation $\sigma^r=\sum_{k\in\Z}2^{kr}\eta_k$ converges in $\dot\Ss'(\R^N)$.
    
    Moreover if $f\in\dot\Ss(\R^N)$, then $\sigma^r\ast f=\sum_{k\in\Z}2^{kr}\eta_k\ast f$ converges in $\dot\Ss(\R^N)$.
\end{lem}

In 1 dimensional case we have,
\begin{lem}[{\cite[Lemma~4.3]{ShiYaoExt}}]\label{Lem::AntiDev::AFourierLemma}
    For $u\in\dot\Ss(\R)$, the function $( 2\pi i\tau)^{-1}\widehat u(\tau)$ is the Fourier transform of the $\dot\Ss$-function $v(t):=\int_{\R_+}u(t-s)ds$. Moreover if $\supp u\subset(0,\infty)$ then $\supp v\subset(0,\infty)$.

    In other words it makes sense to say $\big(\frac d{dt}\big)^{-1}u(t):=\int_{\R_+} u(t-s)ds$.
\end{lem}
Repeating the lemma $m$ times we see that $\big(\frac d{dt}\big)^{-m}u\in\dot\Ss(\R)$ is supported in $(0,\infty)$ as well.

The key to prove Theorem~\ref{Thm::RefAtD} is to modify \cite[Proposition~4.4]{ShiYaoExt} slightly.
\begin{prop}\label{Prop::AntiDev::RefSigma}
    There are $\sigma^m_k\in\dot\Ss'(\R^N)$ for $m\ge0$ and $1\le k\le N$ such that the following hold.
    \begin{enumerate}[(i)]
        \item\label{Item::AntiDev::RefSigma::Schwartz} For every $f\in\dot\Ss(\R^N)$, we have $\sigma^m_k\ast f\in\dot\Ss(\R^N)$ for all  $m\ge0$ and $1\le k\le N$.
        \item\label{Item::AntiDev::RefSigma::Supp} If $\supp f\subset(0,\infty)^N$ then $\supp(\sigma^m_k\ast f)\subset(0,\infty)^N$ as well.
        \item\label{Item::AntiDev::RefSigma::Sum} $f=\sum_{k=1}^N\sigma^0_k\ast f$ and $\frac {\partial^m}{\partial x_k^m}\sigma^m_k\ast f(x)=\sigma^0_k\ast f(x)$ for all $m\ge1$ and $1\le k\le N$.
    \end{enumerate}
\end{prop}
\begin{proof}
    Let $g(t)$ be the function as in the proof of \cite[Theorem~4.1]{RychkovExtension}. That is, we have a $g \in \Ss (\R)$ satisfying:
	$$\supp g \subseteq [1, \infty);\quad \int_{\R} g = 1,\text{ and }\int_{\R} t^k g(t)dt = 0\quad\text{for all }k \in \Z_+.$$
	
    Therefore
    $$u_j(t):=2^{j} g (2^jt) - 2^{j-1} g (2^{j-1}t),\quad j\in\Z,$$ 
    are Schwartz functions supported in $\R_+$ and have integral $0$. By the moment condition of $g$ (we have $\int t^kg(t)dt=0$, $k\ge1$) we have $\int t^ku_j(t)dt=0$ for all $k\ge1$ as well. Therefore $u_j\in\dot\Ss(\R)$. By Lemma~\ref{Lem::AntiDev::AFourierLemma}, $\big( \frac{d}{dt} \big)^{-m} u_j\in\dot\Ss(\R)$ is a well-defined moment vanishing Schwartz function supported in $\R_+$.

    For $1\le k\le N$, $m\ge0$ and $j\in\Z$, we define $\vartheta^m_{k,j}\in\Ss(\R^N)$ as
    \begin{equation}\label{Eqn::AntiDev::RefSigma::DefMu}
	\vartheta^m_{k,j}(x)
	:=\Big(\prod_{l=1}^{k-1}2^jg(2^jx_l)\Big)\cdot (\partial_t^{-m}u_j)(x_k)\cdot
	\Big(\prod_{l=k+1}^N2^{j-1}g(2^{j-1}x_l)\Big). 
	\end{equation}
    Taking Fourier transform of \eqref{Eqn::AntiDev::RefSigma::DefMu} we have
    \begin{align}\label{Eqn::AntiDev::ProofCase=1::MuFourier1}
	\widehat \vartheta^m_{k,j} ( \xi)
	&=\prod_{l=1}^{k-1} \widehat{g} (2^{-j} \xi_l)\cdot   
	\frac{\widehat{g} (2^{-j}\xi_k)-\widehat{g}(2^{1-j}\xi_k)}{(2\pi i\xi_k)^m}\cdot \prod_{l=k+1}^N \widehat{g}(2^{ 1-j}\xi_l)= 2^{-jm}\widehat\vartheta^m_{k,0}(2^{-j}\xi).
    \end{align}
    In particular $D_{x_k}^m\vartheta^m_{k,j}=\vartheta^0_{k,j}$ for all $m\ge1$, $1\le k\le N$ and $j\in\Z$.

    By Lemma~\ref{Lem::AntiDev::AFourierLemma} $\partial_t^{-m}u_j\in\dot\Ss(\R)$, therefore $\widehat\vartheta^m_{k,j}(\xi)=O(|\xi_k|^\infty)=O(|\xi|^\infty)$ which implies $\vartheta^m_{k,j}\in\dot\Ss(\R^N)$. Applying Lemma~\ref{Lem::AntiDev::AFourierLemma} with $\supp g\subset(0,\infty)$ and $\supp \partial_t^{-m}u_j\subset(0,\infty)$ we get $\supp\vartheta^m_{k,j}\subset(0,\infty)^N$.
    
    Now define
    \begin{equation*}
        \sigma^m_k(x):=\sum_{j\in\Z}\vartheta^m_{k,j}(x)=\sum_{j\in\Z}2^{j(N-m)}\vartheta^m_{k,0}(2^jx).
    \end{equation*}

    By Lemma~\ref{Lem::AntiDev::S0Conv}, $\sigma^m_k\ast f=\sum_{j\in\Z}\vartheta^m_{k,j}\ast f $ converges in $\dot\Ss(\R^N)$ for every $f\in\dot\Ss(\R^N)$, which gives \ref{Item::AntiDev::RefSigma::Schwartz}. If in addition $\supp f\subset(0,\infty)^N$ then $\supp(\vartheta^m_{k,j}\ast f)\subset(0,\infty)^N+(0,\infty)^N=(0,\infty)^N$, which gives \ref{Item::AntiDev::RefSigma::Supp}. 
    
    Finally for \ref{Item::AntiDev::RefSigma::Sum}, $D_{x_k}^m\sigma^m_k\ast f=\sigma^0_k\ast f$ follows from $D_{x_k}^m\vartheta^m_{k,j}=\vartheta^0_{k,j}$ as we have
    \begin{equation*}
        \frac{\partial^m}{\partial x_k^m}(\sigma^m_k\ast f)=\sum_{j\in\Z}\frac{\partial^m}{\partial x_k^m}\vartheta^m_{k,j}\ast f=\sum_{j\in\Z}\vartheta^0_{k,j}\ast f=\sigma^0_k\ast f.
    \end{equation*}

    For $\xi\in\R^N\backslash\{0\}$ we have
    \begin{align*}
        \sum_{k=1}^N\widehat\sigma^0_k(\xi)=&\sum_{j\in\Z}\sum_{k=1}^N\widehat\vartheta^0_{k,j}(\xi)=\sum_{j\in\Z}\widehat{g} (2^{-j} \xi_1) \cdots \widehat g(2^{-j}\xi_N)-\widehat{g} (2^{1-j} \xi_1)\cdots\widehat g(2^{1-j}\xi_N)
        \\
        =&\lim\limits_{J\to+\infty}\widehat g (2^{-J} \xi_1) \cdots \widehat g(2^{-J}\xi_N)-\widehat{g} (2^J \xi_1) \cdots \widehat g(2^J\xi_N)=1.
    \end{align*}

    Therefore $\hat f(\xi)=\sum_{k=1}^N(\sigma^0_k\ast f)^\wedge(\xi)$ pointwise for $f\in\dot\Ss(\R^N)$ and $\xi\in\R^N$. Taking Fourier inversion we get $f=\sum_{k=1}^N\sigma^0_k\ast f$ and complete the proof.
\end{proof}

We also recall some ingredients from Littlewood-Paley decomposition. 
\begin{prop}[{\cite[Proposition~3.10]{ShiYaoExt}}]\label{Prop::AntiDev::RychOp}
    Let $\eta_0,\theta_0\in\dot\Ss(\R^N)$. We set $\eta_j(x):=2^{jN}\eta_0(2^jx)$ and $\theta_j(x):=2^{jN}\theta_0(2^jx)$ for $j\ge1$. For $r\in\R$ we define the operator
    \begin{equation*}
        T^{\eta,\theta,r}f:=\sum_{j=1}^\infty2^{jr}\eta_j\ast\theta_j\ast f,\qquad f\in\Ss'(\R^N).
    \end{equation*}

    Then $T^{\eta,\theta,r}:\Ss'(\R^N)\to\Ss'(\R^N)$ is defined. And  $T^{\eta,\theta,r}:\Fs_{pq}^s(\R^N)\to \Fs_{pq}^{s-r}(\R^N)$ is bounded for all $s\in\R$ and $p,q\in(0,\infty]$.    
    In particular $T^{\eta,\theta,r}$ is bounded $H^{s,p}(\R^N)\to H^{s-r,p}(\R^N)$ and $\Co^s(\R^N)\to \Co^{s-r}(\R^N)$ for all $s\in\R$, $1<p<\infty$.
\end{prop}
Here \cite{ShiYaoExt} only treat the cases $(p,q)\notin\{\infty\}\times(0,\infty)$. The boundedness of the case $p=\infty>q$ follows from \cite[Proposition~17]{YaoExtensionMorrey}.


\begin{proof}[Proof of Theorem~\ref{Thm::RefAtD}]
    By passing to an (in fact unique) invertible real linear transform we can assume that $v_j=e_j$ for $1\le j\le N$, i.e. $D_{v_j}=\frac\partial{\partial x_j}$ and $V_v=(0,\infty)^N$.

    A result by Rychkov in \cite[Proposition~2.1 and Theorem~4.1(a)]{RychkovExtension} shows that there are families of Schwartz functions $(\phi_j,\psi_j)_{j=0}^\infty\subset\Ss(\R^N)$ such that (see also Remark~\ref{Rmk::Space::ExtOp}):
    \begin{itemize}
        \item $\phi_j,\psi_j\in\dot\Ss(\R^N)$ and $\phi_j(x)=2^{(j-1)N}\phi_1(2^{j-1}x)$, $\psi_j(x)=2^{(j-1)N}\psi_1(2^{j-1}x)$ for all $j\ge1$.
        
        \item For every $f\in\Ss'(\R^N)$, $f=\sum_{j=0}^\infty\psi_j\ast\phi_j\ast f$ with the sum converging in $\Ss'(\R^N)$.
        \item $\supp\phi_j,\supp\psi_j\subset (0,\infty)^N$ for all $j\ge0$.
    \end{itemize}

    Now for every $m\ge0$ and $1\le k\le N$, using Proposition~\ref{Prop::AntiDev::RefSigma} we define $\tilde\psi^{m,k}_j\in\dot\Ss(\R^N)$ for $j\ge1$ by
    \begin{equation*}
        \tilde\psi^{m,k}_j(x):=2^{jm}\sigma^m_k\ast\psi_j.
    \end{equation*}
    This gives the following: for every $m\ge1$, $1\le k\le N$ and $j\ge1$, 
    \begin{itemize}
        \item $\tilde\psi^{m,k}_j(x)=2^{(j-1)N}\tilde\psi^{m,k}_1(2^{j-1}x)$.
        \item $\psi_j=\sum_{k=1}^N\tilde\psi^{0,k}_j$ and $D_{x_k}^m\tilde\psi^{m,k}_j=2^{jm}\tilde\psi^{0,k}_j$.
        \item $\supp\tilde\psi^{m,k}_j\subset(0,\infty)^N$.
    \end{itemize}
    
    We define  $\Sc^{m,k}f=\mu^{m,k}\ast f$  by
    \begin{equation}
        \mu^{0,0}:=\psi_0\ast\phi_0,\qquad \mu^{m,k}:=\sum_{j=1}^\infty 2^{-jm}\tilde\psi^{m,k}_j\ast\phi_j,\quad m\ge0,\quad 1\le k\le N.
    \end{equation}

    Clearly $\mu^{0,0}\in\Ss(\R^N)$. By Proposition~\ref{Prop::AntiDev::RychOp}, $\mu^{m,k}=\Sc^{m,k}\delta_0\in\Ss'(\R^N)$ makes sense for all $m\ge0$ and $1\le k\le N$. 
    Immediately \ref{Item::RefAtD::Sum} follows from $\psi_j=\sum_{k=1}^N\tilde\psi^{0,k}_j$ and $D_{x_k}^m\tilde\psi^{m,k}_j=2^{jm}\tilde\psi^{0,k}_j$ for all $j\ge1$.

    Since $\tilde\psi^{m,k}_j$, $\psi_j$ and $\phi_j$ are all supported in $(0,\infty)^N$, taking summations we get \ref{Item::RefAtD::Supp}.

    By Proposition~\ref{Prop::AntiDev::RychOp}, $\Sc^{m,k}:\Fs_{pq}^s\to \Fs_{pq}^{s+m}$ for all $m\ge0$, $1\le k\le N$, $p,q\in(0,\infty]$ and $s\in\R$. The boundedness $\Sc^{0,0}:\Fs_{pq}^s\to \Fs_{pq}^{s+m}$ can follow from e.g. \cite[Proposition~3.14]{ShiYaoExt}. This proves \ref{Item::RefAtD::Bdd}.
\end{proof}

Using the notation \eqref{Eqn::Space::TildeSpace} for the space $\widetilde\Fs_{pq}^s$, we have
\begin{cor}\label{Cor::RefAtD::ScBdd}
    Let $\omega\subset\R^N$ be a special Lipschitz domain (see Definition~\ref{Defn::Space::SpecDom}) and let $v_1,\dots,v_N\in\R^N$ be such that $V_v\subset\{x_N<-|x'|\}$. Then the operators $\Sc^{m,k}f=\mu^{m,k}\ast f$ from Theorem~\ref{Thm::RefAtD} are bounded
    \begin{equation}\label{Eqn::RefAtD::ScBdd}
        \Sc^{0,0},\Sc^{m,k}:\widetilde\Fs_{pq}^s(\omega^c)\to\widetilde\Fs_{pq}^{s+m}(\omega^c),\quad \text{for all}\quad p,q\in(0,\infty],\quad s\in\R,\quad m\ge0,\quad 1\le k\le N.
    \end{equation}
\end{cor}
\begin{proof}
    The assumption $V_v\subset\{x_N<-|x'|\}$ implies $\omega^c+V_v\subseteq\omega^c$ (see Remark~\ref{Rmk::Space::SpecDom+Cone}). If $f\in\Ss'(\R^N)$ satisfies $f|_\omega=0$, then $\Sc^{m,k}f|_\omega=0$ holds as well. The result \eqref{Eqn::RefAtD::ScBdd} then follows from Theorem~\ref{Thm::RefAtD}~\ref{Item::RefAtD::Bdd}.
\end{proof}

\section{Weighted Estimates on Local Domains}\label{Section::WeiEst}


In this section we focus on local and redo the weighted estimate of the Leray-Koppelman kernel. 
\begin{assu}\label{Assumption::StrongDomain}
    In this section we consider the special domain $\omega\subset\C^n$ given by
    \begin{equation*}
        \omega=\omega_\sigma=\{(z',x_n+iy_n)\in\C^n:y_n>\sigma(z',x_n)\}.
    \end{equation*}
    Here $\sigma:\R^{2n-1}_{(z',x_n)}\to\R_{y_n}$ is a Lipschitz function such that $\sigma(0)=0$,  $\|\nabla\sigma\|_{L^\infty(\R^{2n-1})}=\sup_{w'\neq w''}\frac{|\sigma(w')-\sigma(w'')|}{|w'-w''|}<1$, and $\sigma|_{2\B^{2n-1}}$ is $C^{1,1}$. 
    
    Let $\rho(z)=\rho_\sigma(z):=\sigma(z',x_n)-y_n$ be the canonical defining function (which means $\rho(0)=0$). We assume that there is a $0<c_0<\frac14$ such that
    \begin{gather}\notag
    \|\nabla^2\sigma\|_{L^\infty(2\B^{2n-1})}=\sup_{\substack{u,v\in\R^{2n-1};\ |u|=|v|=1}}\ \essup_{\gamma\in2\B^{2n-1}}\bigg|\sum_{j,k=1}^{2n-1}D_{jk}^2\sigma(\gamma)\cdot u_jv_k\bigg|<\frac1{c_0};
    \\
    \label{Eqn::WeiEst::SuppFun}
        |\partial\rho(\zeta)\cdot(\zeta-z)|\ge c_0(\rho(\zeta)-\rho(z)+|z-\zeta|^2),\qquad z\in\B^{2n}\cap\omega,\quad\zeta\in\B^{2n}\backslash\overline\omega.
    \end{gather}

    Here \eqref{Eqn::WeiEst::SuppFun} is not necessarily true for general $(z,\zeta)\in\omega\times\overline\omega^c$.
\end{assu}

We set $\delta(w)=\delta_\omega(w):=\min(1,\dist(w,b\omega))$. Clearly there is a $C>0$ depending only on $\|\nabla\sigma\|_{L^\infty(2\B^{2n-1})}$ such that 
\begin{equation*}
    C^{-1}|\rho(w)|\le\delta(w)\le C|\rho(w)|,\qquad w\in\B^{2n}.
\end{equation*}

Let us recall the Cauchy-Fantappi\`e form and Leray-Koppelman form associated to the defining function $\rho$, and their corresponding integral operators:
\begin{align}
\label{Eqn::WeiEst::DefF}
    F(z,\zeta)&=F^\rho(z,\zeta):=\frac{\partial\rho(\zeta)\wedge(\dbar\partial\rho(\zeta))^{n-1}}{(2\pi i)^n(\partial\rho(\zeta)\cdot(\zeta-z))^n},\qquad z\in \B^{2n}\cap\omega,\quad\zeta\in\B^{2n}\backslash\overline\omega;
    \\\label{Eqn::WeiEst::DefFOp}
    \Fc g(z)&=\Fc^\rho g(z):=\int_{\B^{2n}\backslash\overline\omega} F^\rho(z,\cdot)\wedge g,\quad z\in\B^{2n}\cap\omega,\qquad g\in L^1(\B^{2n}\backslash\overline\omega;\wedge^{0,1});
    \\
\label{Eqn::WeiEst::DefK}
    K(z,\zeta)&=K^\rho(z,\zeta):=\frac{ b(z,\zeta)\wedge \partial\rho(\zeta)}{(2\pi i)^n}\wedge\sum_{k=1}^{n-1}\frac{(\dbar b(z,\zeta))^{n-1-k}\wedge(\dbar \partial\rho(\zeta))^{k-1}}{|z-\zeta|^{2(n-k)}(\partial\rho(\zeta)\cdot(\zeta-z))^k}=:\sum_{q=0}^{n-2}K_q(z,\zeta);
\\\label{Eqn::WeiEst::DefKOp}
    \Kc_q g(z)&=\Kc_q^\rho g(z):=\int_{\B^{2n}\backslash\overline\omega} K^\rho_{q-2}(z,\cdot)\wedge g,\quad z\in\B^{2n}\cap\omega,\qquad g\in L^1(\B^{2n}\backslash\overline\omega;\wedge^{0,q}),\quad 2\le q\le n.
\end{align}
Here $K(z,\zeta)$ and $F(z,\zeta)$ are both forms in the product spaces $(\B^{2n}\cap\omega)\times(\B^{2n}\backslash\overline\omega)$. $K$ has total degree $(n,n-2)$, and $K_q$ is the components that has degree $(0,q)$ in $z$ and degree $(n,n-2-q)$ in $\zeta$. $F$ has total degree $(n,n-1)$, which has degree $(0,0)$ in $z$ and $(n,n-1)$ in $\zeta$.
Also notice that $F$ is holomorphic in $z$. 
\begin{lem}\label{Lem::WeiEst::LPFormula}
    Following Assumption~\ref{Assumption::StrongDomain}, the bi-degree forms $B(z,\zeta)$, $F^\rho(z,\zeta)$ and $K^\rho(z,\zeta)$ given in \eqref{Eqn::Intro::DefB}, \eqref{Eqn::WeiEst::DefF} and \eqref{Eqn::WeiEst::DefK} satisfy $\dbar_{z,\zeta}K^\rho(z,\zeta)=B(z,\zeta)-F^\rho(z,\zeta)$ for $z\in\B^{2n}\cap\omega$ and $\zeta\in\B^{2n}\backslash\overline\omega$.

    As a result, for the associated operators $\Fc$ (see \eqref{Eqn::WeiEst::DefFOp}), $\Bc=\sum_{q=1}^n\Bc_q$ (see \eqref{Eqn::Space::DefBOp}) and $\Kc=\sum_{q=2}^n\Kc_q$ (see \eqref{Eqn::WeiEst::DefKOp}),
    we have $\dbar\Kc-\Kc\dbar=\Bc-\Fc$ on the domain where both sides are defined.
    
    More precisely
    $\Bc_1g-\Fc g=-\Kc_2\dbar g$ for all $g\in L^1(\B^{2n};\wedge^{0,1})$ with compact support in $\B^{2n}$ such that $g|_\omega=0$ and $\dbar g\in L^1$; and for $2\le q\le n$, $\Bc_qg=\dbar\Kc_qg-\Kc_{q+1}\dbar g$ for all $g\in L^1(\B^{2n};\wedge^{0,q})$ with compact support in $\B^{2n}$ such that $g|_\omega=0$ and $\dbar g\in L^1$ (still we set $\Kc_{n+1}=0$).
\end{lem}
\begin{proof}
    Notice that $F(z,\cdot),K(z,\cdot)$ are bounded functions in $\zeta$, so $\Kc_qg$ and $\Fc g$ are defined in $\Omega$ for $g\in L^1$.

    The form $\partial\rho(\zeta)=\sum_{j=1}^n\partial_j\rho(\zeta)d\zeta_j$ is holomorphic in $z$ (in fact constant in $z$) and is nonvanishing for $z\in\B^{2n}\cap\omega$ and $\zeta\in\B^{2n}\backslash\overline\omega$ by  \eqref{Eqn::WeiEst::SuppFun}. Using e.g. \cite[Lemma~11.1.1]{ChenShawBook} we have $\dbar_{z,\zeta}K=B-F$ in the sense of distributions.
    
    For $1\le q\le n-1$, $K_{q-1}$ is a $(2n-2)$-form and has degree $(n,n-q-1)$ in $\zeta$. Since $g$ has compact support in $\B^{2n}$, $g|_\omega=0$ and its distributional $\dbar g$ belongs to $L^1(\B^{2n})$, integration by parts gives
    $$0=\int\dbar_\zeta(K_{q-1}(z,\cdot)\wedge g)=\int(\dbar_\zeta K_{q-1})(z,\cdot)\wedge g+\int K_{q-1}(z,\cdot)\wedge\dbar_\zeta g.$$
    On the other hand by convention \eqref{Eqn::Intro::IntConv},
    \begin{align*}
        &\dbar_z\int K_{q-2}(z,\cdot)\wedge g=\sum_{j=1}^nd\bar z_j\wedge\Big(\int\big(\tfrac\partial{\partial\bar z_j}K_{q-2}\big)(z,\cdot)\wedge g\Big)=\sum_{j=1}^n(-1)^{2n-2+q}\Big(\int\big(\tfrac\partial{\partial\bar z_j}K_{q-2}\big)(z,\cdot)\wedge g\Big)\wedge d\bar z_j
        \\
        =&\sum_{j=1}^n(-1)^{2n-2+q}\int\Big(\big(\tfrac\partial{\partial\bar z_j}K_{q-2}\big)(z,\cdot)\wedge g\wedge d\bar z_j\Big)=\sum_{j=1}^n\int d\bar z_j\wedge\big(\tfrac\partial{\partial\bar z_j}K_{q-2}\big)(z,\cdot)\wedge  g=\int\dbar_zK_{q-2}(z,\cdot)\wedge g.
    \end{align*}
    
    The formula $\dbar\Kc-\Kc\dbar=\Bc-\Fc$ then follows from $\dbar_{z,\zeta}K=B-F$.
\end{proof}

The purpose of this section is to give the following standard estimates. Here we use $\zeta=(\zeta',\xi_n+i\eta_n)$.

\begin{prop}\label{Prop::WeiEst::WeiEst}
    Suppose $\omega$ satisfies Assumption~\ref{Assumption::StrongDomain}. Let $\delta(z):=\min(1,\dist(z,b\omega))$.
    
    For every $l,m\ge0$, $0<s< l+m-3/2$ (in particular $l+m\ge2$), there is a $C>0$, depending only on $s,l,m,\|\nabla\sigma\|_{C^{0,1}(\B^{2n-1})}$ and the $c_0$ in \eqref{Eqn::WeiEst::SuppFun}, such that, for every $\alpha\in\N_{z,\bar z}^{2n}$ with $|\alpha|\le l$, 
    \begin{align}
        \label{Eqn::WeiEst::Fz}\int_{\B^{2n}\cap\omega}\delta(z)^s|D_z^\alpha D_{\eta_n}^mF^\rho(z,\zeta)|d\Vol(z)&\le C\delta(\zeta)^{s+1-l-m},&\zeta\in\B^{2n}\backslash\overline\omega;
        \\
    \label{Eqn::WeiEst::Fzeta}\int_{\B^{2n}\backslash\overline\omega}\delta(\zeta)^s|D_z^\alpha D_{\eta_n}^mF^\rho(z,\zeta)|d\Vol(\zeta)&\le C\delta(z)^{s+1-l-m},&  z\in\B^{2n}\cap\omega;
    \\
        \label{Eqn::WeiEst::K1z}\int_{\B^{2n}\cap\omega}\delta(z)^s|D_z^\alpha D_{\eta_n}^mK^\rho(z,\zeta)|d\Vol(z)&\le C\delta(\zeta)^{s+\frac32-l-m},&\zeta\in\B^{2n}\backslash\overline\omega;
        \\
        \label{Eqn::WeiEst::K1zeta}\int_{\B^{2n}\backslash\overline\omega}\delta(\zeta)^s|D_z^\alpha D_{\eta_n}^mK^\rho(z,\zeta)|d\Vol(\zeta)&\le C\delta(z)^{s+\frac32-l-m},& z\in\B^{2n}\cap\omega;
        \\
        \label{Eqn::WeiEst::K2z}\int_{\B^{2n}\cap\omega}|\delta(z)^sD_z^\alpha D_{\eta_n}^mK^\rho(z,\zeta)|^{\frac{2n+2}{2n+1}}d\Vol(z)&\le C\delta(\zeta)^{(s+1-l-m)\frac{2n+2}{2n+1}},&\zeta\in\B^{2n}\backslash\overline\omega;
        \\
        \label{Eqn::WeiEst::K2zeta}\int_{\B^{2n}\backslash\overline\omega}|\delta(\zeta)^sD_z^\alpha D_{\eta_n}^mK^\rho(z,\zeta)|^{\frac{2n+2}{2n+1}}d\Vol(\zeta)&\le C\delta(z)^{(s+1-l-m)\frac{2n+2}{2n+1}},&z\in\B^{2n}\cap\omega.
    \end{align}
\end{prop}

We follow the same argument from \cite[Section~7]{YaoCXFinite}.

Recall the constant $c_0>0$ in \eqref{Assumption::StrongDomain}. For $\zeta\in \B^{2n}$ and $r>0$ we define the $r$-ellipsoid associated to $\rho$:
\begin{equation}
    E_r(\zeta)=E^\rho_r(\zeta):=\{\zeta+w:w\in B\big(0,\sqrt{c_0r}\big):\max_{\theta\in[0,2\pi]}|\rho(\zeta)-\rho(\zeta+e^{i\theta}w)|<r\}.
\end{equation}

For $0<r<1$, let $\zeta\in\B^{2n}\backslash\omega$ and $z=\zeta+w\in\B^{2n}\cap\omega$. Suppose $|w|^2<c_0r$, by assumption \eqref{Assumption::StrongDomain} we have $\{\zeta+te^{i\theta}w:0\le t\le1\}\subset 2\B^{2n}$. By the fundamental theorem of calculus,
\begin{align*}
    |\rho(\zeta+e^{i\theta}w)-\rho(\zeta)-2\re(\partial\rho(\zeta)\cdot e^{i\theta}w)|&=\int_0^1|\nabla\rho(\zeta+te^{i\theta}w)-\nabla\rho(\zeta)|\cdot|e^{i\theta}w|dt\le\frac{\|\nabla^2\rho\|_{L^\infty(2\B^{2n})}|w|^2}2<\frac r2.
\end{align*}
Therefore $\max_{\theta\in[0,2\pi]}|\rho(\zeta+e^{i\theta}w)-\rho(\zeta)|<2|\partial\rho(\zeta)\cdot w|+\frac r2$ in this case. When $z\notin E_r(\zeta)$, then either $|w|^2\ge c_0r$ which implies $|\partial\rho(\zeta)\cdot w|\ge c_0^2r$, or $|\partial\rho(\zeta)\cdot w|>r/4\ge c_0^2r$. Therefore,
\begin{equation}\label{Eqn::WeiEst::SuppFun2}
    |\partial\rho(\zeta)\cdot(\zeta-z)|\ge c_0^2 r,\quad\text{for}\quad\zeta\in\B^{2n}\backslash\omega,\quad z\in\B^{2n}\cap\omega\backslash E_r(\zeta).
\end{equation}

For fixed $\zeta$, let $w_1$ and $w_2$ satisfy $\zeta+w_1\in E_r(\zeta)\cap\B^{2n}$ and $\zeta+w_1+w_2\in E_r(\zeta+w_1)\cap \B^{2n}$. Then for every $\theta\in[0,2\pi]$, $w_1+w_2\in B(0,2\sqrt{c_0r})=B(0,\sqrt{4c_0r})$ and 
\begin{align*}
    |\rho(\zeta+e^{i\theta}(w_1+w_2))-\rho(\zeta)|
    \le &|\rho(\zeta+e^{i\theta}w_1)-\rho(\zeta)|+|\rho(\zeta+w_1+e^{i\theta}w_2)-\rho(\zeta+w_1)|
    \\
    &+|\rho(\zeta+e^{i\theta}(w_1+w_2))-\rho(\zeta+e^{i\theta}w_1)-\rho(\zeta+w_1+e^{i\theta}w_2)+\rho(\zeta+w_1)|
    \\
    \le& r+r+2\|\nabla^2\rho\|_{L^\infty(2\B^{2n})}|w_1||w_2|<2r +\tfrac2{c_0}\cdot c_0r=4r.
\end{align*}
Therefore $\zeta+w_1+w_2\in E_{4r}(\zeta)$, in particular
\begin{equation}\label{Eqn::WeiEst::FlipE}
    \text{for every }z,\zeta\in\B^{2n},\qquad z\in E_r(\zeta)\quad\Longrightarrow\quad \zeta\in E_{4r}(z).
\end{equation}

\begin{lem}\label{Lem::WeiEst::BasisEst}
    Keeping the defining function $\rho(z)=\sigma(z',x_n)-y_n$ from above. For every $l,m\ge0$ there is a constant $C=C(l,m,c_0)>0$ such that, for every $1\le k\le n$,  $0<r<1$, $\zeta\in\B^{2n}\backslash\omega$ and $z\in\B^{2n}\cap\omega\backslash E_r(\zeta)$,
    \begin{equation}\label{Eqn::WeiEst::BasisEst}
        \sum_{|\alpha|\le l}\Big|D_z^\alpha D_{\eta_n}^m\frac{\partial\rho(\zeta)\wedge(\dbar\partial\rho(\zeta))^{k-1}}{(\partial\rho(\zeta)\cdot(\zeta-z))^k}\Big|\le Cr^{-l-m-k}.
    \end{equation}
\end{lem}
\begin{proof}
    Clearly $|D_z^\alpha D_{\eta_n}^m(\partial\rho(\zeta)\wedge(\dbar\partial\rho(\zeta))^{k-1})|\le(1+\|\nabla\sigma\|_{C^{0,1}})(1+\|\nabla^2\sigma\|_{L^\infty})^{k-1}\lesssim1$ as it vanishes when $\alpha\neq0$ or $m\ge1$.

    Since $\rho(\zeta)$ is linear in $\eta_n$, by chain rule we see that for $r>0$, $\zeta\in\B^{2n}\backslash\omega$ and $z\in\B^{2n}\cap\omega\backslash E_r(\zeta)$,
    $$|(D_z^\alpha D_{\eta_n}^m (\partial\rho(\zeta)\cdot(\zeta-z))^{-k})|\lesssim_l|D_{\eta_n}^m (\partial\rho(\zeta)\cdot(\zeta-z))^{-|\alpha|-k}\cdot(\nabla\rho)^{\otimes |\alpha|})|\lesssim r^{-l-m-k},\quad\text{for }|\alpha|\le l.$$ 
    Moreover the implied constant depends only on $m,l,c_0$. This finishes the proof.
\end{proof}

Integrating on $E_r(\zeta)$ we have,

\begin{lem}\label{Lem::WeiEst::IntEst}
    For every $m,l\ge0$ there is a constant $C=C(l,m,c_0,\|\nabla\sigma\|_{L^\infty(\B^{2n-1})})>0$ such that for every $z\in\B^{2n}\cap \omega$ and $\zeta\in\B^{2n}\backslash\omega$, and for every $0<r<1$,
    \begin{align}
    \label{Eqn::WeiEst::IntEst::Fz}
    \sum_{|\alpha|\le l}\int_{\B^{2n}\cap\omega\cap E_r(\zeta)\backslash E_{\frac r2}(\zeta)}|D_z^\alpha D_{\eta_n}^mF^\rho(w,\zeta)|d\Vol(w)&\le Cr^{1-l-m};
    \\
    \label{Eqn::WeiEst::IntEst::Fzeta}
    \sum_{|\alpha|\le l}\int_{\B^{2n}\cap E_r(z)\backslash (E_{\frac r2}(z)\cup\omega)}|D_z^\alpha D_{\eta_n}^mF^\rho(z,w)|d\Vol(w)&\le Cr^{1-l-m};
    \\
    \label{Eqn::WeiEst::IntEst::K1z}
        \sum_{|\alpha|\le l}\int_{\B^{2n}\cap\omega\cap E_r(\zeta)\backslash E_{\frac r2}(\zeta)}|D_z^\alpha D_{\eta_n}^mK^\rho(w,\zeta)|d\Vol(w)&\le Cr^{\frac32-l-m};
        \\\label{Eqn::WeiEst::IntEst::K1zeta}
        \sum_{|\alpha|\le l}\int_{\B^{2n}\cap E_r(z)\backslash (E_{\frac r2}(z)\cup\omega)}|D_z^\alpha D_{\eta_n}^mK^\rho(z,w)|d\Vol(w)&\le Cr^{\frac32-l-m};
        \\\label{Eqn::WeiEst::IntEst::K2z}
        \sum_{|\alpha|\le l}\int_{\B^{2n}\cap\omega\cap E_r(\zeta)\backslash E_{\frac r2}(\zeta)}|D_z^\alpha D_{\eta_n}^mK^\rho(w,\zeta)|^{\frac{2n+2}{2n+1}}d\Vol(w)&\le Cr^{(1-l-m)\frac{2n+2}{2n+1}};
        \\\label{Eqn::WeiEst::IntEst::K2zeta}
        \sum_{|\alpha|\le l}\int_{\B^{2n}\cap E_r(z)\backslash (E_{\frac r2}(z)\cup\omega)}|D_z^\alpha D_{\eta_n}^mK^\rho(z,w)|^{\frac{2n+2}{2n+1}}d\Vol(w)&\le Cr^{(1-l-m)\frac{2n+2}{2n+1}}.
    \end{align}
\end{lem}
\begin{proof}

    We only estimate \eqref{Eqn::WeiEst::IntEst::Fz}, \eqref{Eqn::WeiEst::IntEst::K1z} and \eqref{Eqn::WeiEst::IntEst::K2z}, since by \eqref{Eqn::WeiEst::FlipE} the estimates for \eqref{Eqn::WeiEst::IntEst::Fzeta}, \eqref{Eqn::WeiEst::IntEst::K1zeta} and \eqref{Eqn::WeiEst::IntEst::K2zeta} follow the same argument.

    By Lemma~\ref{Lem::WeiEst::BasisEst} we see that for $\zeta\in \B^{2n}\backslash\omega$ and $z\in\B^{2n}\cap \omega\backslash E_{r/2}(\zeta)$,
    \begin{equation*}
        \sum_{|\alpha|\le l}|D_z^\alpha D^m_{\eta_n}F^\rho(z,\zeta)|\lesssim r^{-l-m-n};\qquad\sum_{|\alpha|\le l}|D_z^\alpha D^m_{\eta_n}K^\rho(z,\zeta)|\lesssim\sum_{k=1}^{n-1}\frac{r^{-l-m-k}}{|z-\zeta|^{2(n-k)-1}}.
    \end{equation*}

    for fixed$r$ we  take a unitary complex coordinate system $u=(u_1,\dots,u_n)=(u',u_n)$ centered at $\zeta$ such that 
    \begin{equation}\label{Eqn::WeiEst::IntEst::CoordTmp}
        \{u:|u_n|<(1+\|\nabla\sigma\|_{L^\infty})^{-1}r,|u'|<\tfrac12\sqrt{c_0r}\}\subseteq E_r(\zeta)\subseteq \{u:|u_n|<2r,|u'|<2\sqrt{c_0r}\}.
    \end{equation}

    Therefore we get \eqref{Eqn::WeiEst::IntEst::Fz} by the following:
    \begin{align*}
        &\sum_{|\alpha|\le l}\int_{\B^{2n}\cap\omega\cap E_r(\zeta)\backslash E_{\frac r2}(\zeta)}|D_z^\alpha D_{\eta_n}^mF^\rho(w,\zeta)|d\Vol(w)\lesssim|E_r(\zeta)|r^{-l-m-n}\approx r^{-l-m+1}.
    \end{align*}
    We get \eqref{Eqn::WeiEst::IntEst::K1z} by the following:
    \begin{align*}
        &\sum_{|\alpha|\le l}\int_{\B^{2n}\cap\omega\cap E_r(\zeta)\backslash E_{\frac r2}(\zeta)}|D_z^\alpha D_{\eta_n}^mK^\rho(w,\zeta)|d\Vol(w)\lesssim\sum_{k=1}^{n-1}\int_{E_r(\zeta)}\frac{r^{-l-m-k}}{|w-\zeta|^{2(n-k)-1}}d\Vol(w)
        \\
        \lesssim&\sum_{k=1}^{n-1}\int_{|u_n|<r,|u'|<\sqrt r}\frac{r^{-l-m-k}}{(|u'|+|u_n|)^{2(n-k)-1}}d\Vol(u)\lesssim\sum_{k=1}^{n-1}\int_0^r t\Big(\int_0^{\sqrt r}\frac{r^{-l-m-k}s^{2n-3}ds}{(t+s)^{2n-2k-1}}\Big)dt
        \\
        \lesssim&\sum_{k=1}^{n-1}r^{-l-m-k}\int_0^r tdt\int_0^{\sqrt r}s^{2k-2}ds\approx\sum_{k=1}^{n-1}r^{-l-m-k}r^2r^{k-\frac12}\approx r^{\frac32-m-l}.
    \end{align*}
    For \eqref{Eqn::WeiEst::IntEst::K2z}, by the similar calculation to that for \eqref{Eqn::WeiEst::IntEst::K1z}:
    \begin{align*}
        &\sum_{|\alpha|\le l}\int_{\B^{2n}\cap\omega\cap E_r(\zeta)\backslash E_{\frac r2}(\zeta)}|D_z^\alpha D_{\eta_n}^mK^\rho(w,\zeta)|^{\frac{2n+2}{2n+1}}d\Vol(w)\lesssim\sum_{k=1}^{n-1}\int_{E_r(\zeta)}\frac{r^{-(l+m+k)\frac{2n+2}{2n+1}}}{|w-\zeta|^{(2n-2k-1)\frac{2n+2}{2n+1}}}d\Vol(w)
        \\
        \lesssim&\sum_{k=1}^{n-1}\int_{|u_n|<r,|u'|<\sqrt r}\frac{r^{-(l+m+k)\frac{2n+2}{2n+1}}}{(|u'|+|u_n|)^{(2n-2k-1)\frac{2n+2}{2n+1}}}d\Vol(u)\lesssim\sum_{k=1}^{n-1}\int_0^r t\Big(\int_0^{\sqrt r}\frac{r^{-(l+m+k)\frac{2n+2}{2n+1}}s^{2n-3}ds}{(t+s)^{(2n-2k-1)\frac{2n+2}{2n+1}}}\Big)dt
        \\
        \lesssim&\sum_{k=1}^{n-1}\int_0^r tdt\int_0^{\sqrt r}\frac{s^{2n-3-(2n-2k-1)\frac{2n+2}{2n+1}}ds}{r^{(l+m+k)\frac{2n+2}{2n+1}}}\lesssim\sum_{k=1}^{n-1}\frac{r^2r^{n-1-(n-k-\frac12)\frac{2n+2}{2n+1}}}{r^{(l+m+k)\frac{2n+2}{2n+1}}}\approx r^{(1-l-m)\frac{2n+2}{2n+1}}.
    \end{align*}
    
    Note that all the implied constants in the estimate above depend only on $m,l,c_0$ and $\|\nabla\sigma\|_{C^{0,1}(2\B^{2n-1})}$, completing the proof.
\end{proof}

\begin{proof}[Proof of Proposition~\ref{Prop::WeiEst::WeiEst}]
    We only prove \eqref{Eqn::WeiEst::Fz}, \eqref{Eqn::WeiEst::K1z} and \eqref{Eqn::WeiEst::K2z} as the estimates for \eqref{Eqn::WeiEst::Fzeta}, \eqref{Eqn::WeiEst::K1zeta} and \eqref{Eqn::WeiEst::K2zeta}  follow the same argument except we replace \eqref{Eqn::WeiEst::IntEst::Fz}, \eqref{Eqn::WeiEst::IntEst::K1z} and \eqref{Eqn::WeiEst::IntEst::K2z} by \eqref{Eqn::WeiEst::IntEst::Fzeta}, \eqref{Eqn::WeiEst::IntEst::K1zeta} and \eqref{Eqn::WeiEst::IntEst::K2zeta} respectively.

    Fix a $\zeta\in\B^{2n}\backslash\overline\omega$ let $J_0\in\N$ be the smallest number such that $E_{2^{-J_0}}(\zeta)\subset\omega^c$. Clearly $2^{-J_0}\approx\delta(\zeta)$ and the implied constant depends only on $c_0$ and $\|\nabla\sigma\|_{L^\infty}$. Therefore applying \eqref{Eqn::WeiEst::IntEst::Fz} we have \eqref{Eqn::WeiEst::Fz}:
    \begin{align*}
        &\sum_{|\alpha|\le l}\int_{\B^{2n}\cap\omega}\delta(z)^s|D^\alpha_zD^m_{\eta_n}F^\rho(z,\zeta)|d\Vol(z)
        \\
        =&\sum_{|\alpha|\le l}\Big(\int_{\B^{2n}\cap\omega\backslash E_1(\zeta)}+\sum_{j=0}^{J_0}\int_{\B^{2n}\cap\omega\cap E_{2^{-j}}(\zeta)\backslash E_{2^{-j-1}}(\zeta)}\Big)\delta(z)^s|D^\alpha_zD^m_{\eta_n}F^\rho(z,\zeta)|d\Vol(z)
        \\
        \lesssim&1+\sum_{j=0}^{J_0}2^{j(l+m-1-s)}\approx 2^{J_0(l+m-1-s)}\approx\delta(\zeta)^{s+1-l-m}.
    \end{align*}
    
    Applying \eqref{Eqn::WeiEst::IntEst::K1z} and \eqref{Eqn::WeiEst::IntEst::K2z} similarly we have
    \begin{align*}
        \sum_{|\alpha|\le l}\int_{\B^{2n}\cap\omega}\delta(z)^s|D^\alpha_zD^m_{\eta_n}K^\rho(z,\zeta)|d\Vol(z)&\lesssim1+\sum_{j=0}^{J_0}2^{j(l+m-\frac32-s)}\approx 2^{J_0(l+m-\frac32-s)}\approx\delta(\zeta)^{s+\frac32-l-m};
        \\
        \sum_{|\alpha|\le l}\int_{\B^{2n}\cap\omega}|\delta(z)^sD^\alpha_zD^m_{\eta_n}K^\rho(z,\zeta)|^{\frac{2n+2}{2n+1}}d\Vol(z)&\lesssim1+\sum_{j=0}^{J_0}2^{j(l+m-1-s)\frac{2n+2}{2n+1}}\approx\delta(\zeta)^{(s+1-l-m)\frac{2n+2}{2n+1}}.
    \end{align*}
    This proves \eqref{Eqn::WeiEst::K1z} and \eqref{Eqn::WeiEst::K2z}.
\end{proof}

Finally we convert the weighted estimate into boundedness of integral operators. Recall the weighted Sobolev spaces in Definition~\ref{Defn::Space::WeiSob}.

\begin{cor}\label{Cor::WeiEst::WeiOp}
    For  $m\ge0$ we define $\Fc^{m,\rho}$ and $\Kc^{m,\rho}_q$ for $2\le q\le n$ by
    \begin{align}\label{Eqn::WeiEst::OpF}
        \Fc^{m,\rho}g(z):=&\int_{\B^{2n}\backslash\overline\omega} (D^m_{\eta_n}F^\rho(z,\cdot))\wedge g,&z\in\B^{2n}\cap\omega,\qquad g\in L^1(\B^{2n}\backslash\overline\omega;\wedge^{0,1});
        \\\label{Eqn::WeiEst::OpK}
        \Kc^{m,\rho}_qg(z):=&\int_{\B^{2n}\backslash\overline\omega}(D^m_{\eta_n}K_{q-2}^\rho(z,\cdot))\wedge g,&z\in\B^{2n}\cap\omega,\qquad  g\in L^1(\B^{2n}\backslash\overline\omega;\wedge^{0,q}).
    \end{align}

    Then for every  $l\ge0$ and $\frac32<s<l+m-\frac12$ (in particular $l+m\ge3$),
    \begin{align}\label{Eqn::WeiEst::SobEstF}
        \Fc^{m,\rho}:&L^p(\B^{2n}\backslash\overline\omega,\delta^{1-s};\wedge^{0,1})\to W^{l,p}(\B^{2n}\cap\omega,\delta^{m+l-s}),& 1\le p\le\infty;
    \\
    \label{Eqn::WeiEst::SobEstK1}
        \Kc^{m,\rho}_q:&L^p(\B^{2n}\backslash\overline\omega,\delta^{1-s};\wedge^{0,q})\to W^{l,p}(\B^{2n}\cap\omega,\delta^{m+l-\frac12-s};\wedge^{0,q-2}),&1\le p\le \infty;
        \\\label{Eqn::WeiEst::SobEstK2}
        \Kc^{m,\rho}_q:&L^p(\B^{2n}\backslash\overline\omega,\delta^{1-s};\wedge^{0,q})\to W^{l,\frac{(2n+2)p}{2n+2-p}}(\B^{2n}\cap\omega,\delta^{m+l-s};\wedge^{0,q-2}),&1\le p\le 2n+2.
    \end{align}

    In particular, using the Triebel-Lizorkin spaces from Definition~\ref{Defn::Space::TLSpace}, for every $s>\frac32$, $\eps>0$ and $2\le q\le n$,
    \begin{align}\label{Eqn::WeiEst::SobBddF}
        \Fc^{m,\rho}:&\widetilde\Fs_{p\infty}^{s-1}(\overline\B^{2n}\backslash\omega;\wedge^{0,1})\to\Fs_{p\eps}^{s-m}(\B^{2n}\cap\omega)&1\le p\le \infty;
        \\\label{Eqn::WeiEst::SobBddK1}
        \Kc^{m,\rho}_q:&\widetilde\Fs_{p\infty}^{s-1}(\overline\B^{2n}\backslash\omega;\wedge^{0,q})\to\Fs_{p\eps}^{s+1/2-m}(\B^{2n}\cap\omega;\wedge^{0,q-2})&1\le p\le \infty;
        \\\label{Eqn::WeiEst::SobBddK2}
        \Kc^{m,\rho}_q:&\widetilde\Fs_{p\infty}^{s-1}(\overline\B^{2n}\backslash\omega;\wedge^{0,q})\to\Fs_{\frac{(2n+2)p}{2n+2-p},\eps}^{s-m}(\B^{2n}\cap\omega;\wedge^{0,q-2})&1\le p\le 2n+2;
    \end{align}
\end{cor}
\begin{rem}
    Using the notations \eqref{Eqn::WeiEst::DefKOp} and \eqref{Eqn::WeiEst::DefFOp} we have $\Fc^{0,\rho}=\Fc^\rho$ and $\Kc^{0,\rho}=\Kc^\rho$.

    Using integration by parts, for $1\le q\le n$, we get the relation
    \begin{equation*}
        \Kc^{m,\rho}_qg=(-1)^m\Kc^{0,\rho}_q\circ D_{y_n}^mg\qquad\big(\Fc^{m,\rho} g=(-1)^m\Fc^{0,\rho}\circ D_{y_n}^mg\text{ if }q=1\big),\qquad g\in \Co_c^\infty(\B^{2n}\backslash\overline\omega;\wedge^{0,q}).
    \end{equation*}
\end{rem}
\begin{proof}
    For $\alpha\in\N^{2n}_{z,\bar z}$, $2\le q\le n$ and $g\in L^1(\B^{2n}\backslash\overline\omega;\wedge^{0,q})$, we have $D^\alpha_z\Kc^{m,\rho}_q g(z)=\int(D^\alpha_zD^m_{\eta_n} K_{q-2})(z,\cdot)\wedge g$, where we have $D^\alpha_z\Fc^{m,\rho} g(z)=\int(D^\alpha_zD^m_{\eta_n} F)(z,\cdot)\wedge g$ when $q=1$. 

    We apply Lemma~\ref{Lem::Space::Schur} with $(X,\mu)=(\B^{2n}\backslash\overline\omega,d\Vol(\zeta))$ and $(Y,\nu)=(\B^{2n}\cap\omega,d\Vol(z))$ and the following:
    \begin{align*}
        \text{for \eqref{Eqn::WeiEst::SobEstF}},&& G(\zeta,z)=&\delta(\zeta)^{s-1}|F^\rho(z,\zeta)|\delta(z)^{m+l-s},&&\gamma=1;
        \\
        \text{for \eqref{Eqn::WeiEst::SobEstK1}},&& G(\zeta,z)=&\delta(\zeta)^{s-1}|K_{q-1}^\rho(z,\zeta)|\delta(z)^{m+l-1/2-s},&&\gamma=1;
        \\
        \text{for \eqref{Eqn::WeiEst::SobEstK2}},&& G(\zeta,z)=&\delta(\zeta)^{s-1}|K_{q-1}^\rho(z,\zeta)|\delta(z)^{m+l-s},&&\gamma=\tfrac{2n+2}{2n+1}.
    \end{align*}
    Therefore by Proposition~\ref{Prop::WeiEst::WeiEst}, for every $|\alpha|\le l$ and $2\le q\le n$,
    \begin{align*}
        D^\alpha\Fc^{m,\rho}:&L^p(\B^{2n}\backslash\overline\omega,\delta^{1-s};\wedge^{0,1})\to L^p(\B^{2n}\cap\omega,\delta^{m+l-s}),&&1\le p\le \infty;
        \\
        D^\alpha\Kc_q^{m,\rho}:&L^p(\B^{2n}\backslash\overline\omega,\delta^{1-s};\wedge^{0,q})\to L^p(\B^{2n}\cap\omega,\delta^{m+l-1/2-s};\wedge^{0,q-2}),&&1\le p\le \infty;
        \\
        D^\alpha\Kc_q^{m,\rho}:&L^p(\B^{2n}\backslash\overline\omega,\delta^{1-s};\wedge^{0,q})\to L^\frac{(2n+2)p}{2n+2-p}(\B^{2n}\cap\omega,\delta^{m+l-s};\wedge^{0,q-2}),&&1\le p\le 2n+2.
    \end{align*}
    These bounds give \eqref{Eqn::WeiEst::SobEstF}, \eqref{Eqn::WeiEst::SobEstK1} and \eqref{Eqn::WeiEst::SobEstK2}. By Proposition~\ref{Prop::Space::HLLem} (see also Corollary~\ref{Cor::Space::IntOpBdd}) we get \eqref{Eqn::WeiEst::SobBddF}, \eqref{Eqn::WeiEst::SobBddK1} and \eqref{Eqn::WeiEst::SobBddK2}. 
\end{proof}

Later in Theorem~\ref{Thm::LocalHT} we need the following auxillary result.
\begin{lem}\label{Lem::WeiEst::DisjointSupp}
    Let $\lambda\in\Co^\infty(\B^{2n}\cap\omega)$ and let $g\in L^1(\B^{2n}\backslash\overline\omega;\wedge^{0,1})$ satisfy $\supp\lambda\cap\supp g=\varnothing$. Then $\lambda\cdot\Fc^{m,\rho}g\in\Co^\infty(\B^{2n}\cap\omega)$.
\end{lem}
Here $\supp\lambda\subset\overline{\B^{2n}}\cap\overline\omega$ and $\supp g\subseteq\overline{\B^{2n}}\backslash\omega$ are closed subsets in $\C^n$ (hence are compact as well).
\begin{proof}
    By \eqref{Eqn::WeiEst::DefF} and Lemma~\ref{Lem::WeiEst::BasisEst}, we see that if $A\subseteq\overline{\B^{2n}}\cap\overline\omega$ and $B\subseteq\overline{\B^{2n}}\backslash\omega$ are two disjoint compact sets, then $D_{\eta_n}^mF(\cdot,\zeta)\in\Co^\infty(A)$ uniformly for $\zeta\in B$. More precisely
    \begin{equation*}
        \sup_{z\in A,\zeta\in B}|D^\alpha_zD^m_{\eta_n}F(z,\zeta)|<\infty,\quad\text{for all }\alpha\in\N^{2n},\quad m\ge0.
    \end{equation*}

    By assumption $\dist(\supp\lambda,\supp g)>0$. Therefore for every $\alpha,\beta$ and $m\ge0$,
    \begin{align*}
        \sup_{z\in\B^{2n}\cap\omega}|D^\alpha\lambda(z)D^\beta\Fc^{m,\rho}g(z)|\le&\|\lambda\|_{C^{|\alpha|}}\sup_{z\in\supp\lambda}\int_{\supp g}|D^\beta_zD^m_{\eta_n}F(z,\zeta)||g(\zeta)|d\Vol(\zeta)
        \\
        \le&\|\lambda\|_{C^{|\alpha|}}\|g\|_{L^1}\sup_{z\in \supp\lambda,\zeta\in \supp g}|D^\beta_zD^m_{\eta_n}F(z,\zeta)|<\infty.
    \end{align*}
    This completes the proof.
\end{proof}

\section{Local Homotopy Formulae on $C^{1,1}$ Special Domains}\label{Section::LocalHT}

For completeness let us recall the definition of strongly $\C$-linearly convexifiable domains.

\begin{defn}\label{Defn::WeiEst::Convex}
    Let $\Omega\subset\C^n$ be a bounded $C^{1,1}$ domain with a $C^{1,1}$ defining function $\varrho:\C^n\to\R$, i.e. we have $\Omega=\{z\in\C^n:\varrho(z)<0\}$ and $\nabla\varrho(\zeta)\neq0$ for all $\zeta\in b\Omega$.


    We say that $\Omega$ is \textit{strongly $\C$-linearly convex}, if there is a $c>0$ such that,
        \begin{equation*}
            |\partial\varrho(\zeta)\cdot(\zeta-z)|\ge c|z-\zeta|^2,\quad\text{for every }z\in\Omega\text{ and }\zeta\in b\Omega.
        \end{equation*}

    We say that $\Omega$ is \textit{locally strongly $\C$-linearly convexifiable}, if for every $\zeta\in b\Omega$ there is a $U\ni\zeta$ and a map $F:U\to\C^n$ which is biholomorphic onto its image, such that 
    \begin{equation*}
            |\partial(\varrho\circ F^{-1})(\eta)\cdot(\eta-w)|\ge c|w-\eta|^2\quad\text{ for every }w\in F(U\cap\Omega)\text{ and }\eta\in F(U\cap b\Omega).
    \end{equation*}
\end{defn}

These definitions are independent of the choice of $\varrho$, after shrinking $U$ if necessary in the local
case. See e.g. \cite[Theorem~1.1]{GongLanzaniCConvex}.

\begin{rem}
    When the domain has $C^2$ boundary, the locally strongly $\C$-linearly convexifiability coincides with the standard strongly pseudoconvexity.

    Recall that we say a $C^2$ domain $\Omega$ (with $C^2$ defining function $\varrho$) is \textit{(Levi) strongly pseudoconvex}, if there is a $c>0$ such that for every $\zeta\in b\Omega$,
        \begin{equation}\label{Eqn::LocalHT::SPsiCX}
            \sum_{j,k=1}^n\frac{\partial^2\varrho}{\partial z_j\partial\bar z_k}(\zeta)w_j\bar w_k\ge c|w|^2,\quad\text{for every }w\in\C^n\text{ such that }\partial\varrho(\zeta)\cdot w=\sum_{j=1}^n\frac{\partial\varrho}{\partial z_j}(\zeta)w_j=0.
        \end{equation}
    The definition does not depend on the choice $\varrho$, see e.g. \cite[Page~56]{RangeSCVBook}.
    
    Note that by \eqref{Eqn::LocalHT::SPsiCX}, strong pseudoconvexity is a local property invariant under biholomorphic changes of coordinates. Since every $C^2$ strongly $\C$-linearly convex domain is strongly pseudoconvex \cite[Lemma~3]{LanzaniStein2013}, every $C^2$ locally strongly $\C$-linearly convexifiable domain is strongly pseudoconvex.
    
    Conversely by \cite[Theorem~II.2.17]{RangeSCVBook}, every $C^2$ strongly pseudoconvex domain is locally strongly real convexifiable, which in particular is strongly $\C$-linearly convexifiable.
\end{rem}

In this section we identify $\R^{2n}$ with $\C^n$ by the correspondence of $u=(u',u'')\in\R^{2n}$ and $u'+iu''\in\C^n$. And we denote by $D_u$ the real directional derivative on the $\zeta$-variable:
\begin{equation}\label{Eqn::LocalHT::DirectionK}
    D_uK(z,\zeta)=\frac d{dt}K(z,\zeta+t(u'+iu''))\Big|_{\R\ni t=0}.
\end{equation}

We are going to construct parametrix formulae on the domain $\omega=\{z\in \C^n:y_n>\sigma(z',x_n)\}$, which satisfies a slightly stronger assumption to Assumption~\ref{Assumption::StrongDomain}.
\begin{assu}\label{Assumption::DomainRotate}
    We assume the open subset $\omega\subset\C^n$ has the form $\omega=\{z\in\C^n:y_n>\sigma(z',x_n)\}$ where $\sigma:\R^{2n-1}\to\R$ is a Lipschitz function such that $\|\nabla\sigma\|_{L^\infty(\R^{2n-1})}<1$. 
    
    For every vector $u\in\{z:y_n<-3|(z',x_n)|\}$, we denote by $\rho_u:\C^n\to\R$ the unique Lipschitz function such that $D_u\rho_u\equiv1$ and $\rho_u(z)=0$ for all $z\in b\omega$. We assume for such $u$, $\rho_u|_{2\B^{2n}}$ is $C^{1,1}$, and there is a $0<c_u<1/4$ such that $\|\nabla^2\rho_u\|_{L^\infty(2\B^{2n})}<1/c_u$ and
    \begin{equation*}
        |\partial\rho_u(\zeta)\cdot(z-\zeta)|\ge c_u(\rho_u(\zeta)-\rho_u(z)+|z-\zeta|^2),\quad z\in\B^{2n}\cap\omega,\quad\zeta\in\B^{2n}\backslash\omega.
    \end{equation*}
\end{assu}
\begin{rem}
    Here we recall for clarification that  $\rho(z)=\sigma(z',x_n)-y_n$ is the unique defining function such that $D_{-e_{2n}}\rho=-\frac\partial{\partial y_n}\rho\equiv1$, where $-e_{2n}=(0,\dots,0,-i)\in\{y_n<-3|(z',x_n)|\}$.
\end{rem}

When $\Omega$ is the domain in Theorem~\ref{Thm::MainThm}, then locally it can be biholomorphic to such $\omega$.
\begin{lem}\label{Lem::LocalHT::LocalDomain}
    Let $\Omega$ be a locally strongly $\C$-linearly convexifiable domain with $C^{1,1}$ boundary. Then for every $\zeta_0\in b\Omega$ there is a neighborhood $U\ni\zeta_0$ and a biholomorphic map $F:U\xrightarrow{\simeq}\B^{2n}$ and a $C^{1,1}$ function $\sigma_0:\R^{2n-1}\to\R$ such that 
    \begin{enumerate}[(i)] 
        \item Let $\omega=\{z\in\C^n:y_n>\sigma_0(z',x_n)\}$, we have $F(\zeta_0)=0$ (thus $\sigma_0(0)=0$), $\omega\cap \B^{2n}=F(U\cap\Omega)$.
        \item\label{Item::LocalHT::LocalDomain::RhoPhi} For every unitary transform $\Psi:\C^n\to\C^n$ such that $\Psi(0,\dots,0,i)\in\{z:y_n>3|(z',x_n)|\}$, the domain $\Psi(\omega)$ satisfies Assumption~\ref{Assumption::StrongDomain}. That is, for each such $\Psi$, we have $\Psi(\omega)=\{y_n>\sigma_\Psi(z',x_n)\}$ for some $C^{1,1}$ function $\sigma_\Psi:\R^{2n-1}\to\R$ such that $\|\nabla\sigma_\Psi\|_{L^\infty}<1$ and there is a $c_\Psi>0$ such that for $\rho_\Psi(z)=\sigma_\Psi(z',x_n)-y_n$,
        \begin{equation}\label{Eqn::LocalHT::SuppFunPhi}
            |\partial\rho_\Psi(\zeta)\cdot(\zeta-z)|\ge c_\Psi(\rho_\Psi(\zeta)-\rho_\Psi(z)+|z-\zeta|^2),\qquad z\in \B^{2n}\cap\Psi(\omega),\quad \zeta\in\B^{2n}\backslash\Psi(\overline{\omega}).
        \end{equation}
    \end{enumerate}

    In particular $\omega$ satisfies Assumption~\ref{Assumption::DomainRotate}.
\end{lem}
By a finer analysis we can choose $\{c_\Psi\}_\Psi$ to be a number that does not depend on $\Psi$. A weaker result is sufficient for our need.
\begin{proof}
    Note that if $\|\nabla\sigma_0\|_{L^\infty(\R^{2n-1})}<\frac12$ and $\Psi(0,\dots,0,i)\in\{z:y_n>3|(z',x_n)|\}$, then $\sigma_\Psi$ is well-defined. And by the fact $\arctan\frac12+\arctan\frac13=\arctan1$, we obtain $\|\nabla\sigma_\Psi\|_{L^\infty(\R^{2n-1})}<1$. 

    By passing to a translation we can assume that $\zeta_0=0$. 
    By Definition~\ref{Defn::WeiEst::Convex}, $b\Omega$ is locally strongly $\C$-linearly convex after a biholomorphic change $\tilde F:\tilde U\to\C^n$ where $\tilde U\ni 0$. Since strongly $\C$-linearly convexity is preserved under complex linear transformation, we can assume that $\tilde F(z)=z+O(|z|^2)$ as $z\to0$, i.e. $\tilde F(0)=0$ and $\partial_z\tilde F(0)=I_n$. By  Implicit Function Theorem, there is an $\eta>0$ and a $C^{1,1}$ function $\tilde\sigma:\eta\B^{2n-1}\to\R$ such that $\eta\B^{2n}\subset\tilde F(\tilde U)$ and  $\eta\B^{2n}\cap \tilde F(\tilde U\cap\Omega)=\{z\in\eta\B^{2n}:y_n>\tilde \sigma(z',x_n)\}$. 
    Clearly $\tilde\sigma(0)=0$ and $\nabla\tilde\sigma(0)=0$. Since $\tilde\sigma\in C^{1,1}$, we have scaling estimate $\|\frac1\delta\nabla(\tilde\sigma(\delta\cdot))\|_{L^\infty(\B^{2n-1})}=\|(\nabla\tilde\sigma)(\delta\cdot)\|_{L^\infty(\B^{2n-1})}\lesssim_\eta\delta\|\nabla\tilde\sigma\|_{L^\infty(\eta\B^{2n-1})}$ for every $\delta\in(0,\eta)$. Choose a bounded extension operator $E:C^{1,1}(\B^{2n-1})\to C^{1,1}(\R^{2n-1})$, and take $\delta>0$ small enough such that $\|\nabla(\tilde\sigma(\delta\cdot))\|_{L^\infty(\B^{2n-1})}<(2\|E\|_{C^{1,1}\to C^{1,1}})^{-1}\delta$,
    take $F(z):=\delta^{-1}\tilde F(z)$, $U:=F^{-1}(\B^{2n})$ and $\sigma_0:=E(\delta^{-1}\tilde\sigma(\delta\cdot))$ we get $\|\nabla\sigma_0\|_{L^\infty(\R^{2n-1})}<\frac12$, thus $\|\nabla\sigma_\Psi\|_{L^\infty(\R^{2n-1})}<1$ holds for all such $\Psi$.

    After further shrinking the chart, \eqref{Eqn::LocalHT::SuppFunPhi} follows from the proof of \cite[Proposition 2.4 and Lemma 2.2]{GongLanzaniCConvex}.
\end{proof}



We now construct the parametrix formulae $f=\Pc f+\Hc'_1\dbar f+\Rc_0f$ for functions $f$, and for $1\le q\le n$, $f=\dbar\Hc'_qf+\Hc'_{q+1}\dbar f+\Rc_qf$ for $(0,q)$ forms $f$ on the domain $\frac12\B^{2n}\cap\omega$.

Take real linearly independent unit vectors $v=(v_1,\dots,v_{2n})\subset\R^{2n}(\simeq\C^n)$ such that
\begin{equation}\label{Eqn::LocalHT::AssumpCone}
    V_v:=\{a_1v_1+\dots+a_{2n}v_{2n}:a_1,\dots,a_{2n}>0\}\subset\{z:y_n<-3|(z',x_n)|\}.
\end{equation}
Note that $\omega-K_v\subset\omega$ and $v_1,\dots,v_{2n}$ all go toward the outer direction of $b\omega$ transversally.

Let $\rho_1,\dots,\rho_{2n}:\C^n\to\R$ be the unique defining functions for $\omega$ such that
\begin{equation}
    D_{v_j}\rho_j(\zeta)\equiv1,\qquad\zeta\in\C^n,\qquad1\le j\le 2n.
\end{equation}



For such $\rho_1,\dots,\rho_{2n}$ we take $K^j=K^{\rho_j}$ and $F^j=F^{\rho_j}$ to be the corresponding Leray-Koppelman kernel and Cauchy-Fantappi\`e kernel given in \eqref{Eqn::WeiEst::DefK} and \eqref{Eqn::WeiEst::DefF}.

Let $\Sc^{m,k}f=\mu^{m,k}\ast f$ for $m=k=0$ or $m\ge0$, $1\le k\le 2n$ be the anti-derivative operators associated to $v_1,\dots,v_{2n}$ which are obtained from Theorem~\ref{Thm::RefAtD}. Recall also from Corollary~\ref{Cor::RefAtD::ScBdd} that
\begin{itemize}
    \item $f=\sum_{j=0}^{2n}\Sc^{0,j}f$ and $\Sc^{0,k}f=D_{v_k}^m\Sc^{m,k}f$ for all $1\le k\le 2n$, $m\ge1$ and $f\in\Ss'(\C^n)$.
    \item $\Sc^{0,0},\Sc^{m,k}:\widetilde\Fs_{pq}^s(\omega^c)\to\widetilde\Fs_{pq}^{s+m}(\omega^c)$ are bounded for all $p,q\in(0,\infty]$, $s\in\R$, $m\ge0$ and $1\le k\le 2n$. 
\end{itemize}

Let $\Ec:\Ss'(\omega)\to\Ss'(\C^n)$ to be the Rychkov's extension operator from Definition~\ref{Defn::Space::ExtOp}. We extend its definition to forms by acting on their components. Here recall that $\Ec=\sum_{j=0}^\infty\psi_j\ast(\1_\omega\cdot(\phi_j\ast f))$ for some $\psi_j,\phi_j$ supported in $-\Kb:=\{(z',x_n+iy_n):y_n<-|(z',x_n)|\}$. Therefore for every $f\in\Ss'(\omega)$,
\begin{equation*}
    \supp[\dbar,\Ec]f\subseteq(-\overline\Kb+\supp f)\backslash\omega.
\end{equation*}
Here we use the convention $\supp f\subseteq\overline\omega$, which is a closed set in $\C^n$.

Since $V_v\subset-\Kb$, we see that for every $m,k$ the following holds as well:
\begin{equation}\label{Eqn::LocalHT::SuppFact}
    \supp\Sc^{m,k}[\dbar,\Ec]f\subseteq(-\overline\Kb+\supp f)\backslash\omega.
\end{equation}

Take $\chi\in C_c^\infty(\B^{2n})$ such that $\chi|_{\frac12\B^{2n}}\equiv1$. We define $\Pc$, $\Hc'_q$ ($1\le q\le n$), $\Rc_q$ ($0\le q\le n-1$) by the following:
\begin{align}
    \label{Eqn::LocalHT::P}
    \Pc' f(z):=&\sum_{j=1}^{2n}\int_{\B^{2n}\backslash\overline\omega}F^j(z,\cdot)\wedge \chi\cdot\Sc^{0,j}[\dbar,\Ec]f;
    \\
    \label{Eqn::LocalHT::H'}
    \Hc'_qf(z):=&\int_{\B^{2n}}B_{q-1}(z,\cdot)\wedge\chi\cdot\Ec f+\sum_{j=1}^{2n}\int_{\B^{2n}\backslash\overline\omega}K_{q-1}^j(z,\cdot)\wedge \chi\cdot\Sc^{0,j}[\dbar,\Ec]f;
    \\
    \label{Eqn::LocalHT::R}
    \Rc_qf(z):=&\int_{\B^{2n}}B_q(z,\cdot)\wedge\big(\dbar\chi\wedge\Ec f+\chi\cdot\Sc^{0,0}[\dbar,\Ec]f\big)-\sum_{j=1}^{2n}\int_{\B^{2n}\backslash\overline\omega}K_q^j(z,\cdot)\wedge \dbar\chi\wedge\Sc^{0,j}[\dbar,\Ec]f.
\end{align}

\begin{thm}\label{Thm::LocalHT}
    Let $\omega=\{z\in\C^n:y_n>\sigma(z',x_n)\}$ where $\sigma:\R^{2n-1}\to\R$ satisfies Assumption~\ref{Assumption::DomainRotate}. 
    
    The $\Pc',\Hc'_q,\Rc_q$ given above satisfy the following (we still use $\Hc'_0=\Rc_n=\Hc'_{n+1}=0$):
    \begin{enumerate}[(i)]
        \item\label{Item::LocalHT::Bdd} $\Pc':\Ss'(\omega)\to\Ss'(\tfrac12\B^{2n}\cap\omega)$, $\Hc'_q:\Ss'(\omega;\wedge^{0,q})\to\Ss'(\tfrac12\B^{2n}\cap\omega;\wedge^{0,q-1})$ and $\Rc_q:\Ss'(\omega;\wedge^{0,q})\to\Ss'(\tfrac12\B^{2n}\cap\omega;\wedge^{0,q})$ are all defined. Moreover for $s\in\R$ and $\eps>0$, we have the following boundedness:
        \begin{align}\label{Eqn::LocalHT::BddP'}
            &\Pc':\Fs_{p\infty}^s(\omega)\to \Fs_{p\eps}^s(\tfrac12\B^{2n}\cap\omega),&p\in[1,\infty];
            \\\label{Eqn::LocalHT::BddH'1}
            &\Hc'_q:\Fs_{p\infty}^s(\omega;\wedge^{0,q})\to\Fs_{p\eps}^{s+1/2}(\tfrac12\B^{2n}\cap\omega;\wedge^{0,q-1}),&p\in[1,\infty],&& 1\le q\le n;
            \\\label{Eqn::LocalHT::BddH'2}
            &\Hc'_q:\Fs_{p\infty}^s(\omega;\wedge^{0,q})\to
                \Fs_{\frac{(2n+2)p}{2n+2-p},\eps}^s(\tfrac12\B^{2n}\cap\omega;\wedge^{0,q-1}),&p\in[1,2n+2],&& 1\le q\le n;
            \\\label{Eqn::LocalHT::BddR}&\Rc_q:\Fs_{p\infty}^s(\omega;\wedge^{0,q})\to\Fs_{p\eps}^{s+\frac12}(\tfrac12\B^{2n}\cap\omega;\wedge^{0,q}),& p\in[1,\infty],&&0\le q\le n-1;
            \\
            \label{Eqn::LocalHT::BddH'n}
            &\Hc'_n:\Fs_{pr}^s(\omega;\wedge^{0,n})\to
                \Fs_{pr}^{s+1}(\tfrac12\B^{2n}\cap\omega;\wedge^{0,n-1}),&p,r\in(0,\infty].
        \end{align}
        As a direct corollary, for every $0\le q\le n-1$ and $s\in\R$,
        \begin{itemize}
            \item $\Rc_q:H^{s,2}(\omega;\wedge^{0,q})\to H^{s+\frac12,2}(\tfrac12\B^{2n}\cap\omega;\wedge^{0,q})$, $\Hc'_{q+1}:H^{s,2}(\omega;\wedge^{0,q+1})\to H^{s+\frac12,2}(\tfrac12\B^{2n}\cap\omega;\wedge^{0,q})$ are bounded.
            
        \end{itemize} 
        
        

        \item\label{Item::LocalHT::HT} For every $0\le q\le n$ and $f\in\Ss'(\omega;\wedge^{0,q})$,
        \begin{align}
            \label{Eqn:LocalHT::HTP}
        f|_{\frac12\B^{2n}\cap\omega}&=\Pc' f+\Hc_1'\dbar f+\Rc_0f,&\dbar\Pc'f&=0,&& \text{when }q=0;
        \\
        \label{Eqn:LocalHT::HTH}
        f|_{\frac12\B^{2n}\cap\omega}&=\dbar\Hc_q' f+\Hc_{q+1}'\dbar f+\Rc_q f,&&&&\text{when }1\le q\le n.
        \end{align}

        In other words \eqref{Eqn::PM::ConcretePM::LocalHTAssum} are  satisfied.
        \item\label{Item::LocalHT::HoloP'} For every $\lambda_1,\lambda_2\in\Co_c^\infty(\tfrac12\B^{2n})$ that have disjoint supports, $[f\mapsto\lambda_1\Pc'(\lambda_2f)]:\Ss'(\omega)\to\Co^\infty(\omega)$. 
    \end{enumerate}
\end{thm}
\begin{rem}
    Note that by Lemmas~\ref{Lem::Space::TLEmbed} and \ref{Lem::Space::SpaceLem}, \eqref{Eqn::LocalHT::BddP'} - \eqref{Eqn::LocalHT::BddR} imply that for every $s\in\R$, $\Pc':H^{s,p}\to H^{s,p}$, $\Hc'_q,\Rc_q:H^{s,p}\to H^{s+\frac12,p}$ for all $1<p<\infty$, $\Hc'_q:H^{s,p}\to H^{s,(2n+2)p/(2n+2-p)}$ for all $1<p<2n+2$, $\Hc'_n:H^{s,p}\to H^{s+1,p}$, and $\Pc':\Co^s\to\Co^s$, $\Hc'_q,\Rc_q:\Co^s\to\Co^{s+1/2}$, $\Hc'_n:\Co^s\to\Co^{s+1}$.
\end{rem}

\begin{proof}[Proof of Theorem~\ref{Thm::LocalHT}]
    Recall the notations $\Bc_qg(z)=\int B_{q-1}(z,\cdot)\wedge g$, $\Fc g(z)=\int F(z,\cdot)\wedge g$ and $\Kc_qg(z)=\int K_{q-2}(z,\cdot)\wedge g$ from \eqref{Eqn::Space::DefBOp}, \eqref{Eqn::WeiEst::DefFOp} and \eqref{Eqn::WeiEst::DefKOp}. We also use notation of mixed degree $\Kc=\sum_{q=2}^{n}\Kc_q$ and $\Bc=\sum_{q=1}^n\Bc_q$. For convenience of analysis, for a smooth form $\psi$ we define operator $\M^\psi f:=\psi\wedge f$. Therefore $[\dbar,\M^\psi]=\M^{\dbar\psi}$.
    In this way, 
    \begin{align}\label{Eqn::LocalHT::CompFormulaP'}
        \Pc'=&\sum_{j=1}^{2n}\Fc^j\M^\chi\Sc^{0,j}[\dbar,\Ec];
        \\\label{Eqn::LocalHT::CompFormulaH'}
        \Hc'=&\Bc\M^\chi\Ec+\sum_{j=1}^{2n}\Kc^j\M^\chi\Sc^{0,j}[\dbar,\Ec];
        \\\label{Eqn::LocalHT::CompFormulaR}
        \Rc=&\Bc\M^{\dbar \chi}\Ec+\Bc\M^\chi\Sc^{0,0}[\dbar,\Ec]-\sum_{j=1}^{2n}\Kc^j\M^{\dbar\chi}\Sc^{0,j}[\dbar,\Ec].
    \end{align}

    In the following we work on the forms $f\in\Ss'(\omega;\wedge^{0,\bullet})=\bigoplus_{q=0}^n\Ss'(\omega;\wedge^{0,q})$ with mixed degrees. For convenience we will omit the notation $\wedge^{0,\bullet}$ in the function spaces. If we write $f=\sum_{q=0}^nf_q$ where $f_q$ is the part with degree $(0,q)$, then $\Pc' f=\Pc'f_0$, $\Hc f=\sum_q\Hc'_qf_q$ and $\Rc f=\sum_q\Rc_qf_q$ (see Convention~\ref{Conv::Intro::MixForm}).
    

    Firstly we show how $\Pc'$ and $\Hc'$ are defined for all extendable distributions on $\omega$. Let $f\in\Ss'(\omega;\wedge^{0,\bullet})$ and let $\psi\in\{\chi,\dbar\chi\}$. By Proposition~\ref{Prop::Space::ExtOp}, $\Ec f\in\Ss '(\C^n)$ is a well-defined distribution, so $\psi\wedge \Ec f\in\Es'(\B^{2n})$ is a compactly supported distribution. By Lemma~\ref{Lem::Space::BMFormula}, we get
    \begin{equation}\label{Eqn::LocalHT::BfDefined}
        \Bc[\psi\wedge\Ec f]\in\Ss'(\B^{2n};\wedge^{0,\bullet})\qquad\text{for }\psi\in\{\chi,\dbar\chi\}\text{ and for all }f\in\Ss'(\omega).
    \end{equation}
    Moreover we have the following estimates: for every $p,r\in(0,\infty]$ and $s\in\R$,
    \begin{align}\label{Eqn::LocalHT::BEBdd1}
        \Bc\M^\psi\Ec:\Fs_{pr}^s(\omega)\xrightarrow[\eqref{Eqn::Space::EBdd}]{\Ec}\Fs_{pr}^s(\C^n)\xrightarrow[\eqref{Eqn::Space::MultBdd}]{\M^\psi}\Fs_{pr}^s(\B^{2n})\xrightarrow[\eqref{Eqn::Space::BMKernelBdd}]{\Bc}\Fs_{pr}^{s+1}(\B^{2n})\twoheadrightarrow\Fs_{pr}^{s+1}(\tfrac12\B^{2n}\cap\omega).
    \end{align}

    Using Lemma~\ref{Lem::Space::TLEmbed} we have $\Fs_{p\infty}^{s+1}(\tfrac12\B^{2n}\cap\omega)\hookrightarrow\Fs_{\frac{(2n+2)p}{2n+2-p},\eps}^{s+\frac1{n+1}}(\tfrac12\B^{2n}\cap\omega)$, which gives
    \begin{equation}\label{Eqn::LocalHT::BEBdd2}
        \Bc\M^\chi\Ec:\Fs_{p\infty}^s(\omega)\to\Fs_{\frac{(2n+2)p}{2n+2-p},\eps}^{s+\frac1{n+1}}(\tfrac12\B^{2n}\cap\omega)\hookrightarrow\Fs_{\frac{(2n+2)p}{2n+2-p},\eps}^{s}(\tfrac12\B^{2n}\cap\omega),\quad p\in(0,2n+2].
    \end{equation}
    
    Note that $([\dbar,\Ec]f)|_{\omega}=0$. Since every tempered distribution has finite order, there is an integer $m\ge0$ such that $(\Sc^{m,j}[\dbar,\Ec]f)|_{\B^{2n}}\in L^1$ for $1\le j\le 2n$. By Theorem~\ref{Thm::RefAtD}~\ref{Item::RefAtD::Sum}, for every $m\ge1$, we have $\Sc^{0,j}=D_{v_j}^m\Sc^{m,j}$ with $(\Sc^{m,j}[\dbar,\Ec]f)|_\omega=0$. Therefore taking integration by parts, for $\psi\in\{\chi,\dbar\chi\}$ and $z\in\B^{2n}\cap\omega$, and for $m$ large enough such that $\Sc^{m,j}[\dbar,\Ec]f\in L^1$
    \begin{align}
        \notag
        \big(\Kc^j\M^\psi\Sc^{0,j}[\dbar,\Ec]f\big)(z)&=\int_{\B^{2n}\backslash\overline\omega}K^j(z,\cdot)\wedge \psi\wedge\big(\Sc^{0,j}[\dbar,\Ec]f\big)=\int_{\B^{2n}\backslash\overline\omega}K^j(z,\cdot)\wedge \psi\wedge D_{v_j}^m\big(\Sc^{m,j}[\dbar,\Ec]f\big)
        \\\notag
        &=(-1)^m\int_{\B^{2n}\backslash\overline\omega}D_{v_j}^m(K^j(z,\cdot)\wedge\psi)\wedge\Sc^{m,j}[\dbar,\Ec]f
        \\
        \label{Eqn::LocalHT::RevIntFormula}
        &=(-1)^m\sum_{k=0}^m{m\choose k}\int_{\B^{2n}\backslash\overline\omega}(D_{v_j}^kK^j(z,\cdot))\wedge (D_{v_j}^{m-k}\psi)\wedge\Sc^{m,j}[\dbar,\Ec]f
        \\
        \label{Eqn::LocalHT::RevIntKFormula}
        &=(-1)^m\sum_{k=0}^m{m\choose k}\Kc^{k,\rho_j}\M^{D_{v_j}^{m-k}\psi}\Sc^{m,j}[\dbar,\Ec]f(z).
    \end{align}
    Here we use $\Kc^{m,\rho_j}g(z)=\int D_{v_j}^mK^{\rho_j}(z,\cdot)\wedge g$ following the notation in \eqref{Eqn::WeiEst::OpF}.
    
    Since $D_{v_j}\rho_j\equiv1$ we see that $D_{v_j}^kK^j(z,\cdot)\in L^\infty(\B^{2n}\backslash\overline\omega)$ for all $z\in\B^{2n}\cap\omega$. Therefore the integrand in \eqref{Eqn::LocalHT::RevIntFormula} is $L^1$, which gives the definedness of $\Kc^j\M^\psi\Sc^{0,j}[\dbar,\Ec] f$ for $f\in\Ss'$ and $1\le j\le 2n$.
    
    More specifically, that $(\Sc^{m,j}[\dbar,\Ec]f)|_{\B^{2n}}\in L^1$ and $(\Sc^{m,j}[\dbar,\Ec]f)|_\omega=0$ implies $\M^{D_{v_j}^{m-k}\psi}\Sc^{m,j}[\dbar,\Ec]f\in L^1(\B^{2n}\backslash\overline\omega)$. Applying \eqref{Eqn::WeiEst::SobEstK1} we see that $\Kc^{k,\rho_j}\M^{D_{v_j}^{m-k}\psi}\Sc^{m,j}[\dbar,\Ec]f\in L^1(\B^{2n}\cap\omega,\delta^{m-1/2-s})$. By \eqref{Eqn::Space::HLSob} it belongs to $\Fs_{1\eps}^{s+1/2}(\B^{2n}\cap\omega)\subset\Ss'(\B^{2n}\cap\omega)$. Taking sum over $0\le k\le m$ we conclude that $\Kc^j\M^\psi\Sc^{0,j}[\dbar,\Ec]f\in\Ss'(\B^{2n}\cap\omega)$.
    
    Combining with \eqref{Eqn::LocalHT::BfDefined}, we conclude that for every $f\in\Ss'(\omega)$, $\Hc'f$ and $\Rc f$ are defined in $\Ss'(\B^{2n}\cap\omega)$. Taking restrictions, they are in $\Ss'(\frac12\B^{2n}\cap\omega)$ as well.

    Replacing the $\Kc^j$ in the above argument by $\Fc^j$ and replacing \eqref{Eqn::WeiEst::SobEstK1} by \eqref{Eqn::WeiEst::SobEstF}, note that $D_{v_j}^kF^j(z,\cdot)\in L^1(\B^{2n}\backslash\overline\omega)$ as well for $z\in\B^{2n}\cap\omega$, we see that $\Pc'f$ is well-defined in $\Ss'(\B^{2n}\cap\omega)$ for $f\in\Ss'(\omega)$. Moreover using the notation \eqref{Eqn::WeiEst::OpF},
    \begin{equation}\label{Eqn::LocalHT::RevIntFFormula}
        \Fc^j\M^\chi\Sc^{0,j}[\dbar,\Ec]f=(-1)^m\sum_{k=0}^m{m\choose k}\Fc^{k,\rho_j}\M^{D_{v_j}^{m-k}\chi}\Sc^{m,j}[\dbar,\Ec] f,\quad\text{for every }m\ge0.
    \end{equation}

    Altogether we get the well-definedness of $\Pc f,\Hc'f,\Rc f\in\Ss'(\frac12\B^{2n}\cap\omega;\wedge^{0,\bullet})$ for all $f\in\Ss'(\omega;\wedge^{0,\bullet})$.

    Applying Corollary~\ref{Cor::WeiEst::WeiOp} and Proposition~\ref{Prop::Space::HLLem} where we use $\delta(z)=\dist(z,b\omega)$, for a given $s\in\R$, take integers $m>\frac32-s$, for every $0\le k\le m$ and for $\phi\in\{D_{v_j}^{m-k}\chi,\dbar D_{v_j}^{m-k}\chi\}$,
    \begin{multline*}
        \Kc^{k,\rho_j}\M^\phi\Sc^{m,j}[\dbar,\Ec]:\Fs_{p\infty}^s(\omega)\xrightarrow[\eqref{Eqn::Space::[D,E]Bdd}]{[\dbar,\Ec]}\widetilde\Fs_{p\infty}^{s-1}(\omega^c)\xrightarrow[\eqref{Eqn::RefAtD::ScBdd}]{\Sc^{m,j}}\widetilde\Fs_{p\infty}^{s+m-1}(\omega^c)
        \\
        \xrightarrow[\eqref{Eqn::Space::MultBdd}]{\M^\phi}\Fs_{p\infty}^{s+m-1}(\overline\B^{2n}\backslash\overline\omega)\begin{cases}
            \xrightarrow[\eqref{Eqn::WeiEst::SobBddK1}]{\Kc^{k,\rho_j}}\Fs_{p\eps}^{s+\frac12}(\B^{2n}\cap\omega),&p\in[1,\infty]
            \\
            \xrightarrow[\eqref{Eqn::WeiEst::SobBddK2}]{\Kc^{k,\rho_j}}\Fs_{\frac{(2n+2)p}{2n+2-p},\eps}^s(\B^{2n}\cap\omega),&p\in[1,2n+2].
        \end{cases}
    \end{multline*}

    The same argument but replace $\Kc^{k,\rho_j}$ and \eqref{Eqn::WeiEst::SobBddK1} by $\Fc^{k,\rho_j}$ and \eqref{Eqn::WeiEst::SobBddF} yields
    \begin{equation*}
        \Fc^{k,\rho_j}\M^{D_{v_j}^{m-k}\chi}\Sc^{m,j}[\dbar,\Ec]:\Fs_{p\infty}^s(\omega)\to\Fs_{p\eps}^s(\B^{2n}\cap\omega),\qquad p\in[1,\infty],\qquad 0\le k\le m.
    \end{equation*}

    Since $\Sc^{0,0}:\widetilde\Fs_{p\infty}^s(\omega^c)\to\widetilde\Fs_{p\infty}^{s+m}(\omega^c)$ for all $m\ge0$, take any $m\ge2$ we are enough to conclude
    \begin{gather}\label{Eqn::LocalHT::S00Bdd}
        \Bc\M^\chi\Sc^{0,0}[\dbar,\Ec]:\Fs_{p\infty}^s(\omega)\to\Fs_{p\eps}^{s+1}(\B^{2n}\cap\omega)\twoheadrightarrow\Fs_{p\eps}^{s+\frac12}\cap\Fs_{\frac{(2n+2)p}{\max(2n+2-p,0)},\eps}^{s}(\tfrac12\B^{2n}\cap\omega),\qquad p\in[1,\infty].
    \end{gather}
    
    Summing over $k$, by \eqref{Eqn::LocalHT::RevIntKFormula} and \eqref{Eqn::LocalHT::RevIntFFormula}, we have, now for every $1\le j\le 2n$ and for $\psi\in\{\chi,\dbar\chi\}$,
    \begin{align*}
        \Kc^j\M^\psi\Sc^{0,j}[\dbar,\Ec]&:\Fs_{p\infty}^s(\omega)\to\Fs_{p\eps}^{s+\frac12}(\B^{2n}\cap\omega)\twoheadrightarrow\Fs_{p\eps}^{s+\frac12}(\tfrac12\B^{2n}\cap\omega),&p\in[1,\infty];
        \\
        \Kc^j\M^\psi\Sc^{0,j}[\dbar,\Ec]&:\Fs_{p\infty}^s(\omega)\to\Fs_{\frac{(2n+2)p}{2n+2-p},\eps}^{s}(\B^{2n}\cap\omega)\twoheadrightarrow\Fs_{\frac{(2n+2)p}{2n+2-p},\eps}^{s}(\tfrac12\B^{2n}\cap\omega),&p\in[1,2n+2];
        \\
        \Fc^j\M^\chi\Sc^{0,j}[\dbar,\Ec]&:\Fs_{p\infty}^s(\omega)\to\Fs_{p\eps}^s(\B^{2n}\cap\omega)\twoheadrightarrow\Fs_{p\eps}^s(\tfrac12\B^{2n}\cap\omega),&p\in[1,\infty].
    \end{align*}

    Taking sum over $1\le j\le 2n$, using \eqref{Eqn::LocalHT::CompFormulaP'} we get \eqref{Eqn::LocalHT::BddP'}. Applying \eqref{Eqn::LocalHT::BEBdd1} to \eqref{Eqn::LocalHT::CompFormulaH'} we get \eqref{Eqn::LocalHT::BddH'1};  
    applying \eqref{Eqn::LocalHT::BEBdd2} to \eqref{Eqn::LocalHT::CompFormulaH'} we get \eqref{Eqn::LocalHT::BddH'2}; 
    applying \eqref{Eqn::LocalHT::BEBdd1} and \eqref{Eqn::LocalHT::S00Bdd} to \eqref{Eqn::LocalHT::CompFormulaR} we get \eqref{Eqn::LocalHT::BddR}.
    
    When $q=n$, $[\dbar,\Ec]$ vanishes on $(0,n)$ forms. Therefore $\Hc'_n=\Bc_n\M^\chi\Ec$ and thus by \eqref{Eqn::LocalHT::BEBdd1} we get \eqref{Eqn::LocalHT::BddH'n}. 
    

    Note that by Lemma~\ref{Lem::Space::TLEmbed}, $\Fs_{p\eps}^s\hookrightarrow\Fs_{pr}^s\hookrightarrow\Fs_{p\infty}^s$ whenever $\eps<r$; and by Lemma~\ref{Lem::Space::SpaceLem}~\ref{Item::Space::SpaceLem::Hsp} $H^{s,2}=\Fs_{22}^s$. Therefore by \eqref{Eqn::LocalHT::BddR} $\Rc_q,\Hc'_q:H^{s,2}\hookrightarrow\Fs_{2\infty}^s\to \Fs_{2\eps}^{s+1/2}\hookrightarrow H^{s+1/2,2}$ are bounded, completing the proof of \ref{Item::LocalHT::Bdd}.

    \medskip
    Recall from Lemmas~\ref{Lem::Space::BMFormula} and \ref{Lem::WeiEst::LPFormula} that $\id=\dbar\Bc+\Bc\dbar$ and $\Bc-\Fc=\dbar\Kc-\Kc\dbar$. For every $f\in\Ss'(\omega;\wedge^{0,\bullet})$ we have $f|_{\frac12\B^{2n}\cap\omega}=(\chi f)|_{\frac12\B^{2n}\cap\omega}$ and $(\dbar\chi\wedge f)|_{\frac12\B^{2n}\cap\omega}=0$, therefore in $\frac12\B^{2n}\cap\omega$,
    \begin{align}
        \notag 
        f&=\chi f=\chi\Ec f=\M^\chi\Ec f=\dbar\Bc\M^\chi\Ec f+\Bc\dbar\M^\chi\Ec f=\dbar(\Bc\M^\chi\Ec f)+\Bc\big(\M^{\dbar\chi}\Ec+\M^\chi[\dbar,\Ec]+\M^\chi\Ec \dbar\big)f
        \\
        \label{Eqn::LocalHT::HTPfB}&=\dbar(\Bc\M^\chi\Ec f)+\Bc\M^\chi\Ec (\dbar f)+\Bc\M^{\dbar\chi}\Ec f+\sum_{j=0}^{2n}\Bc\M^\chi\Sc^{0,j}[\dbar,\Ec] f.
    \end{align}

    To prove \ref{Item::LocalHT::HT}, note that $F^j(z,\zeta)$ is holomorphic in $z$, so $\dbar \Fc^j\equiv0$ hence $\dbar\Fc^{m,\rho_j}=(-1)^m\dbar\Fc^j\circ D_{v_j}^m=0$ as well. Therefore by \eqref{Eqn::LocalHT::RevIntFFormula}, $\dbar\Pc' f=0$ as well for all $f\in\Ss'(\omega)$.
    
    Recall from Lemma~\ref{Lem::WeiEst::LPFormula} that $\Bc-\Fc^j=\dbar\Kc^j-\Kc^j\dbar$ where both sides map functions from $\B^{2n}\backslash\overline\omega$ to $\B^{2n}\cap\omega$. Recall that $\Sc^{0,j}$ are convolution operators which means $\dbar\Sc^{0,j}=\Sc^{0,j}\dbar$. Therefore for $1\le j\le 2n$,
    \begin{align}
        \notag&
        \Bc\M^\chi\Sc^{0,j}[\dbar,\Ec]=\Fc^j\M^\chi\Sc^{0,j}[\dbar,\Ec]+(\dbar\Kc^j-\Kc^j\dbar)\M^\chi\Sc^{0,j}[\dbar,\Ec] 
        \\
        \notag=&\Fc^j\M^\chi\Sc^{0,j}[\dbar,\Ec]+\dbar(\Kc^j\M^\chi\Sc^{0,j}[\dbar,\Ec] )-\Kc^j\M^{\dbar\chi}\Sc^{0,j}[\dbar,\Ec]+\Kc^j\M^\chi\Sc^{0,j}[\dbar,\Ec]\dbar
        \\
        \label{Eqn::LocalHT::HTPfK}=&\Fc^j\M^\chi\Sc^{0,j}[\dbar,\Ec]+\dbar(\Kc^j\M^\chi\Sc^{0,j}[\dbar,\Ec])+(\Kc^j\M^\chi\Sc^{0,j}[\dbar,\Ec]) \dbar-\Kc^j\M^{\dbar\chi}\Sc^{0,j}[\dbar,\Ec].
    \end{align}

    Taking sum over $1\le j\le 2n$ for \eqref{Eqn::LocalHT::HTPfK}, and combining \eqref{Eqn::LocalHT::HTPfB} we get
    \begin{equation*}
        f|_{\frac12\B^{2n}\cap\omega}=\Pc'f+\dbar\Hc'f+\Hc'\dbar f+\Rc f.
    \end{equation*}

    Recall that we already have $\dbar\Pc'f=0$. By comparing homogeneous components we obtain \eqref{Eqn:LocalHT::HTP} and \eqref{Eqn:LocalHT::HTH}. This completes the proof of \ref{Item::LocalHT::HT}.

    \medskip
    Finally assume $\lambda_1,\lambda_2\in\Co_c^\infty(\tfrac12\B^{2n})$ have disjoint supports. We pick a $\theta\in C_c^\infty(\B^{2n})$ such that $$\supp\theta\subset\{z:\dist(z,b\omega)<\tfrac{1-\|\nabla\sigma\|_{L^\infty}}2\dist(\supp\lambda_1,\supp\lambda_2)\}\quad\text{and }\theta\equiv1\text{ in a neighborhood of }b\omega\cap\supp\lambda_2.$$

    Let $f\in\Ss'(\omega)$. By construction of $\theta$ we have $(1-\theta)\lambda_2f\in\Es'(\omega\cap\frac12\B^{2n})$ as it vanishes near the boundary. Therefore by Lemma~\ref{Lem::Space::DisjointSuppSmooth} $[f\mapsto[\dbar,\Ec]((1-\theta)\lambda_2 f)]:\Ss'(\omega)\to\widetilde\Co^\infty(\omega^c)$, hence by compositions
    \begin{equation}\label{Eqn::LocalHT::DisjointSuppSmooth}
        \lambda_1\Pc'((1-\theta)\lambda_2f)=\sum_{j=1}^{2n}\M^{\lambda_1}\Fc^j\M^\chi\Sc^{0,j}\big([\dbar,\Ec]\M^{(1-\theta)\lambda_2}f\big)\in\Co^\infty(\omega).
    \end{equation}
    
    By Lemma~\ref{Lem::Space::SsLim}~\ref{Item::Space::SsLim::Cup}, there is a $s<0$ such that $\lambda_2f\in\Co^{s}(\tfrac12\B^{2n}\cap\omega)$.     
    Take $m>1-s$, by \eqref{Eqn::LocalHT::SuppFact} we get for all $0\le k\le m$,
    \begin{equation*}
        \supp\big((D^{m-k}_{v_j}\chi)\cdot\Sc^{m,j}[\dbar,\Ec](\theta\lambda_2f)\big)\subseteq(\supp(\theta\lambda_2)+(-\Kb))\cap\B^{2n}\backslash\omega.
    \end{equation*}

    Let us denote two compact sets $\Lambda_1:=\overline\omega\cap\supp\lambda_1$ and $\Lambda_2:=(\supp(\theta\lambda_2)+(\overline{-\Kb}))\cap\overline{\B^{2n}}\backslash\omega$ for convenience. Recall that $\supp\chi\subset\B^{2n}$, thus $D^{m-k}_{v_j}\chi\cdot\Sc^{m,j}[\dbar,\Ec](\theta\lambda_2f)\in L^1(\Lambda_2)$. By construction we see that $\Lambda_1\cap\Lambda_2=\varnothing$. In other words $\dist(\Lambda_1,\Lambda_2)>0$.

    Applying Lemma~\ref{Lem::WeiEst::DisjointSupp} with $\lambda=\lambda_1$ and $g=D^{m-k}_{v_j}\chi\cdot\Sc^{m,j}[\dbar,\Ec](\lambda_2f)$ we get $$\lambda_1\cdot \Fc^{k,\rho_j}\M^{D_{v_j}^{m-k}\chi}\Sc^{m,j}[\dbar,\Ec](\theta\lambda_2 f)\in\Co^\infty(\tfrac12\B^{2n}\cap\omega),\quad\text{ for all }0\le k\le m\text{ and }1\le j\le 2n.$$ Taking sum over $k$ via \eqref{Eqn::LocalHT::RevIntFFormula} we get $\lambda_1\cdot\Fc^j\M^\chi\Sc^{0,j}[\dbar,\Ec](\theta\lambda_2f)\in\Co^\infty(\frac12\B^{2n}\cap\omega)$. Taking sum over $j$ via \eqref{Eqn::LocalHT::CompFormulaP'} we get $\lambda_1\cdot\Pc'(\theta\lambda_2f)\in\Co^\infty(\frac12\B^{2n}\cap\omega)$. Since $\supp\lambda_1\subset\frac12\B^{2n}$, combining this estimate with \eqref{Eqn::LocalHT::DisjointSuppSmooth} we get $\lambda_1\cdot\Pc'(\lambda_2f)\in\Co^\infty(\omega)$. By keeping track of the compositions we see that $[f\mapsto\lambda_1\Pc'(\lambda_2f)]:\Ss'\to\Co^\infty$ is continuous, completing the proof of \ref{Item::LocalHT::HoloP'}.
\end{proof}

\section{Proof of the Theorem}\label{Section::ProofThm}
In this section we state a more general result than Theorem~\ref{Thm::MainThm} via Triebel-Lizorkin spaces. This result unify the Sobolev estimates and H\"older estimates, and include the case for $p=1$.

\begin{thm}\label{Thm::GenThm}
    Let $\Omega\subset\C^n$ be a bounded $C^{1,1}$ domain, which is locally strongly $\C$-linearly convexifiable.
    
    Then there are operators $\Hc_q:\Ss'(\Omega;\wedge^{0,q})\to\Ss'(\Omega;\wedge^{0,q-1})$ that map $(0,q)$-forms to $(0,q-1)$-forms with distributional coefficients for $1\le q\le n$ (we set $\Hc_{n+1}:=0$) and $\Pc:\Ss'(\Omega)\to\Ss'(\Omega)$ on functions, such that
\begin{enumerate}[(i)]
    \item{\normalfont(Homotopy formula)}\label{Item::GenThm::HT} $f=\dbar\Hc_q f+\Hc_{q+1}\dbar f$ for all $1\le q\le n$ and $(0,q)$-forms $f\in\Ss'(\Omega;\wedge^{0,q})$; and $f=\Pc f+\Hc_1\dbar f$ for all functions $f\in\Ss'(\Omega)$.
    \item{\normalfont(Skew Bergman projection estimates)}\label{Item::GenThm::P} For every $s\in\R$, $p\in[1,\infty]$ and $\eps>0$,  $\Pc:\Fs_{p\infty}^s(\Omega)\to \Fs_{p\eps}^s(\Omega)$.
    \item{\normalfont($\frac12$ estimates)}\label{Item::GenThm::Sob} For every $s\in\R$, $p\in[1,\infty]$ and $\eps>0$, $\Hc_q:\Fs_{p\infty}^s(\Omega;\wedge^{0,q})\to \Fs_{p\eps}^{s+\frac12}(\Omega;\wedge^{0,q-1})$.
    \item{\normalfont($L^p$-$L^q$ estimates)}\label{Item::GenThm::Lp} For $s\in\R$, $p\in[1,2n+2]$ and $\eps>0$, $\Hc_q:\Fs_{p\infty}^s(\Omega;\wedge^{0,q})\to \Fs_{\frac{2(n+1)p}{2n+2-p},\eps}^s(\Omega;\wedge^{0,q-1})$.

    \item\label{Item::GenThm::Hn} For $q=n$, additionally $\Hc_n:\Fs_{pr}^s(\Omega;\wedge^{0,n})\to \Fs_{pr}^{s+1}(\Omega;\wedge^{0,n-1})$ for all $s\in\R$ and $p,r\in(0,\infty]$.
\end{enumerate}
\end{thm}
\begin{rem}
    Theorem~\ref{Thm::GenThm} implies Theorem~\ref{Thm::MainThm} by Lemmas~\ref{Lem::Space::TLEmbed} and \ref{Lem::Space::SpaceLem}. 
    
    Indeed, for $s\in\R$ and $1<p<\infty$ we have $H^{s,p}(\Omega)=\Fs_{p2}^s(\Omega)$ and $\Co^s(\Omega)=\Fs_{\infty\infty}^s(\Omega)$. Therefore by \ref{Item::GenThm::Sob} with $\eps<2$ we have 
    \begin{align*}
        \Hc_q:&H^{s,p}=\Fs_{p2}^s\subset\Fs_{p\infty}^s\to\Fs_{p\eps}^{s+1/2}\subset\Fs_{p2}^{s+1/2}=H^{s+1/2,p};
        \\
        \Hc_q:&\Co^s=\Fs_{\infty\infty}^s\to\Fs_{\infty\eps}^{s+1/2}\subset\Fs_{\infty\infty}^{s+1/2}=\Co^{s+1/2}.
    \end{align*}
    The same argument works for $L^p$-$L^q$ estimates and for skew Bergman projections $\Pc$ as well, we omit the details.
\end{rem}
\begin{proof}[Proof of Theorem~\ref{Thm::GenThm}]
    The result follows from combining Proposition~\ref{Prop::PM::ConcretePM} and Theorem~\ref{Thm::LocalHT}.


    By Definition~\ref{Defn::WeiEst::Convex} and Lemma~\ref{Lem::LocalHT::LocalDomain}, for every $\zeta\in b\Omega$ there are a special Lipschitz domain $\omega_\zeta=\omega\subset\C^n$ that satisfies Assumption~\ref{Assumption::DomainRotate}, a  map $\Phi:\tfrac12\B^{2n}\to\C^n$ which is biholomorphic onto its image such that $\Phi(0)=\zeta$ and $\omega\cap\frac12\B^{2n}=\Phi^{-1}(\Omega)$

    By Theorem~\ref{Thm::LocalHT}, since $\omega$ satisfies Assumption~\ref{Assumption::DomainRotate}, there are operators $\Pc^\zeta=\Pc':\Ss'(\omega)\to\Ss'(\frac12\B^{2n}\cap\omega)$, $\Hc^\zeta_q=\Hc'_q:\Ss'(\omega;\wedge^{0,q})\to\Ss'(\frac12\B^{2n}\cap\omega;\wedge^{0,q-1})$ for $1\le q\le n$ and $\Rc^\zeta_q=\Rc_q:\Ss'(\omega;\wedge^{0,q})\to\Ss'(\frac12\B^{2n}\cap\omega;\wedge^{0,q})$ for $0\le q\le n-1$, such that (we set $\Hc^\zeta_{n+1}=\Rc^\zeta_n=0$),
    \begin{enumerate}[a)]
        \item\label{Item::GenThm::HT0} For $q=0$, $f|_{\frac12\B^{2n}\cap\omega}=\Pc^\zeta f+\Hc^\zeta_1\dbar f+\Rc_0f$ and $\dbar\Pc^\zeta f=0$ for all $f\in\Ss'(\omega)$.
        \item\label{Item::GenThm::HTq} For $1\le q\le n$, $f|_{\frac12\B^{2n}\cap\omega}=\dbar\Hc^\zeta_q f+\Hc^\zeta_{q+1}\dbar f+\Rc_qf$ for all $f\in\Ss'(\omega;\wedge^{0,q})$.
        \item\label{Item::GenThm::RBdd} $\Rc_q^\zeta:\Fs_{pr}^s\to\Fs_{pr}^{s+\frac12}$, $\Rc_q^\zeta:H^{s,2}\to H^{s+\frac12,2}$ are bounded for all $0\le q\le n-1$, $p\in[1,\infty]$, $r\in(0,\infty]$, $s\in\R$.
        \item\label{Item::GenThm::P'Bdd} $\Pc^\zeta:\Fs_{p\infty}^s\to\Fs_{p\eps}^s$ is bounded for all $p\in[1,\infty]$, $\eps>0$, $s\in\R$.
        \item\label{Item::GenThm::H'SobBdd} $\Hc^\zeta_q:\Fs_{p\infty}^s\to\Fs_{p\eps}^{s+\frac12}$ are bounded for $1\le q\le n$, $p\in[1,\infty]$, $\eps>0$ and $s\in\R$. 
        
        \item\label{Item::GenThm::H'Cpt} In particular $\Hc^\zeta_q:\Fs_{p\infty}^s\to\Fs_{p\eps}^s\subset\Fs_{p\infty}^s$ and $\Hc^\zeta_q:H^{s,2}\to H^{s+\frac12,2}$ are bounded for $p\in[1,\infty]$ and $s\in\R$.
        \item\label{Item::GenThm::H'LpBdd} $\Hc^\zeta_q:\Fs_{p\infty}^s\to\Fs_{\frac{(2n+2)p}{2n+2-p},\eps}^s$ are bounded for all $1\le q\le n$, $p\in[1,2n+2]$, $\eps>0$ and $s\in\R$.
        \item\label{Item::GenThm::H'nBdd} $\Hc^\zeta_n:\Fs_{pr}^s\to\Fs_{pr}^{s+1}$ are bounded for all $p,r\in(0,\infty]$ and $s\in\R$.
        \item\label{Item::GenThm::PSupp}For every $\chi_1,\chi_2\in\Co_c^\infty(\frac12\B^{2n})$ that have disjoint supports, $[f\mapsto\chi_1\cdot\Pc^\zeta(\chi_2f)]:\Ss'(\omega)\to\Co^\infty(\omega)$.
    \end{enumerate}

    \noindent\ref{Item::GenThm::HT}: Now by \ref{Item::GenThm::HT0}, \ref{Item::GenThm::HTq}, \ref{Item::GenThm::RBdd}, \ref{Item::GenThm::H'SobBdd}, \ref{Item::GenThm::PSupp}, the assumptions of Proposition~\ref{Prop::PM::ConcretePM}~\ref{Item::PM::ConcretePM::LocalHT} - \ref{Item::PM::ConcretePM::DPBdd} are all satisfied, with $t_0=1/2$ in \ref{Item::PM::ConcretePM::SobBdd}. By Proposition~\ref{Prop::PM::ConcretePM}~\ref{Item::PM::ConcretePM::HT} we obtain the homotopy operators $\Pc=\Pc^\Omega:\Ss'(\Omega)\to\Ss'(\Omega)$ and $\Hc_q=\Hc_q^\Omega:\Ss'(\Omega;\wedge^{0,q})\to\Ss'(\Omega;\wedge^{0,q-1})$ for $1\le q\le n$ for Theorem~\ref{Thm::GenThm}~\ref{Item::GenThm::HT}.

    For the boundedness of $\Pc$ and $\Hc_q$ we apply Proposition~\ref{Prop::PM::ConcretePM}~\ref{Item::PM::ConcretePM::GlobalBdd} by verifying \ref{Item::PM::ConcretePM::GlobalBdd::R} - \ref{Item::PM::ConcretePM::GlobalBdd::Bddq}.

    \noindent\ref{Item::GenThm::P}: By \ref{Item::GenThm::RBdd}, \ref{Item::GenThm::P'Bdd} and \ref{Item::GenThm::H'Cpt}, $\Rc_0^\zeta,\Hc_1^\zeta:\Fs_{p\infty}^s\to\Fs_{p\infty}^s$ and $\Rc_0^\zeta,\Pc^\zeta,\Hc_1^\zeta:\Fs_{p\infty}^s\to\Fs_{p\eps}^s$ are bounded. By \ref{Item::PM::ConcretePM::GlobalBdd::R}, \ref{Item::PM::ConcretePM::GlobalBdd::H'Cpt} and \ref{Item::PM::ConcretePM::GlobalBdd::Bdd0} we get $\Pc=\Pc^\Omega:\Fs_{p\infty}^s\to\Fs_{p\eps}^s$.

    \smallskip\noindent\ref{Item::GenThm::Sob}: By \ref{Item::GenThm::RBdd}, \ref{Item::GenThm::H'SobBdd} and \ref{Item::GenThm::H'Cpt}, $\Rc_q^\zeta,\Hc_q^\zeta,\Hc_{q+1}^\zeta:\Fs_{p\infty}^s\to\Fs_{p\infty}^s$ and $\Hc_q^\zeta:\Fs_{p\infty}^s\to\Fs_{p\eps}^{s+1/2}$ are bounded. By \ref{Item::PM::ConcretePM::GlobalBdd::R}, \ref{Item::PM::ConcretePM::GlobalBdd::H'Cpt} and \ref{Item::PM::ConcretePM::GlobalBdd::Bddq} we get $\Hc_q=\Hc_q^\Omega:\Fs_{p\infty}^s\to\Fs_{p\eps}^{s+1/2}$.


    \smallskip\noindent\ref{Item::GenThm::Lp}: By \ref{Item::GenThm::RBdd}, \ref{Item::GenThm::H'Cpt} and \ref{Item::GenThm::H'LpBdd}, $\Rc_q^\zeta,\Hc_q^\zeta,\Hc_{q+1}^\zeta:\Fs_{p\infty}^s\to\Fs_{p\infty}^s$ and $\Hc_q^\zeta:\Fs_{p\infty}^s\to\Fs_{\frac{(2n+2)p}{2n+2-p},\eps}^s$ are bounded. By \ref{Item::PM::ConcretePM::GlobalBdd::R}, \ref{Item::PM::ConcretePM::GlobalBdd::H'Cpt} and \ref{Item::PM::ConcretePM::GlobalBdd::Bddq} we get $\Hc_q=\Hc_q^\Omega:\Fs_{p\infty}^s\to\Fs_{\frac{(2n+2)p}{2n+2-p},\eps}^s$.

    \smallskip\noindent\ref{Item::GenThm::Hn}: By \ref{Item::GenThm::RBdd}, \ref{Item::GenThm::H'Cpt} and \ref{Item::GenThm::H'nBdd}, $\Rc_n^\zeta,\Hc_n^\zeta:\Fs_{pr}^s\to\Fs_{pr}^s$ and $\Hc_n^\zeta:\Fs_{pr}^s\to\Fs_{pr}^{s+1}$ are bounded. By \ref{Item::PM::ConcretePM::GlobalBdd::R}, \ref{Item::PM::ConcretePM::GlobalBdd::H'Cpt} and \ref{Item::PM::ConcretePM::GlobalBdd::Bddq} we get $\Hc_n=\Hc_n^\Omega:\Fs_{pr}^s\to\Fs_{pr}^{s+1}$.
\end{proof}

\begin{rem}
    For the standard construction \eqref{Eqn::Intro::LKFormula}, which has been used in \cite{ShiYaoC2,ShiYaoCk}, the bound  $\Hc_n:\Fs_{pr}^s\to\Fs_{pr}^{s+1}$ follows directly from Lemma~\ref{Lem::Space::BMKernel}. 
    
    However, in our construction, this is not that direct. Note that the operator $\Hc_n$ is obtained from \eqref{Eqn::PM::AbsPM::PH}, which involves taking the inverse of an abstract spectral projection of $R_n$. However by \eqref{Eqn::LocalHT::H'}, the operator $R_n$ in \eqref{Eqn::PM::ConcretePM::Rq} is not zero.
\end{rem}

\appendix
\section{Review of Function Spaces and Integral Operators}\label{Section::Space}

In this section we review the materials of function spaces and certain integral operators. Most of the results below 
are mentioned and covered in \cite{YaoCXFinite}.

\begin{note}\label{Note::Space::Dist}
We denote by $\Ss'(\R^N)$ the \textit{space of tempered distributions}.

For an arbitrary open subset $U\subseteq\R^N$, we denote by $\Ds'(U)$ the \textit{space of distributions} in $U$. We denote by $\Ss'(U):=\{\tilde f|_U:\tilde f\in\Ss'(\R^N)\}\subsetneq\Ds'(U)$ the space of distributions which can be extended to tempered distributions in $\R^N$. We denote by $\Es'(U)$ the space of distributions which are compactly supported in $U$.
\end{note}
Note that $\Ds'(U),\Ss'(U),\Es'(U)$ are all topological vector spaces. See e.g. \cite[Chapter~6]{GrandpaRudin}. See also \cite[(3.1) and Proposition~3.1]{RychkovExtension} for an equivalent description of $\Ss'(U)$.


\begin{defn}[Weighted Sobolev]\label{Defn::Space::WeiSob}
Let $U\subseteq\R^N$ be an arbitrary open set. Let $\varphi:U\to[0,\infty)$ be a non-negative continuous function, we define for $k\ge0$ and $1\le p\le\infty$,
\begin{equation}\label{Eqn::Space::WeiSobo}
\begin{gathered}
W^{k,p}(U,\varphi):=\{f\in W^{k,p}_\loc(U):\|f\|_{W^{k,p}(U,\varphi)}<\infty\},
\\
\|f\|_{W^{k,p}(U,\varphi)}:=\bigg(\sum_{|\alpha|\le k}\int_U |\varphi D^\alpha f|^p\bigg)^\frac1p \quad 1 \le p < \infty;\quad
\|f\|_{W^{k,\infty}(U,\varphi)}:=\sup\limits_{|\alpha|\le k}\|\varphi D^\alpha f\|_{L^\infty(U)}.
\end{gathered}
\end{equation}
We define $W^{k,p}(U):=W^{k,p}(U,\1)$ where $\1(x)\equiv1$ is the constant function.
\end{defn}

For completeness we recall the definitions of negative Sobolev, H\"older and BMO spaces. These characterizations are not directly used in the paper.

\begin{defn}[Negative Sobolev]For $k\in\Z_+$, $1<p<\infty$ and $U\subseteq\R^N$, we define the negative Sobolev space $W^{-k,p}(U):=\{\sum_{|\alpha|\le k}D^\alpha g_\alpha:g_\alpha\in L^p(U)\}$ with norm 
    \begin{equation*}
        \|f\|_{W^{-k,p}(U)}=\inf\Big\{\sum_{|\alpha|\le k}\|g_\alpha\|_{L^p(U)}:f=\sum_{|\alpha|\le k}D^\alpha g_\alpha\text{ as distributions}\Big\}.
    \end{equation*}
\end{defn}
\begin{defn}[Sobolev-Bessel]\label{Defn::Space::Sob}
    Let $s\in\R$. For $1<p<\infty$, we define the Bessel potential space $H^{s,p}(\R^N)$ to be the set of all tempered distributions $f\in\Ss'(\R^N)$ such that
    \begin{equation}\label{Eqn::Space::HspDefn}
        \|f\|_{H^{s,p}(\R^N)}:=\|(I-\Delta)^\frac s2f\|_{L^p(\R^N)}<\infty.
    \end{equation}
    Here we use the standard (negative) Laplacian $\Delta=\sum_{j=1}^ND_{x_j}^2$.
    
    
    On an open subset $U\subseteq\R^N$, we define $H^{s,p}(U):=\{\tilde f|_U:\tilde f\in H^{s,p}(\R^N)\}$ for $s\in\R$, $1< p<\infty$ with norm $\|f\|_{H^{s,p}(U)}:=\inf_{\tilde f|_U=f}\|\tilde f\|_{H^{s,p}(\R^N)}$.

\end{defn}
\begin{defn}[H\"older-Zygmund]\label{Defn::Space::Hold}Let $U\subseteq\R^N$ be an open subset. We define the H\"older-Zygmund space $\Co^s(U)$ for $s\in\R$ by the following:
\begin{itemize}
    \item For $0<s<1$, $\Co^s(U)$ consists of all $f\in C^0(U)$ such that $\|f\|_{\Co^s(U)}:=\sup\limits_U|f|+\sup\limits_{x,y\in U}\frac{|f(x)-f(y)|}{|x-y|^s}<\infty$.
    \item $\Co^1(U)$ consists of all $f\in C^0(U)$ such that $\|f\|_{\Co^1(U)}:=\sup\limits_U|f|+\sup\limits_{x,y\in U;\frac{x+y}2\in U}\frac{|f(x)+f(y)-2f(\frac{x+y}2)|}{|x-y|}<\infty$.
    \item For $s>1$ recursively, $\Co^s(U)$ consists of all $f\in \Co^{s-1}(U)$ such that $\nabla f\in\Co^{s-1}(U;\C^N)$. We define $\|f\|_{\Co^s(U)}:=\|f\|_{\Co^{s-1}(U)}+\sum_{j=1}^N\|D_j f\|_{\Co^{s-1}(U)}$.
    \item For $s\le0$ recursively, $\Co^s(U)$ consists of all distributions that have the form $g_0+\sum_{j=1}^N\partial_jg_j$ where $g_0,\dots,g_N\in \Co^{s+1}(U)$. We define $\|f\|_{\Co^s(U)}:=\inf\{\sum_{j=0}^N\|g_j\|_{\Co^{s+1}(U)}:f=g_0+\sum_{j=1}^N\partial_j g_j\in\Ds'(U)\}$.
    \item We define by $\Co^\infty(U):=\bigcap_{s>0}\Co^s(U)$ the space of bounded smooth functions.
\end{itemize}
    
\end{defn}
\begin{defn}[BMO space]
    Let $U\subseteq\R^N$ be an open connected subset. Fix a $0<a\le1$ such that there exists a ball $B\subset U$ with volume $a$. We define $\BMO(U)$ and $\bmo(U)$ (see \cite[Definition 1.2]{ChangHardy}) to be the spaces consisting of $f\in L^1_\loc(U)$ such that: 
\begin{gather*}
    \|f\|_{\BMO(U)}:=\sup_{B\subseteq U}\frac1{|B|}\int_{B}\Big|f(x)-{\frac1{|B|}\int_{B}f}\Big|dx<\infty,\quad
    \|f\|_{\bmo(U)}:=\|f\|_{\BMO(U)}+\sup_{\substack{B\subseteq U\\|B|>a}}\frac{1}{|B|}\int_{B}|f|<\infty.
\end{gather*}
Here $B$ denotes the balls in $\R^N$.
\end{defn}
\begin{rem}
    Clearly $\bmo(U)\subset L^1_\loc(U)$ is embedding. When $U$ is a bounded connected Lipschitz domain $\BMO(U)=\bmo(U)/\{c\cdot\1_U:c\in\C\}$ ignores the constant functions. In some studies e.g. \cite[Section~4]{McNealStein} people do not distinguish $\BMO(U)$ and $\bmo(U)$, despite that $\|\cdot\|_{\BMO(U)}$ is only a semi-norm.
\end{rem}

\begin{defn}[Triebel-Lizorkin]\label{Defn::Space::TLSpace}
    Let $\phi=(\phi_j)_{j=0}^\infty$ be a sequence of Schwartz functions satisfying:
\begin{itemize}
    \item\label{Item::Space::LambdaFour} The Fourier transform $\hat\phi_0(\xi)=\int_{\R^N}\phi_0(x)e^{-2\pi ix\xi}dx$ satisfies $\supp\hat\phi_0\subset\{|\xi|<2\}$ and $\hat\phi_0|_{\{|\xi|<1\}}\equiv1$.
    \item\label{Item::Space::LambdaScal}  $\phi_j(x)=2^{jN}\phi_0(2^jx)-2^{(j-1)N}\phi_0(2^{j-1}x)$ for $j\ge1$.
\end{itemize}

Let $0< p,q\le\infty$ and $s\in\R$, we define the Triebel-Lizorkin norm $\|\cdot\|_{\Fs_{pq}^s(\phi)}$ as
\begin{align}
\label{Eqn::Space::TLNorm1}
    \|f\|_{\Fs_{pq}^s(\phi)}&:=\|(2^{js}\phi_j\ast f)_{j=0}^\infty\|_{L^p(\R^N;\ell^q(\N))}=\bigg(\int_{\R^N}\Big(\sum_{j=0}^\infty|2^{js}\phi_j\ast f(x)|^q\Big)^\frac pqdx\bigg)^\frac1p,&p<\infty;
    \\
    \label{Eqn::Space::TLNorm2}
    \|f\|_{\Fs_{\infty q}^s(\phi)}&:=\sup_{x\in\R^N,J\in\Z}2^{NJ\frac1q}\|(2^{js}\phi_j\ast f)_{j=\max(0,J)}^\infty\|_{L^q(B(x,2^{-J});\ell^q)},&p=\infty.
\end{align}
Here for $q=\infty$ we take the usual modifications: we replace the $\ell^q$ sum by the supremum over $j$.

We define $\Fs_{pq}^s(\R^N)$ with its norm given by a fixed choice of $\phi$.

For an arbitrary open subset $U\subseteq\R^N$, we define
\begin{align}
    \label{Eqn::Space::TLDomain}\Fs_{pq}^s(U)&:=\{\tilde f|_U:\tilde f\in\Fs_{pq}^s(\R^N)\}&\text{with}\quad \|f\|_{\Fs_{pq}^s(U)}:=\inf_{\tilde f|_U=f}\|\tilde f\|_{\Fs_{pq}^s(\R^N)};
    \\
    \label{Eqn::Space::TildeSpace}\widetilde\Fs_{pq}^s(\overline U)&:=\{f\in \Fs_{pq}^s(\R^N):f|_{\overline{U}^c}=0\}&\text{ as a closed subspace of }\Fs_{pq}^s(\R^N).
\end{align}
\end{defn}
\begin{rem}\label{Rmk::Space::TLRmk}
\begin{enumerate}[(i)]
    \item When $p$ or $q<1$, \eqref{Eqn::Space::TLNorm1} and \eqref{Eqn::Space::TLNorm2} are only quasi-norms. For convenience we still use the terminology ``norms'' to refer them.
    \item Different choices of $\phi$ result in equivalent norms. See \cite[Proposition~2.3.2 (ii)]{TriebelTheoryOfFunctionSpacesI} and \cite[Propositions 1.3 and 1.8]{TriebelTheoryOfFunctionSpacesIV}.
    
    
    \item When $p=q=\infty$, \eqref{Eqn::Space::TLNorm2} can be written as $\|f\|_{\Fs_{\infty\infty}^s(\phi)}=\sup_{j\ge0}\|2^{js}\phi_j\ast f\|_{L^\infty(\R^N)}$. It is more common to be referred as the \textit{Besov norm} $\Bs_{\infty\infty}^s$. See Definition~\ref{Defn::Hardy::Besov},  also \cite[(1.15)]{TriebelTheoryOfFunctionSpacesIV} and \cite[Remark~2.3.4/3]{TriebelTheoryOfFunctionSpacesI}.
    \item For an arbitrary open subset $U\subseteq\R^N$, we have $\Fs_{pq}^s(\overline{U}^c)=\Fs_{pq}^s(\R^N)/\widetilde\Fs_{pq}^s(\overline U)$, or equivalently \\$\Fs_{pq}^s(U)=\Fs_{pq}^s(\R^N)/\widetilde\Fs_{pq}^s(U^c)$. See also \cite[Remark~4.3.2/1]{TriebelInterpolation}.
\end{enumerate}
\end{rem}

\begin{lem}[Equivalent norms via derivatives]\label{Lem::Space::EqvNormDer}
    Let $\Omega\subset\R^N$ be a bounded Lipschitz domain. For $p,q\in(0,\infty]$, $s\in\R$ and $m\in\Z_+$, there is a $C>0$ such that 
    \begin{equation}\label{Eqn::Space::EqvNormDer}
        C^{-1}\|f\|_{\Fs_{pq}^s(\Omega)}\le\sum_{|\alpha|\le m}\|D^\alpha f\|_{\Fs_{pq}^{s-m}(\Omega)}\le C\|f\|_{\Fs_{pq}^s(\Omega)}.
    \end{equation}
    In particular $\nabla:\Fs_{pq}^s(\Omega)\to\Fs_{pq}^{s-1}(\Omega;\C^N)$ is bounded.
\end{lem}
\begin{proof}
    The case $p<\infty$ is given in \cite[Theorem~1.1]{ShiYaoExt}. The case $p=\infty$ is given in \cite[Theorem~2]{YaoExtensionMorrey}. 
\end{proof}
If we consider the case of total space $\Omega=\R^N$, the equivalent norm result \eqref{Eqn::Space::EqvNormDer} can be found in e.g. \cite{TriebelTheoryOfFunctionSpacesIV}.
\begin{lem}[Sobolev embedding]\label{Lem::Space::TLEmbed}
    Let $p_0,p_1,q_0,q_1\in(0,\infty]$ and $s_0,s_1\in\R$ be such that $s_0-s_1\ge N\max(0,\frac1{p_0}-\frac1{p_1})$, and either $q_0\le q_1$ or $s_0-s_1>0$. Let $\Omega\subset\R^N$ be a bounded open subset. We have embedding $\Fs_{p_0q_0}^{s_0}(\Omega)\hookrightarrow\Fs_{p_1q_1}^{s_1}(\Omega)$.

    If in addition $s_0-s_1> N\max(0,\frac1{p_0}-\frac1{p_1})$, then the embedding is compact.
\end{lem}
\begin{proof}
    See e.g. \cite[Remark~2.87]{TriebelTheoryOfFunctionSpacesIV}. In the reference, it is stated for smooth domains. For an arbitrary bounded open subset $\Omega$ we can take a smooth bounded open set $\tilde\Omega\supset\Omega$. Since $\|f\|_{\Fs_{pq}^s(\Omega)}=\inf_{\tilde f|_\Omega=f}\|\tilde f\|_{\Fs_{pq}^s(\tilde\Omega)}$, the embedding $\Fs_{p_0q_0}^{s_0}(\tilde\Omega)\hookrightarrow\Fs_{p_1q_1}^{s_1}(\tilde \Omega)$ implies $\Fs_{p_0q_0}^{s_0}(\Omega)\hookrightarrow\Fs_{p_1q_1}^{s_1}(\Omega)$.

    Now suppose $s_0-s_1> N\max(0,\frac1{p_0}-\frac1{p_1})$, let $(f_k)_{k=1}^\infty\subset\Fs_{p_0q_0}^{s_0}(\Omega)$ be a bounded sequence. We can take extensions $(\tilde f_k)_{k=1}^\infty\subset\Fs_{p_0q_0}^{s_0}(\tilde\Omega)$ which is also a bounded sequence. By the known compactness from the reference $(\tilde f_k)_k$ has convergent subsequence $(\widetilde f_{k_j})_{j=1}^\infty$ in $\Fs_{p_1q_1}^{s_1}(\tilde\Omega)$. Therefore, the subsequence $( f_{k_j})_{j=1}^\infty$ converges in $\Fs_{p_1q_1}^{s_1}(\Omega)$ to $(\lim_{j\to\infty}\tilde f_{k_j})|_\Omega$. Hence, the embedding on $\Omega$ is compact.
\end{proof}

\begin{lem}\label{Lem::Space::SpaceLem}
    Let $\Omega\subseteq\R^N$ be either a bounded Lipschitz domain or the total space $\R^N$.
    \begin{enumerate}[(i)]
        \item\label{Item::Space::SpaceLem::Hsp} $H^{s,p}(\Omega)=\Fs_{p2}^s(\Omega)$  for all $s\in\R$ and $1< p<\infty$;
        \item\label{Item::Space::SpaceLem::Wkp} $W^{k,p}(\Omega)=\Fs_{p2}^k(\Omega)$ for all $k\in\Z$ and $1<p<\infty$;
        \item\label{Item::Space::SpaceLem::Zygmund} $\Co^s(\Omega)=\Fs_{\infty\infty}^s(\Omega)$ for all $s\in\R$;
        \item\label{Item::Space::SpaceLem::BMO} $\bmo(\Omega)=\Fs_{\infty2}^0(\Omega)$.
    \end{enumerate}
\end{lem}
\begin{proof}
    When $\Omega=\R^N$ see \cite[Theorems 2.5.5, 2.5.7]{TriebelTheoryOfFunctionSpacesI} for \ref{Item::Space::SpaceLem::Hsp}, \ref{Item::Space::SpaceLem::Wkp} for $k\ge0$, \ref{Item::Space::SpaceLem::Zygmund} for $s>0$. See \cite[Proposition~1.13]{TriebelTheoryOfFunctionSpacesIV} for \ref{Item::Space::SpaceLem::BMO}. 


    When $\Omega$ is a bounded Lipschitz domain, \ref{Item::Space::SpaceLem::Hsp} follows from taking restrictions to domain via $H^{s,p}(\R^N)=\Fs_{p2}^s(\R^N)$, \ref{Item::Space::SpaceLem::Wkp} for $k\ge0$ 
    follow from \cite[Theorem~1.122]{TriebelTheoryOfFunctionSpacesIII}.

    \smallskip Here let us prove \ref{Item::Space::SpaceLem::Wkp} for $k\le0$. Let $l=-k$ for convenience. For every $g_\alpha\in L^p(\R^N)=\Fs_{p2}^0(\R^N)$, by e.g. \cite[Theorem~2.3.8]{TriebelTheoryOfFunctionSpacesI} we get $D^\alpha g_\alpha\in \Fs_{p2}^{-|\alpha|}(\R^N)$. This guarantees $W^{-l,p}(\Omega)\subseteq\Fs_{p2}^{-l}(\Omega)$ with the inclusion being embedding. Conversely for $f\in\Fs_{p2}^{-l}(\Omega)$ take an extension $\tilde f\in\Fs_{p2}^{-l}(\R^N)$ (where $f=\tilde f$ if $\Omega=\R^N$) and take an $m\ge l/2$, we have $ f=(I-\Delta)^m((I-\Delta)^{-m}\tilde f|_\Omega)$. By \cite[Theorem~2.3.8]{TriebelTheoryOfFunctionSpacesI} and the case $k\ge0$, $(I-\Delta)^{-m}\tilde f\in \Fs_{p2}^{2m-l}(\R^N)=W^{2m-l,p}(\R^N)$ thus $(I-\Delta)^{-m}\tilde f|_\Omega\in=W^{2m-l,p}(\Omega)$. By expanding $(I-\Delta)^m$ we can write $f=\sum_{|\beta|\le 2m}D^\beta h_\beta$ for some $h_\beta\in W^{2m-l,p}(\Omega)$. Since $D^\gamma h_\beta\in L^p$ for all $|\gamma|\le 2m-l$, by reorganizing the derivatives we conclude the embedding $\Fs_{p2}^{-l}(\Omega)\subseteq W^{-l,p}(\Omega)$, which completes the proof.

\smallskip
    For \ref{Item::Space::SpaceLem::Zygmund}, note that the $\|\cdot\|_{\Co^s}$ are defined intrinsically, not by extensions. The statement is slightly different from \cite[Theorem~1.122 (1.400)]{TriebelTheoryOfFunctionSpacesIII}.    
    
    For $0<s\le 1$ this can follow from \cite[Theorem~3.18]{DispaBesovNorm}, which says $\|f\|_{\Fs_{\infty\infty}^s(\Omega)}\approx\|f\|_{\Co^s(\Omega)}$ provided that either side is finite.    
    
    When $s>1$, the Definition~\ref{Defn::Space::Hold} we use is $\|f\|_{\Co^s(\Omega)}=\sum_{|\alpha|\le1}\|D^\alpha f\|_{\Co^{s-1}(\Omega)}$. On the other hand by Lemma~\ref{Lem::Space::EqvNormDer} we have $\|f\|_{\Fs_{\infty\infty}^s(\Omega)}\approx\sum_{|\alpha|\le1}\|D^\alpha f\|_{\Fs^{s-1}_{\infty\infty}(\Omega)}$. These two properties also guarantee $\|f\|_{\Co^s(\Omega)}\approx\|f\|_{\Fs_{\infty\infty}^s(\Omega)}$ for $s>1$.

    For other different versions of H\"older-Zygmund norms when $s>0$ see also \cite[Section~5]{GongHolderAq}.

    For $s\le0$, see e.g. \cite[Remark~4.6~(vi)]{YaoCXFinite} for a proof. The proof in the reference follows from the fact that $\Co^s(\R^N)=\Fs_{\infty\infty}^s(\R^N)$ and $\Co^t(\Omega)=\Fs_{\infty\infty}^t(\Omega)$ when $t>0$. The argument is identical to the proof of $W^{k,p}(\Omega)=\Fs_{p2}^k(\Omega)$ for $k\ge0$.

    \smallskip
    For \ref{Item::Space::SpaceLem::BMO} when $\Omega$ is bounded Lipschitz, by e.g. \cite[Theorem~1.4]{ChangHardy} (the construction is given in \cite[Section~3]{JonesBMO}) for every $f\in\bmo(\Omega)$ there is an extension $\tilde f\in\bmo(\R^N)$ such that $\|\tilde f\|_{\bmo(\R^N)}\lesssim\|f\|_{\bmo(\Omega)}$ where the implied constant is independent of $f$. In particular $\bmo(\Omega)=\{\tilde f|_\Omega:\tilde f\in\bmo(\R^N)\}$. Since we know $\bmo(\R^N)=\Fs_{\infty2}^0(\R^N)$, taking restrictions to $\Omega$ we get $\Fs_{\infty2}^0(\Omega)=\bmo(\Omega)$. 
\end{proof}
\begin{rem}\label{Rmk::Space::RmkHspHardy}
    Lemma~\ref{Lem::Space::SpaceLem}~\ref{Item::Space::SpaceLem::Hsp} extends to the case $0<p\le1$ with $H^{s,p}$ being the \textit{fractional Hardy space}. The definition of $H^{s,p}(\R^N)$ is to replace the $L^p$ norm in \eqref{Eqn::Space::HspDefn} by the \textit{local Hardy norm}: more precisely, for $0<p\le1$ and for $\phi_0\in\Ss(\R^N)$ be as in Definition~\ref{Defn::Space::TLSpace},
    \begin{equation}\label{Eqn::Space::DefnFractionalHardy}
        \|f\|_{H^{s,p}(\R^N)}:=\|(I-\Delta)^\frac s2f\|_{h^p(\R^N)},\qquad\text{where}\quad\|g\|_{h^p(\R^N)}:=\Big\|\sup_{0<\delta<1}|\delta^{-N}\phi_0(\cdot/\delta)\ast g|\Big\|_{L^p(\R^N)}.
    \end{equation}

    By \cite[Theorems~2.3.8 and~2.5.8/1]{TriebelTheoryOfFunctionSpacesI}, we have $H^{s,p}(\R^N)=\Fs_{p2}^s(\R^N)$ with equivalent quasi-norms. Moreover $h^p(\R^N)=L^p(\R^N)$ holds with equivalent norms for all $1<p<\infty$, see \cite[Theorems~2.5.6~(i) and~2.5.8/1]{TriebelTheoryOfFunctionSpacesI}.

    As a result, by using standard definition $H^{s,p}(U):=\{\tilde f|_U:\tilde f\in H^{s,p}(\R^N)\}$, we maintain the coincidence $H^{s,p}(U)=\Fs_{p2}^s(U)$.
\end{rem}

\begin{rem}
    For $s>0$, in \cite{DispaBesovNorm} the theorem only shows that $\|\cdot\|_{\Co^s(\Omega)}$ is an equivalent norm of $\Fs_{\infty\infty}^s(\Omega)$. Without further information given in the reference, it is possible that $\Fs_{\infty\infty}^s(\Omega)\subseteq\Co^s(\Omega)$ is a proper closed subspace. This can be the case if $\Omega$ is not a Lipschitz domain. Fortunately, on Lipschitz domains, every $\Co^s$ function admits a $\Co^s$ extension to $\R^N$. For a more general result on extension of Zygmund functions, see \cite[Theorem~18.5]{BesovBookVol2}, where the $l$-horn condition in the reference can be found in \cite[Section~8]{BesovBookVol1}.
\end{rem}

\begin{lem}\label{Lem::Space::SsLim}
    Let $1<p<\infty$ and let $\Omega\subset\R^N$ be a bounded Lipschitz domain.
    \begin{enumerate}[(i)]
        \item\label{Item::Space::SsLim::Cap} $\bigcap_{s>0}H^{s,p}(\Omega)=\bigcap_{k=1}^\infty W^{k,p}(\Omega)=\Co^\infty(\Omega)$;
        \item\label{Item::Space::SsLim::Cup} $\bigcup_{s>0}H^{-s,p}(\Omega)=\bigcup_{k=1}^\infty W^{-k,p}(\Omega)=\bigcup_{s>0}\Co^{-s}(\Omega)=\Ss'(\Omega)=\{\tilde f|_\Omega:\tilde f\in\Ds'(\R^N)\}$.
    \end{enumerate}
\end{lem}

Note that the results are false for general unbounded domains.
\begin{proof}
    Since $\Omega$ is a bounded set, take $\chi\in\Co_c^\infty(\R^N)$ be such that $\chi|_\Omega=1$, if $f=\tilde f|_\Omega$ for some $\tilde f\in\Ds'(\R^N)$, then $f=(\chi\tilde f)|_\Omega$ holds with $\chi\tilde f\in\Es'(\R^N)\subset\Ss'(\R^N)$. Therefore $\Ss'(\Omega)=\{\tilde f|_\Omega:\tilde f\in\Ds'(\R^N)\}$.

    See e.g. \cite[Lemma~4.15]{YaoCXFinite} for $\bigcup_{s>0}\Co^{-s}(\Omega)=\Ss'(\Omega)$. Recall by definition $\bigcap_{s>0}\Co^{s}(\Omega)=\Co^\infty(\Omega)$.

    We have embedding $\Fs_{\infty\infty}^{s+\delta}(\Omega)\subset\Fs_{p2}^s(\Omega)\subset\Fs_{\infty\infty}^{s-\frac Np}(\Omega)$ for every $s\in\R$, $\delta>0$ and $1<p<\infty$ (see e.g. \cite[Proposition~4.6]{TriebelTheoryOfFunctionSpacesIII}). By Lemma~\ref{Lem::Space::SpaceLem} this is $\Co^{k+\delta}(\Omega)\subset H^{k,p}(\Omega)=W^{k,p}(\Omega)\subset\Co^{k-\frac Np}(\Omega)$. Taking $k\to+\infty$ we get \ref{Item::Space::SsLim::Cap}, taking $k\to-\infty$ we get \ref{Item::Space::SsLim::Cup}.
\end{proof}

For $U\subseteq\R^N$ open, for $s\in\R$ and $1<p<\infty$, similar to \eqref{Eqn::Space::TildeSpace} we can define 
\begin{equation}\label{Eqn::Space::TildeHsp}
    \widetilde H^{s,p}(\overline U):=\{f\in H^{s,p}(\R^N):\supp f\subseteq\overline U\},\qquad \widetilde \Co^s(\overline U):=\{f\in \Co^s(\R^N):\supp f\subseteq\overline U\}.
\end{equation}
Recall that in literature we use have space $H_0^{s,p}(U)$ given by the closure of $C_c^\infty(U)$ under $\|\cdot\|_{H^{s,p}(\R^N)}$.
\begin{lem}\label{Lem::Space::DenseHs}
    Let $\Omega\subseteq\R^N$ be a bounded Lipschitz domain. Let $1<p<\infty$.
    \begin{enumerate}[(i)]
        \item\label{Item::Space::DenseHs::>0}For $s\ge0$, we have identification $H_0^{s,p}(\Omega)=\widetilde H^{s,p}(\overline\Omega)=\widetilde\Fs_{p2}^s(\overline \Omega)$. Moreover $H_0^{s,p}(\Omega)'=H^{-s,p'}(\Omega)$.
        \item\label{Item::Space::DenseHs::<0} For $s\le0$, $\Co_c^\infty(\Omega)$ is a dense subset of $H^{s,p}(\Omega)$.
    \end{enumerate}
\end{lem}

\begin{proof}
    See e.g. \cite[Theorem~3.5]{TriebelLipschitzDomain} for both \ref{Item::Space::DenseHs::>0} and \ref{Item::Space::DenseHs::<0} where we recall from Lemma~\ref{Lem::Space::SpaceLem}~\ref{Item::Space::SpaceLem::Hsp} that $H^{s,p}=\Fs_{p2}^s$. We also refer to \cite[Proposition~2.12]{ShiYaoC2} for a brief explanation of \ref{Item::Space::DenseHs::>0}.
\end{proof}

\begin{rem}
    We have the restriction map $f\mapsto f|_\Omega$, which is bounded $\widetilde H^{s,p}(\overline \Omega)\to H^{s,p}(\Omega)$ for all $s\in\R$ and all open $\Omega\subseteq\R^N$. When $\Omega$ is a bounded Lipschitz domain, this map is injective for $s\ge0$. However for such $\Omega$ when $s<1/p-1$, this is not injective due to the existence of distributions $h\in H^{s,p}(\R^N)$ such that $\supp h\subseteq b\Omega$ (we shall not say $h\in H^{s,p}(b\Omega)$ since $b\Omega$ is a hypersurface: $H^{s,p}(b\Omega)$ in literature, e.g. \cite[Chapter~7]{TriebelTheoryOfFunctionSpacesII}, refers to the space on a Riemannian manifold). See \cite[Section~3.2]{TriebelLipschitzDomain} for more discussions.
\end{rem}
\begin{defn}\label{Defn::Space::SpecDom}
    A \textit{special Lipschitz domain} on $\R^N$ is an open subset $\omega=\{(x',x_N):x_N>\sigma(x')\}$ where $\sigma:\R^{N-1}\to\R$ is a Lipschitz function such that $\|\nabla\sigma\|_{L^\infty}<1$ in the sense that $$\sup_{x',y'\in\R^{N-1};x'\neq y'}|\sigma(x')-\sigma(y')|/|x'-y'|<1.$$
\end{defn}
\begin{rem}\label{Rmk::Space::SpecDom+Cone}
    Let $\Kb:=\{(x',x_N):x_N>|x'|\}$ be an open cone. Then we have $\omega+\Kb\subseteq\omega$.
\end{rem}

\begin{defn}\label{Defn::Space::ExtOp}
For a special Lipschitz domain $\omega\subset\R^N$ the \textit{Rychkov's universal extension operator} $\Ec=\Ec_\omega$ for $\omega$ is given by the following:
\begin{equation}\label{Eqn::Space::ExtOp}
\Ec_\omega f: =\sum_{j=0}^\infty\psi_j\ast(\1_{\omega}\cdot(\phi_j\ast f)),\qquad f\in\Ss'(\omega).
\end{equation}
Here $(\psi_j)_{j=0}^\infty$ and $(\phi_j)_{j=0}^\infty$ are families of Schwartz functions that satisfy the following properties: 
\begin{itemize}
    \item\label{Item::Space::PhiScal} \textit{Scaling condition}: $\phi_j(x)=2^{(j-1)N}\phi_1(2^{j-1}x)$ and $\psi_j(x)=2^{(j-1)N}\psi_1(2^{j-1}x)$ for $j\ge2$.
	\item\label{Item::Space::PhiMomt} \textit{Moment condition}: $\int\phi_0=\int\psi_0=1$, $\int x^\alpha\phi_0(x)dx=\int x^\alpha\psi_0(x)dx=0$ for all multi-indices $|\alpha|>0$, and $\int x^\alpha\phi_1(x)dx=\int x^\alpha\psi_1(x)dx=0$ for all $|\alpha|\ge0$.
	\item\label{Item::Space::PhiApprox}\textit{Approximate identity}: $\sum_{j=0}^\infty\phi_j=\sum_{j=0}^\infty\psi_j\ast\phi_j=\delta_0$ is the Dirac delta measure.
	\item\label{Item::Space::PhiSupp} \textit{Support condition}: $\phi_j,\psi_j$ are all supported in the negative cone $-\Kb=\{(x',x_N):x_N<-|x'|\}$.
\end{itemize}
\end{defn}
For an extension operator on a bounded Lipschitz domain one needs to combine \eqref{Eqn::Space::ExtOp} with partition of unity. See \cite[(6.1)]{ShiYaoExt} or \cite[(4.14)]{YaoCXFinite}. In the paper we only need to use extension on special Lipschitz domains.

\begin{prop}{\normalfont\cite{RychkovExtension}}\label{Prop::Space::ExtOp} There are $(\phi_j)_{j=0}^\infty,(\psi_j)_{j=0}^\infty$ which satisfy the assumptions in Definition~\ref{Defn::Space::ExtOp}. The extension operator $\Ec$ is bounded $\Ec:\Ss'(\omega)\to\Ss'(\R^N)$ and
\begin{equation}\label{Eqn::Space::EBdd}
    \Ec:\Fs_{pq}^s(\omega)\to\Fs_{pq}^s(\R^N),\qquad p,q\in(0,\infty],\quad s\in\R.
\end{equation}

In particular
\begin{equation}\label{Eqn::Space::[D,E]Bdd}
    [\nabla,\Ec]:\Fs_{pq}^s(\omega)\to\widetilde\Fs_{pq}^{s-1}(\omega^c;\C^N),\qquad p,q\in(0,\infty],\quad s\in\R.
\end{equation}
\end{prop}
\begin{proof}
    The result $\Ec:\Fs_{pq}^s\to\Fs_{pq}^s$ for $p<\infty$ or $p=q=\infty$ is done in \cite[Theorem~4.1]{RychkovExtension}. The case $p=\infty>q$ is done in \cite[Theorem~4.6]{ZhuoTriebelType} (see also \cite[Remark~20]{YaoExtensionMorrey}). A complete proof of $\Ec:\Ss'\to\Ss'$ can be found in \cite[Proposition~3.10~(i)]{ShiYaoExt}.

    By Lemma~\ref{Lem::Space::EqvNormDer} we get $[\nabla,\Ec]:\Fs_{pq}^s\to\Fs_{pq}^{s-1}$. Since $([\nabla,\Ec]f)|_\omega=0$ for all $f\in\Ss'(\omega)$, using the notation \eqref{Eqn::Space::TildeSpace} we get \eqref{Eqn::Space::[D,E]Bdd}.
\end{proof}
\begin{rem}\label{Rmk::Space::ExtOp} 
    Fix a $\R$-linear basis $v=(v_1,\dots,v_N)$ for $\R^N$ and let $V_v:=\{a_1v_1+\dots+a_Nv_N:a_1,\dots,a_N>0\}$. By passing to an invertible linear transform, we can replace the support condition by $\supp\phi_j,\supp\psi_j\subset V_v$ while keeping the scaling, moment and identity conditions.
\end{rem}

If $f$ has compact support in $\omega$, then $[\nabla,\Ec]f$ is smooth.
\begin{lem}\label{Lem::Space::DisjointSuppSmooth}
    Keeping the same notations in Proposition~\ref{Prop::Space::ExtOp}, $\Ec:\Es'(\omega)\to \Co^\infty(\omega^c)$ is continuous. More precisely for every $M>0$ there is a $C_M>0$ such that
    \begin{equation}\label{Eqn::Space::DisjointSuppSmooth}
        \|\Ec f\|_{\Co^M(\omega^c)}\le C_M\|f\|_{\Co^{-M}}\text{ for every }f\in\Co^{-M}(\omega)\text{ such that }\dist(\supp f,\omega^c)>\tfrac1M.
    \end{equation}
    In particular  $[\nabla,\Ec]:\Es'(\omega)\to \widetilde\Co^\infty(\omega^c)$ is also continuous.
\end{lem}
\begin{proof}
    Fix a $M>0$. We identify $f$ with its zero extension, for which $\|f\|_{\Co^{-M}(\R^N)}\lesssim_M\|f\|_{\Co^{-M}(\omega)}$ under the support assumption. By Theorem~\ref{Thm::RefAtD} or \cite[Lemma~6.11]{ShiYaoExt} (applying to $\omega^c$) there are operators $(\Sc^\alpha)_{|\alpha|\le M+1}:\Co^{-M}(\R^N)\to L^\infty(\R^N)$ such that $g=\sum_{|\alpha|\le M+1}D^\alpha\Sc^\alpha g$ for all $g\in\Co^{-M}(\R^N)$ and $\supp(\Sc^\alpha g)\subseteq \Kb+\supp g$. For every $f\in\Co^{-M}(\omega)$ such that $\dist(\supp f,\omega^c)>\frac1M$, we have $\dist(\supp\Sc^\alpha f,\omega^c)>\frac1M$ as well.
    By the dilation assumption, for every $M'>M$ and $\alpha\in\N^N$, we have $|D^\alpha\phi_j(x)|+|D^\alpha\psi_j(x)|\lesssim_{M',\alpha} 2^{j(N+|\alpha|)}(1+2^j|x|)^{-M'}$ for all $j\ge0$. Therefore,
    \begin{align*}
        \|\Ec f\|_{\Co^M(\omega^c)}&\lesssim\sum_{|\alpha|,|\beta|\le M+1}\|D^\alpha\Ec D^\beta(\Sc^\beta f)\|_{L^\infty(\omega^c)}\lesssim\sum_{|\alpha|,|\beta|\le M+1}\sum_{j=0}^\infty\|D^\alpha\psi_j\ast(\1_{\omega}\cdot(D^\beta\phi_j\ast \Sc^\beta f))\|_{L^\infty(\omega^c)}
        \\
        &\le\sum_{|\alpha|,|\beta|\le M+1}\sum_{j=0}^\infty\sup_{x\in\omega^c}\int_{\dist(y,\omega^c)>\frac1M}\int_\omega|D^\alpha\psi_j(x-z)||D^\beta\phi_j(z-y)|dzdy\|\Sc^\beta f\|_{L^\infty}
        \\
        &\lesssim \sum_{j=0}^\infty\int_{|u|>1/M}\frac{2^{j(2M+2)}\cdot2^{jN}du}{(1+2^j|u|)^{N+2M+3}}\|f\|_{\Co^{-M}}\lesssim \sum_{j=0}^\infty\frac{2^{j(2M+2+N)}}{(1+2^j/M)^{N+2M+3}}\|f\|_{\Co^{-M}}\lesssim \|f\|_{\Co^{-M}}.
    \end{align*}

    The same estimate holds with $\omega^c$ replaced by $\{x:\dist(x,\omega^c)<\frac1{2M}\}$. Since $[\nabla,\Ec]f=\nabla(\Ec f)-\Ec(\nabla f)$ vanishes in $\omega$, replacing $f$ in \eqref{Eqn::Space::DisjointSuppSmooth} by $f$ and $\nabla f$ gives the continuity of $[\nabla,\Ec]:\Es'(\omega)\to \widetilde\Co^\infty(\omega^c)$.
\end{proof}
\begin{defn}
    Let $U,V\subseteq\R^N$ be open sets (not necessarily bounded), we say $\Phi:U\to V$ is \textit{diffeomorphism}, if it is bijective smooth and $\Phi^{-1}$ is also smooth.
    
    We say $\Phi$ is a \textit{bounded diffeomorphism}, if in addition, for every multi-index $\alpha\neq0$, $\sup_{x\in U}|D^\alpha\Phi(x)|+\sup_{y\in V}|D^\alpha(\Phi^{-1})(y)|<\infty$.
\end{defn}

\begin{lem}\label{Lem::Space::ExtDiffeo}
    Let $\Phi:U\to V$ be a diffeomorphism. For every $x_0\in U$ there is an open subset $x_0\in U'\subset U$ such that $\Phi|_{U'}$ admits an extension to a bounded diffeomorphism $\tilde\Phi:\R^N\to\R^N$.
\end{lem}
\begin{proof}
    In the proof for a matrix $A\in\R^{N\times N}$ we use the standard matrix norm $|A|:=\sup_{v\in\R^N;|v|=1}|Av|$.

    By composing an invertible affine linear transform, which itself is clearly a bounded diffeomorphism, we can assume that $x_0=0$, $\Phi(0)=0$ and $\nabla\Phi(0)=I\in\R^{N\times N}$ is the identity matrix. Therefore by smoothness of $\Phi$ there is a $C>0$ (depending on $\nabla^2\Phi$) such that $B(0,1/C)\subset U$, $|\nabla\Phi(x)-I|\le\frac12$ and $|\Phi(x)-x|\le C|x|^2$ for all $|x|<1/C$.

    Fix a $\chi\in C_c^\infty(B(0,1))$ such that $\1_{B(0,1/2)}\le\chi\le\1_{B(0,1)}$. Let $\chi_R(x)=\chi(Rx)$ for $R>C$, and let $\tilde \Phi_R(x)=\chi_R(x)\cdot \Phi(x)+(1-\chi_R(x))\cdot x$. Therefore $\tilde\Phi:\R^N\to\R^N$ satisfies $\tilde\Phi_R|_{B(0,1/R)}=\Phi|_{B(0,1/R)}$.
    
     Since $\tilde\Phi_R(x)-x=\chi_R(x)\cdot(\Phi(x)-x)$ and $\|\nabla\chi_R\|_{L^\infty}=R\|\nabla\chi\|_{L^\infty}$, we see that
    \begin{equation*}
        \sup_{x\in\R^N}|\nabla\tilde\Phi_R(x)-I|\le\sup_{|x|<R^{-1}}|\nabla\chi_R(x)||\Phi(x)-x|+\sup_{|x|<R^{-1}}|\nabla\Phi(x)-I|\le R\|\nabla\chi\|_{L^\infty}\cdot CR^{-2}+\tfrac12.
    \end{equation*}

    By taking $R$ large enough we have $\sup_{x\in\R^N}|\nabla\tilde\Phi_R(x)-I|<\frac23$. Fix such $R$, take $U':=B(0,(2R)^{-1})$ and $\tilde\Phi:=\tilde \Phi_R$, we see that $\tilde\Phi:\R^N\to\R^N$ has all bounded derivatives and $\tilde \Phi|_{U'}=\Phi|_{U'}$.
    
    Since $\sup_x|\nabla\tilde\Phi(x)-I|<\frac23$, the map $\id-\tilde\Phi:\R^N\to\R^N$ is a contraction. By the fixed point theorem for every $y\in\R^N$ there is a unique $x\in\R^N$ such that $x=y+x-\tilde\Phi(x)$, i.e. $\tilde\Phi(x)=y$. We conclude that $\tilde\Phi$ is surjective. On the other hand,    
    $|\tilde\Phi_R(x_1)-\tilde\Phi_R(x_2)|\ge(1-\frac23)|x_1-x_2|=\frac13|x_1-x_2|$ for all $x_1,x_2\in\R^N$. We conclude that $\tilde\Phi_R$ is injective and $\sup_y|\nabla(\Phi^{-1})(y)|<3$. The boundedness of $\sup_y|D^\alpha(\tilde\Phi^{-1})(y)|$ then follows from taking chain rules to the inverse functions.
\end{proof}

\begin{prop}\label{Prop::Space::CompBdd}
    Let $\psi\in\Co^\infty(\R^N)$, we define the multiplier operator $\M^\psi f:=\psi f$.

    \begin{enumerate}[(i)]
        \item We have boundedness $\M^\psi:\Fs_{pq}^s(\R^N)\to\Fs_{pq}^s(\R^N)$ for all $p,q\in(0,\infty]$ and $s\in\R$. In particular, for every open subsets $\Omega\subset\R^N$,
    \begin{equation}\label{Eqn::Space::MultBdd}
        \M^\psi:\Fs_{pq}^s(\Omega)\to\Fs_{pq}^s(\Omega),\quad p,q\in(0,\infty],\quad s\in\R.
    \end{equation}
        \item Let $U,V\subseteq\R^N$ be open subsets and let $\Phi:U\to V$ be a diffeomorphism. Assume $\psi\in\Co_c^\infty(V)$. Then the map $\Phi^*\circ\M^\psi f=\M^{\psi\circ\Phi}\circ\Phi^*f=(\psi f)\circ\Phi$ is bounded $\Phi^*\circ\M^\psi:\Fs_{pq}^s(\R^N)\to\Fs_{pq}^s(\R^N)$. In particular, for every open subsets $\Omega,\omega\subset\R^N$ such that $\Phi^{-1}(V\cap\Omega)=U\cap\omega$,
    \begin{equation}\label{Eqn::Space::CompBdd}
        \Phi^*\circ\M^\psi:\Fs_{pq}^s(\Omega)\to\Fs_{pq}^s(\omega),\quad p,q\in(0,\infty],\quad s\in\R.
    \end{equation}
    \end{enumerate}

\end{prop}
\begin{proof}
    See e.g. \cite[Theorem~2.28]{TriebelTheoryOfFunctionSpacesIV} for boundedness of $\M^\psi$ on $\Fs_{pq}^s(\R^N)$. The domain case \eqref{Eqn::Space::MultBdd} follows from the standard extension argument via \eqref{Eqn::Space::TLDomain}.

    We now prove \eqref{Eqn::Space::CompBdd}. By \cite[Theorem~2.25]{TriebelTheoryOfFunctionSpacesIV}, for every bounded diffeomorphism $\tilde\Phi:\R^N\to\R^N$, the pullback map $\tilde\Phi^*=[f\mapsto f\circ\tilde\Phi]:\Fs_{pq}^s(\R^N)\to\Fs_{pq}^s(\R^N)$ is bounded for all $p,q,s$. 

    Since $\supp\psi$ is now compact, by Lemma~\ref{Lem::Space::ExtDiffeo} there is a finitely open cover $\bigcup_{\nu=1}^MU'_\nu\supset\Phi^{-1}(\supp\psi)$ such that for each $\nu$, $\Phi|_{U'_\nu}$ admits a bounded diffeomorphism extension $\tilde\Phi_\nu:\R^N\to\R^N$. Take a partition of unity $\{\chi_\nu\in\Co_c^\infty(U'_\nu)\}_{\nu=1}^M$ such that $\psi=\sum_{\nu=1}^M\chi_\nu\psi$ we have, for every $f\in\Fs_{pq}^s(\R^N)$,
    \begin{align*}
        \|(\psi f)\circ\Phi\|_{\Fs_{pq}^s}\le\sum_{\nu=1}^M\|\Phi^*(\chi_\nu\psi f)\|_{\Fs_{pq}^s}=\sum_{\nu=1}^M\|\tilde\Phi_\nu^*(\chi_\nu\psi f)\|_{\Fs_{pq}^s}\lesssim_{p,q,s,\tilde\Phi_\nu}\sum_{\nu=1}^M\|\chi_\nu\psi f\|_{\Fs_{pq}^s}\overset{\eqref{Eqn::Space::MultBdd}}\lesssim\|f\|_{\Fs_{pq}^s}.
    \end{align*}
    We get $\Phi^*\circ\M^\psi:\Fs_{pq}^s(\R^N)\to\Fs_{pq}^s(\R^N)$. By passing to domain, since $(\psi f)\circ\Phi|_{\omega\backslash V}\equiv0$, we get \eqref{Eqn::Space::CompBdd} as well.
\end{proof}

\begin{lem}[{Schur's test, see e.g. \cite[Appendix B]{RangeSCVBook}}]\label{Lem::Space::Schur}
    Let $(X,\mu)$ and $(Y,\nu)$ be two measure spaces. Let $G:X\times Y\to\C$ be a measurable function (with respect to $\mu\otimes\nu$). Let $1\le \gamma\le\infty$ and $A>0$ that satisfy
    \begin{equation*}
        \|G\|_{L^\infty_yL^\gamma_x}=\essup_{y\in Y}\Big(\int_X|G(x,y)|^\gamma d\mu(x)\Big)^{1/\gamma}\le A;\qquad \|G\|_{L^\infty_xL^\gamma_y}=\essup_{x\in X}\Big(\int_Y|G(x,y)|^\gamma d\nu(y)\Big)^{1/\gamma}\le A.
    \end{equation*}
    
    Then the integral operator $Tf(y):=\int_XG(x,y)f(x)d\mu(x)$ is bounded $T:L^p(X,d\mu)\to L^q(Y,d\nu)$, with operator norm $\|T\|_{L^p\to L^q}\le A$ for all $1\le p,q\le\infty$ such that $\frac1q=\frac1p+\frac1\gamma-1$.
\end{lem}
\begin{rem}
    Alternatively one can show $\|T\|_{L^{\gamma'}\to L^\infty}\le\|G\|_{L^\infty_yL^\gamma_x}$, thus by duality $\|T\|_{L^1\to L^\gamma}\le \|G\|_{L^\infty_xL^\gamma_y}$. By complex interpolation, i.e. the Riesz-Thorin theorem,
    \begin{equation*}
        \|T\|_{L^{\frac\gamma{\gamma-\theta}}\to L^{\frac\gamma{1-\theta}}}\le\|G\|_{L^\infty_yL^\gamma_x}^\theta\|G\|_{L^\infty_xL^\gamma_y}^{1-\theta},\qquad 0\le\theta\le1.
    \end{equation*}
    We leave the proof to the readers.
\end{rem}

One key ingredient to achieve $1/2$-estimates for the homotopy operators is the Hardy's distance inequality, which builds a bridge between weighted Sobolev spaces and fractional Sobolev spaces / H\"older-Zygmund spaces. In our case we use a strong result via Triebel-Lizorkin spaces.
\begin{prop}[{Strong Hardy's distance inequality \cite[Proposition~5.3]{YaoCXFinite}}]\label{Prop::Space::HLLem}
    Let $\omega=\{x_N>\sigma(x')\}$ be a special Lipschitz domain, let $\delta(x)=\min(1,\dist(x,b\omega))$ for $x\in\R^N$.
    
    Let $1\le p\le \infty$ and $s\in\R$. Then we have the following embeddings between Triebel-Lizorkin spaces and weighted Sobolev spaces:
    \begin{align}\label{Eqn::Space::HLTilde}
        &\widetilde\Fs_{p\infty}^s(\omega^c)\hookrightarrow L^p(\overline\omega^c,\delta^{-s}),&&\text{for every }s>0;
        \\\label{Eqn::Space::HLSob}
        &W^{m,p}(\omega,\delta^{m-s})\hookrightarrow\Fs_{p\eps}^s(\omega),&&\text{for every integer }m>s\text{ and every }\eps>0.
    \end{align}
\end{prop}

\begin{rem}
    For the result \eqref{Eqn::Space::HLSob}, the norm \eqref{Eqn::Space::WeiSobo} of $W^{m,p}(\omega,\delta^{m-s})$ is intrinsic while the norm \eqref{Eqn::Space::TLDomain} of $\Fs_{p\eps}^s(\omega)$ relies on its extensibility to $\R^N$. The extensibility issue can be resolved by using Rychkov's extension $\Ec f$ in \eqref{Eqn::Space::ExtOp} for $f\in W^{m,p}(\omega,\delta^{m-s})$, where each summand $\psi_j\ast(\1_\omega\cdot(\phi_j\ast f))$ makes sense without choosing any a priori extension of $f$.

    To prove \eqref{Eqn::Space::HLSob}, if one only cares $\eps\ge1$ the result can be done by applying duality argument on \eqref{Eqn::Space::HLTilde}. The duality argument avoid using intrinsic characterization of $\Fs_{p\eps}^s(\omega)$, see \cite[Section~5]{YaoExtensionMorrey} for details. In \cite[Appendix~A]{YaoExtensionMorrey} a direct proof of \eqref{Eqn::Space::HLSob} is given using an intrinsic characterization by Rychkov: for $p,q\notin\{\infty\}\times(0,\infty)$,
    \begin{equation*}
        \|f\|_{\Fs_{pq}^s(\omega)}\approx_{p,q,s}\|(2^{js}\phi_j\ast f)_{j=0}^\infty\|_{L^p(\omega;\ell^q(\N))}.
    \end{equation*}
    When $p=\infty>q$ the intrinsic characterization is also true but a modification is required.
    
    If one assume $f\in L^1_\loc(\omega)$, the right hand side above also makes sense without a priori extensions. We can pick $\Ec f$ from \eqref{Eqn::Space::ExtOp} for an extension to $\R^N$ as well.
\end{rem}

The $\frac12$-estimates are then the corollary to the following:
\begin{cor}\label{Cor::Space::IntOpBdd}
    Let $U\subset\R^N$ and $V\subset\R^M$ be two bounded Lipschitz domains. Let $\delta_U(x):=\dist(x,bU)$ and $\delta_V(y):=\dist(y,bV)$.
        Let $r,s>0$, $\gamma\in[1,\infty]$ and $m\in\N$. Suppose $S\in L^1_\loc(U\times V)$ satisfies
    \begin{equation*}
        \sup_{y\in V}\delta_V(y)^r\Big(\sum_{|\alpha|\le m}\int_{U}|D^\alpha_yS(x,y)\delta_U(x)^s|^\gamma dx\Big)^{1/\gamma}+\sup_{x\in U}\delta_U(x)^s\Big(\sum_{|\alpha|\le m}\int_{V}|D^\alpha_yS(x,y)\delta_V(y)^r|^\gamma dy\Big)^{1/\gamma}<\infty.
    \end{equation*}

    Then $Tf(y):=\int_US(x,y)f(x)dx$ defines a linear operator with boundedness
    \begin{equation}\label{Eqn::Space::IntOpBdd}
        T:\widetilde\Fs_{p\infty}^s(\overline U)\to\Fs_{q\eps}^{m-r}(V),\quad\text{for all }\eps>0\quad\text{and}\quad p,q\in[1,\infty]\text{ such that }\tfrac1q=\tfrac1p+\tfrac1\gamma-1.
    \end{equation}
    In particular $T:\widetilde H^{s,p}(\overline U)\to H^{m-r,q}(V)$ for all such $1<p<\frac\gamma{\gamma-1}$. And if $\gamma=1$, then in addition $T:\widetilde \Co^s(\overline U)\to \Co^{m-r}(V)$.
\end{cor}

\begin{proof}
    By taking $\mu=dx$, $\nu=dy$ and $G(x,y)=\delta_V(y)^rD^\alpha S(x,y)\delta_U(x)^s$ in Lemma~\ref{Lem::Space::Schur}, the Schur's test yields the boundedness $D^\alpha T:L^p(U,\delta_U^{-s})\to L^q(V,\delta_V^r)$ for all $1\le p\le\frac\gamma{\gamma-1}$ and $|\alpha|\le m$. In other words $T:L^p(U,\delta_U^{-s})\to W^{m,q}(V,\delta_V^r)$ holds. Applying Proposition~\ref{Prop::Space::HLLem}, the Strong Hardy's inequalities yield the embeddings $\widetilde\Fs_{p\infty}^s(\overline U)\hookrightarrow L^p(U,\delta_U^{-s})$ and $W^{m,q}(V,\delta_V^r)\hookrightarrow\Fs_{q\eps}^{m-r}(V)$ for $1\le p\le\frac\gamma{\gamma-1}$ and $\eps>0$. Taking compositions we conclude $T:\widetilde\Fs_{p\infty}^s(\overline U)\to\Fs_{q\eps}^{m-r}(V)$.

    By Lemma~\ref{Lem::Space::SpaceLem}, $\widetilde H^{s,p}(\overline U)=\widetilde\Fs_{p2}^s(\overline U)\subset\widetilde\Fs_{p\infty}^s(\overline U)$, $H^{m-r,q}(V)=\Fs_{q2}^{m-r}(V)\supset\Fs_{q\eps}^{m-r}(V)$ if $1<p,q<\infty$, $\widetilde \Co^s(\overline U)=\widetilde\Fs_{\infty\infty}^s(\overline U)$ and $\Co^{m-r}(V)=\Fs_{\infty\infty}^{m-r}(V)\supset\Fs_{\infty\eps}^{m-r}(V)$. The bound $T:\widetilde H^{s,p}\to H^{m-r,q}$ and $T:\widetilde \Co^s\to \Co^{m-r}$ follow immediately.
\end{proof}

Recall the Bochner-Martinelli form $B(z,\zeta)=\sum_{q=0}^{n-1}B_q(z,\zeta)$ in \eqref{Eqn::Intro::DefB}. One can see that its coefficients are constant linear combinations of derivatives of the Newtonian potential. We use the associated integral operator 
\begin{equation}\label{Eqn::Space::DefBOp}
    \Bc_qg(z):=\int_{\C^n}B_{q-1}(z,\cdot)\wedge g,\qquad 1\le q\le n,
\end{equation}
whenever the integral makes sense.

\begin{lem}\label{Lem::Space::BMKernel}
    For every bounded open set $\Uc\subset\C^n$ and $1\le q\le n$, the $\Bc_q$ defines an operator
    \begin{equation}\label{Eqn::Space::BMKernelBdd}
        \Bc_q:\widetilde\Fs_{pr}^s(\overline\Uc;\wedge^{0,q})\to\Fs_{pr}^{s+1}(\Uc;\wedge^{0,q-1})\quad\text{for all }p,r\in(0,\infty],\ s\in\R.
    \end{equation}

    In particular, for every $\lambda,\chi\in\Co_c^\infty(\C^n)$, the map $[f\mapsto \lambda\cdot\Bc_q[\chi f]]$ is bounded $H^{s,p}(\C^n;\wedge^{0,q})\to H^{s+1,p}(\C^n;\wedge^{0,q-1})$ and $\Co^s(\C^n;\wedge^{0,q})\to \Co^{s+1}(\C^n;\wedge^{0,q-1})$ for all $s\in\R$ and $1<p<\infty$.
\end{lem}
\begin{proof}
    For $(p,r)\notin\{\infty\}\times(0,\infty)$ see e.g. \cite[Lemma~6.1]{YaoCXFinite}. Recall from Lemma~\ref{Lem::Space::SpaceLem} that $H^{s,p}=\Fs_{p2}^s$ and $\Co^s=\Fs_{\infty\infty}^s$. We give an alternative proof below. Notice that the $\Bc_q$ is defined on distributions because the convolutions of the Newtonian potential and compactly supported distributions are well-defined.

    One can see that from e.g. \cite[Chapter~IV.1]{RangeSCVBook} $B_{q-1}(z,\zeta)$ is  the constant linear combinations of the derivatives of the Newtonian potential $G(z-\zeta):=-\frac{(n-2)!}{4\pi^n}|z-\zeta|^{2-2n}$ for $n\ge2$. By e.g. \cite[Theorem~1.24]{TriebelTheoryOfFunctionSpacesIV} $\nabla:\Fs_{pr}^{s+2}(\C^n)\to\Fs_{pr}^{s+1}(\C^n;\C^{2n})$ is bounded, by passing to domain we get $\nabla:\Fs_{pr}^{s+2}(\Uc)\to\Fs_{pr}^{s+1}(\Uc;\C^{2n})$ as well. Therefore, it suffices to prove $[f\mapsto G\ast f]:\widetilde\Fs_{pr}^s(\overline\Uc)\to\Fs_{pr}^{s+2}(\Uc)$ for all $0<p,r\le\infty$ and $s\in\R$.

    Let $(\phi_j)_{j=0}^\infty$ be a Littlewood-Paley family for $\R^N=\C^n$ ($N=2n$) from Definition~\ref{Defn::Space::TLSpace}. Define $(\psi_j)_{j=1}^\infty\subset\Ss(\R^N)$ be such that $\psi_j(x)=2^{(j-1)N}\psi_1(2^{j-1}x)$ for $j\ge2$, and $\hat\psi_1\in\Co_c^\infty(\R^N\backslash\{0\})$, $\hat\psi_1|_{\supp\hat\phi_1}\equiv1$ for $j\ge1$. We see that $\phi_j=\psi_j\ast\phi_j$ for all $j\ge1$. Therefore $f=\phi_0\ast f+\sum_{j=1}^\infty\psi_j\ast \phi_j\ast f$ for $f\in\Ss'(\R^N)$.

    Let $\eta_j:=2^{2j}G\ast\psi_j$ for $j\ge1$, we see that $\eta_j\in\dot\Ss(\R^N)$ (is Schwartz and has all moment vanishing) and $\eta_j(x)=2^{(j-1)N}\eta_1(2^{j-1}x)$ for $j\ge2$ as well. By Proposition~\ref{Prop::AntiDev::RychOp}, with $\theta_j=\phi_j$ and $r=-2$ in the statement, we see that $[f\mapsto\sum_{j=1}^\infty G\ast\psi_j\ast\phi_j\ast f]:\Fs_{pr}^s(\R^N)\to\Fs_{pr}^{s+2}(\R^N)$. That is, $[f\mapsto G\ast(f-\phi_0\ast f)]:\Fs_{pr}^s(\C^n)\to\Fs_{pr}^{s+2}(\C^n)$ is bounded.

    On the other hand, $G\ast \phi_0\in\Co^\infty_\loc(\R^N)$ therefore $[f\mapsto G\ast \phi_0\ast f]:\Es'(\R^N)\to\Co^\infty_\loc(\R^N)$. In particular $[f\mapsto G\ast \phi_0\ast f]:\widetilde\Fs_{pr}^s(\overline\Uc)\to\Co^\infty(\Uc)$. 

    Since $G\ast f=G\ast\phi_0\ast f+G\ast(f-\phi_0\ast f)$. Combining the above results we obtain \eqref{Eqn::Space::BMKernelBdd}.
\end{proof}

\begin{lem}\label{Lem::Space::BMFormula}
    Let $0\le q\le n$. For every $(0,q)$ form $g$ whose coefficients are compactly supported distribution, i.e. $g\in\Es'(\C^n;\wedge^{0,q})$, we have Bochner-Martinelli formula $g=\dbar\Bc_qg+\Bc_{q+1}\dbar g$.

    In other words, using the mixed degree convention we have $\id=\dbar\Bc+\Bc\dbar$.
\end{lem}
Here we set $\Bc_{0}=\Bc_{n+1}=0$ as usual.
\begin{proof}
    See e.g. \cite[Theorem~11.1.2]{ChenShawBook}. In the reference $g=\dbar\Bc_qg+\Bc_{q+1}\dbar g$ holds for $g\in C^1_c$. The case where $g$ is compactly supported distribution can be done by taking smooth approximations.
\end{proof}
\begin{rem}
    Fix $s\le0$, $1<p<\infty$ and a bounded Lipschitz domain $\Omega\subset\C^n$. From above we have $g=\dbar\Bc g+\Bc\dbar g$ for $g\in C_c^\infty(\Omega;\wedge^{0,\bullet})$. Taking $\|\cdot\|_{H^{s,p}(\R^N)}$ approximations, this is true for $g\in \widetilde H^{s,p}(\overline\Omega)$ as well.
    
    Although by Lemma~\ref{Lem::Space::DenseHs}~\ref{Item::Space::DenseHs::<0} $C_c^\infty(\Omega)\subset H^{s,p}(\Omega)$ is dense in $H^{s,p}(\Omega)$, this formula does \textbf{not} extend to all $g\in H^{s,p}(\Omega;\wedge^{0,\bullet})$ because (the restricted operator) $\Bc:C_c^\infty(\Omega;\wedge^{0,\bullet})\to\Co^\infty(\Omega;\wedge^{0,\bullet})$ cannot extend to a bounded map $H^{s-1,p}(\Omega;\wedge^{0,\bullet})\to\Ds'(\Omega;\wedge^{0,\bullet})$!

    Indeed, suppose $\Bc_1$ has continuous extension to $H^{-1,p}(\Omega;\wedge^{0,1})\to\Ds'(\Omega)$, then for every function $g\in L^p(\Omega)$ we have $g=\Bc_1\dbar g$ in $\Omega$ (since $\Bc_0=0$). However take $g=\1_\Omega$ with $L^p$ approximations $(g_k)_{k=1}^\infty\subset C_c^\infty(\Omega)$, then  $0=(\dbar g)|_\Omega=\lim_{j\to\infty}\dbar g_j|_\Omega$ converging in $H^{-1,p}(\Omega) $. By the continuity assumption $\Bc_1\dbar g=0$, but $g\neq0$ in $\Omega$, contradiction!

    This idea will also be used in the proof of Lemma~\ref{Lem::Space::IllPosed} below.
\end{rem}

In the paper we construct solution operators using integral representations. This is a different method from the $\dbar$-Neumann approach. It is important to notice that  the canonical solution is never well-posed on spaces of distributions with large enough order. This is mentioned implicitly in \cite{ShiYaoCk,YaoCXFinite}. For completeness, let us formulate the statement here.
\begin{lem}\label{Lem::Space::IllPosed}
    Let $n\ge1$ and $\Omega\subset\C^n$ be a bounded pseudoconvex domain. For every $1\le q\le n$ the canonical solution operator $\dbar^*N_q:L^2(\Omega;\wedge^{0,q})\to L^2(\Omega;\wedge^{0,q-1})$ on $(0,q)$ forms cannot extend to a continuous linear operator $\dbar^*N_q:H^{-2,2}(\Omega;\wedge^{0,q})\to \Ds'(\Omega;\wedge^{0,q-1})$. In particular the $\dbar$-Neumann problem is not well-posed on $H^{-2,2}(\Omega;\wedge^{0,q})$ for all $1\le q\le n$.
\end{lem}
\begin{proof}
    For a $\dbar$-closed $(0,q)$-form $f$ on $\Omega$ with $L^2$ coefficients, the canonical solution $u=\sum_Iu_Id\bar z^I:=\dbar^*N_qf$ is a $(0,q-1)$ form on $\Omega$ such that $\dbar u=f$ and $\|u\|_{L^2(\Omega;\wedge^{0,q-1})}^2:=\sum_I\|u_I\|_{L^2}^2$ is minimal. Here $N_q:L^2(\Omega;\wedge^{0,q})\to L^2(\Omega;\wedge^{0,q})$ is the $\dbar$-Neumann solution operator. See e.g. \cite[Corollary~4.4.2]{ChenShawBook}. 

    In the proof we use the fact that $\dbar:H^{-1,2}(\Omega;\wedge^{0,q-1})\to H^{-2,2}(\Omega;\wedge^{0,q})$. Recall that $\Box_q=\dbar_{q-1}\dbar^*_q+\dbar^*_{q+1}\dbar_q=-\sum_{j=1}^n\partial^2_{z_j\bar z_j}$ is the complex Laplacian acting on components.

    Suppose by contrast that we have continuous extension $S_q:H^{-2,2}\to\Ds'$ where $S_q|_{L^2}=\dbar^*N_q$. Taking adjoint we have $S_q^*:C_c^\infty(\Omega;\wedge^{0,q-1})\to H_0^{2,2}(\Omega;\wedge^{0,q})$. Since $S_q^*g=N_q\dbar g$ for every $g\in C_c^\infty(\Omega;\wedge^{0,q-1})$, we have $\Box_qS_q^*g=\dbar g$ as distributions.
    
    Take a $\chi\in C_c^\infty(\Omega;\R)$ such that $\int\chi\neq0$. Consider $g:=\chi d\bar z_2\wedge\dots\wedge d\bar z_q\in C_c^\infty(\Omega;\wedge^{0,q-1})$ and $h:=z_1d\bar z_1\wedge\dots\wedge d\bar z_q\in H^{2,2}(\Omega;\wedge^{0,q})$. Applying $\dbar^*$ formally, we have $\dbar^* h=-d\bar z_2\wedge\dots\wedge d\bar z_q\in \Co^\infty(\Omega;\wedge^{0,q-1})$ and $\Box_qh=0$. Therefore
    \begin{equation*}
        0=(S_q^*g,\Box_qh)_{H_0^{2,2},H^{-2,2}}=(\Box_qS_q^*g,h)_{L^2,L^2}=(\Box_qN_q\dbar g,h)_{L^2,L^2}
        =(\dbar g,h)_{C_c^\infty,\Ds'}=(g,\dbar^*h)_{C_c^\infty,\Ds'}=-\textstyle\int\chi\neq0.
    \end{equation*}        
    This gives the contradiction.
\end{proof}
\begin{rem}
    Although the boundary value problem is not well-posed on distributions in $\Omega$, it is still possible to talk about the problems on $\overline\Omega$, namely for distributions $f$ and $u$ on $\C^n$ supported in $\overline\Omega$. We refer the reader to \cite{NegativeBVPBook}.
\end{rem}

\section{The Counterexample for $\frac12$-estimates on Hardy Spaces}\label{Section::CounterHardy}

In this section we show that the $\dbar$ problem does not admit any solution operator with boundedness $H^{s,p}\to H^{s+1/2,p}$ if $p<1$. To make the result more general we recall the Besov space here.

\begin{defn}\label{Defn::Hardy::Besov}
    Let $\phi=(\phi_j)_{j=0}^\infty\subset\Ss(\R^N)$ be the family in Definition~\ref{Defn::Space::TLSpace}. For $0<p,q\le\infty$ and $s\in\R$, we define the Besov norm $\|\cdot\|_{\Bs_{pq}^s(\phi)}$ as 
    \begin{equation*}
            \|f\|_{\Bs_{pq}^s(\phi)}:=\|(2^{js}\phi_j\ast f)_{j=0}^\infty\|_{\ell^q(\N;L^p(\R^N))}=\bigg(\sum_{j=0}^\infty\Big(\int_{\R^N}|2^{js}\phi_j\ast f(x)|^pdx\Big)^\frac qp\bigg)^\frac1q.
    \end{equation*}
    Here for $q=\infty$ we replace $\ell^q$ sum by supremum, for $p=\infty$ we replace the $L^p$ integral by essential supremum.

    We define $\Bs_{pq}^s(\R^N)$ with its norm given by a fixed choice of $\phi$.

    For an arbitrary open subset $U\subseteq\R^N$ we define 
    \begin{equation*}
        \Bs_{pq}^s(U):=\{\tilde f|_U:\tilde f\in\Bs_{pq}^s(\R^N)\},\qquad\text{with}\quad\|f\|_{\Bs_{pq}^s(U)}:=\inf\{\|\tilde f\|_{\Bs_{pq}^s(\R^N)}:\tilde f|_U=f\}.
    \end{equation*}
\end{defn}

\begin{rem}\label{Rmk::Hardy::Besov}
\begin{enumerate}[(i)]
    \item Different choices of $\phi$ also results in equivalent Besov norms. See \cite[Proposition~2.3.2 (i)]{TriebelTheoryOfFunctionSpacesI}.
    \item\label{Item::Hardy::Besov::Embed} When $p=q$ we have equality $\|f\|_{\Fs_{pp}^s(\phi)}=\|f\|_{\Bs_{pp}^s(\phi)}$. But when $p\neq q$, $\Bs_{pq}^s(\R^N)\neq\Fs_{pq}^s(\R^N)$ see \cite[Theorem~2.3.9]{TriebelTheoryOfFunctionSpacesI}. In general we have embeddings $\Bs_{p,\min(p,q)}^s(\R^N)\subset\Fs_{pq}^s(\R^N)\subset\Bs_{p,\max(p,q)}^s(\R^N)$, see \cite[Proposition~2.3.2]{TriebelTheoryOfFunctionSpacesI}. In our case here we will use the fact that $\Fs_{p\infty}^s(U)\subset\Bs_{p \infty}^s(U)$, which follows from $\Fs_{p\infty}^s(\R^N)\subset\Bs_{p \infty}^s(\R^N)$ and taking restriction on domain.
\end{enumerate}
\end{rem}
\begin{rem}
    The results for Besov spaces are generically easier than those for Triebel-Lizorkin spaces because we have real interpolations,
    \begin{equation}\label{Eqn::Hardy::RealInterpo}
        (\Fs_{pq_0}^{s_0}(\R^N),\Fs_{pq_1}^{s_1}(\R^N))_{\theta,q}=\Bs_{p, q}^{(1-\theta)s_0+\theta s_1}(\R^N),\qquad p,q_0,q_1,q\in(0,\infty],\quad\theta\in(0,1),\quad s_0\neq s_1.
    \end{equation}
    The case $p<\infty$ follows from e.g. \cite[Theorem~2.4.2 (ii)]{TriebelTheoryOfFunctionSpacesI}. The case $p=\infty$ is given by the following inclusions: we use $s_\theta=(1-\theta)s_0+\theta s_1$,
    \begin{equation*}
        \Bs_{\infty q}^{s_\theta}(\R^N)=(\Bs_{\infty q_0}^{s_0}(\R^N),\Bs_{\infty q_1}^{s_1}(\R^N))_{\theta,q}\subseteq(\Fs_{\infty q_0}^{s_0}(\R^N),\Fs_{\infty q_1}^{s_1}(\R^N))_{\theta,q}\subseteq (\Bs_{\infty \infty}^{s_0}(\R^N),\Bs_{\infty\infty}^{s_1})_{\theta,q}=\Bs_{\infty q}^{s_\theta}(\R^N).
    \end{equation*}
    Here both real interpolations identity of Besov spaces follow from \cite[Theorem~2.4.2 (i)]{TriebelTheoryOfFunctionSpacesI}, the middle two inclusions follow from \cite[Theorem~2.9 (ii)]{TriebelTheoryOfFunctionSpacesIV}.

    For a bounded Lipschitz domain $U\subset\R^N$ we have linear extension operator $\Ec:\Ss'(U)\to\Ss'(\R^N)$ such that $\Ec:\Fs_{pq}^s(U)\to\Fs_{pq}^s(\R^N)$ is bounded for all $p,q\in(0,\infty]$ and $s\in\R$. This can be done by taking standard partition of unity argument to \eqref{Eqn::Space::ExtOp}, see e.g. \cite[Lemma~6.3]{ShiYaoExt} and also \cite[Theorem~3.6]{GoncalvesHaroskeSkrzypczak}. So by e.g. \cite[Theorem~1.2.4]{TriebelInterpolation} we have $(\Fs_{pq_0}^{s_0}(U),\Fs_{pq_1}^{s_1}(U))_{\theta,q}=\Bs_{p, q}^{(1-\theta)s_0+\theta s_1}(U)$ for the same indices in \eqref{Eqn::Hardy::RealInterpo}

    Therefore by definition of interpolation spaces (see e.g. \cite[Section 1.2.2 and 1.3]{TriebelInterpolation}), for the same operators $\Hc_q$ and $\Pc$ in Theorem~\ref{Thm::GenThm} we have, for every $p\in[1,\infty]$, $r\in(0,\infty]$, $s\in\R$ and $q=1,\dots,n-1$,
    \begin{align*}
        \Pc&:\Bs_{pr}^s(\Omega)\to\Bs_{pr}^s(\Omega);
        \\
        \Hc_q&:\Bs_{pr}^s(\Omega;\wedge^{0,q})\to \Bs_{pr}^{s+1/2}(\Omega;\wedge^{0,q-1});
        \\
        \Hc_q&:\Bs_{pr}^s(\Omega;\wedge^{0,q})\to \Bs_{\frac{(2n+2)p}{2n+2-p},r}^s(\Omega;\wedge^{0,q-1}),&\text{provided }p\le 2n+2;
        \\
        \Hc_n&:\Bs_{pr}^s(\Omega;\wedge^{0,n})\to \Bs_{pr}^{s+1}(\Omega;\wedge^{0,n-1}).
    \end{align*}
    We omit the details to reader.
\end{rem}

The main focus of this section is to prove the following:
\begin{thm}\label{Thm::Hardy::CountEx}
    Fix $n\ge2$ and $m\in[2,\infty)$. Let $\chi\in C_c^\infty((\frac12,1))$ be such that $\int t\chi(t)dt=1$. We define
    \begin{gather}
    \label{Eqn::Hardy::Omegam}\Omega_m:=\{(z_1,z'',z_n)\in\C^n:|z_1-1|^2+|z''|^2+|z_n|^m<1\};
    \\
    \label{Eqn::Hardy::Exmpf}
        f_\eps(z):=\chi\big(|z_n|/\eps\big)e^{-z_1/\eps}d\bar z_n,\qquad 0<\eps\le1,\qquad z\in\C^n.
    \end{gather}

    Then for every $p\in(0,1)$, $r\in(0,\infty]$, $s,t\in\R$, there exists a $C>0$ such that for every $0<\eps<1/5$,
    \begin{equation}\label{Eqn::Hardy::CountExEst}
        \|f_\eps\|_{\Fs_{pr}^s(\Omega_m)}\le C\eps^{(n+3/2)/p-s}\quad\text{but}\quad\inf\big\{\|u_\eps\|_{\Bs_{p\infty}^{t}(\Omega_m)}:\dbar u_\eps=f_\eps\big\}\ge\tfrac1C\eps^{2-\frac1m-t+\frac{n-1/2+2/m}p}.
    \end{equation}
\end{thm}

The estimates \eqref{Eqn::Hardy::CountExEst} rely on the following two facts:
\begin{prop}[{\cite[Proposition~2.3.1/1]{EdmundsTriebelFunctionSpaces}}]\label{Prop::Hardy::Dila}
    Let $p\in(0,\infty)$, $r\in(0,\infty]$ and $s>N(\frac1{\min(p,1)}-1)$. There is a $C=C_{N,p,r,s}>0$ such that
    \begin{equation*}
        \|g(\cdot/\eps)\|_{\Fs_{pr}^s(\R^N)}\le C\eps^{N/p-s}\|g\|_{\Fs_{pr}^s(\R^N)},\qquad\text{for all }g\in\Fs_{pr}^s(\R^N)\text{ and }0<\eps<1.
    \end{equation*}
\end{prop}

\begin{prop}\label{Prop::Hardy::BddBelow}
    Let $\mu\in\Ss(\R^N)$, we denote $\mu_\delta(x):=\delta^{-N}\mu(x/\delta)$ for $0<\delta\le1$. Then for every $s>0$ and $0<p<\infty$ there is a $C=C_{N,\mu,s,p}>0$ such that
    \begin{equation*}
        \sup_{0<\delta<1}\delta^s\|\mu_\delta\ast g\|_{L^p(\R^N)}\le C\|g\|_{\Bs_{p\infty}^{-s}(\R^N)},\qquad g\in\Bs_{p\infty}^{-s}(\R^N)
    \end{equation*}
\end{prop}
\begin{proof}
    The result follows from \cite[Theorem~5.1 (ii)]{Bui}. In the reference we take $q=\infty$, $w=\1$, $a=1$ and $\alpha=-s<0$, so by the assumptions in \cite[Theorem~3.1]{Bui} no moment vanishing condition is needed for $\mu$. The statement \cite[(5.4)]{Bui} uses local Hardy space $h^p$ on the left hand inequality, see \eqref{Eqn::Space::DefnFractionalHardy} for the definition of $\|\cdot\|_{h^p}$. The notation $\phi_\delta(x)=\delta^{-N}\phi(x/\delta)$ in the reference can be found in \cite[Page 220]{Bui} in the definition of global Hardy space $H^p$. In fact we have $h^p=\Fs_{p2}^0$ and $H^p=\dot\Fs_{p2}^0$.

    Now take a $\phi$ such that $\int\phi=1$ to define the $h^p$ norm, then for every $f\in\Co^\infty(\R^N)$ we have $$\|f\|_{L^p}=\big\|\lim_{\delta\to0}\phi_\delta \ast f\big\|_{L^p}\le\big\|\sup_{0<\delta<1}|\phi_\delta\ast f|\big\|_{L^p}=\|f\|_{h^p}.$$ Now for $g\in\Bs_{p\infty}^{-s}(\R^N)$, we have $\mu_\delta\ast g\in\Co^\infty(\R^N)$. Therefore applying \cite[Theorem~5.1 (ii)]{Bui} $\sup_{0<\delta<1}\delta^s\|\mu_\delta\ast g\|_{L^p(\R^N)}\lesssim\sup_{0<\delta<1}\delta^s\|\mu_\delta\ast g\|_{h^p(\R^N)}\lesssim\|g\|_{\Bs_{p\infty}^{-s}(\R^N)}$.
\end{proof}

\begin{proof}[Proof of Theorem~\ref{Thm::Hardy::CountEx}]
    Choose $\theta\in \Co^\infty(\R;[0,1])$ such that $\theta|_{(-\infty,-1]}=0$ and $\theta|_{[0,\infty)}=1$. We define
\begin{equation}\label{Eqn::Hardy::DefGEps}
    g^\eps(z)=g^\eps(x_1+iy_1,z'',z_n):=\chi(|z_n|)e^{-z_1}\theta\big(2x_1-\eps y_1^2-\eps|z''|^2-\eps^{m-1}|z_n|^m\big)d\bar z_n,\qquad\eps\in(0,1].
\end{equation}
Since $\Omega_m\subset\{x_1>\frac12(y_1^2+|z''|^2+|z_n|^m)\}$, we see that $\theta\big(2x_1-\eps y_1^2-\eps|z''|^2-\eps^{m-1}|z_n|^m\big)\equiv1$ if $z\in\eps^{-1}\Omega_m$. Therefore $g^\eps(\cdot/\eps)|_{\Omega_m}=f_\eps|_{\Omega_m}$, where dilation is applied to the coefficient, with $d\bar z_n$ fixed. 

Next we claim that for every $p\in(0,\infty)$, $r\in(0,\infty]$ and $s\in\R$,
\begin{equation}\label{Eqn::Hardy::ClaimGEps}
    \|g^\eps\|_{\Fs_{pr}^s(\C^n)}
    \lesssim_{p,r,s}\eps^{-(2n-3)/(2p)},\qquad\text{for every }\eps\in(0,1]
\end{equation}

Indeed, by the support condition of $\chi$ and $\theta$ we have
\begin{equation}\label{Eqn::Hardy::SuppGEps}
    \supp g^\eps\subset\big\{(x_1+iy_1,z'',z_n):x_1\ge-\tfrac12,\quad\eps(y_1^2+|z''|^2)\le1+2x_1,\quad\tfrac12\le|z_n|\le1\big\}=:K_\eps.
\end{equation}

Indeed, $\chi,\theta\in\Co^\infty(\R)$, applying chain rules to \eqref{Eqn::Hardy::DefGEps} along with the support controls \eqref{Eqn::Hardy::SuppGEps} we see that for every multiindex $\alpha$,
\begin{equation*}
        |D^\alpha g^\eps(z)|\lesssim_{\chi,\theta,\alpha}(1+\eps |y_1|+\eps|z''|)^{|\alpha|}\1_{K_\eps}(z)e^{-x_1}\lesssim_{\chi,\theta,\alpha}(1+x_1)^{\frac{|\alpha|}2}e^{-x_1}\1_{K_\eps}(z)\lesssim_\alpha e^{-\frac{x_1}2}\1_{K_\eps}(z).
\end{equation*}

We conclude that $|D^\alpha g^\eps(z)|\lesssim_\alpha e^{-\frac{x_1}2}\1_{K_\eps}(z)$ for every $\eps\in(0,1]$ and $z\in\C^n$.

Now let $\phi=(\phi_j)_{j=0}^\infty\subset\Ss(\C^n)$ be the family in Definition~\ref{Defn::Space::TLSpace} that defines the Triebel-Lizorkin norm. By definition for every $M\in\Z_+$ we have $|\phi_0(z)|\lesssim_M (1+|z|)^{-M}$ and $|\Delta^{-M}\phi_j(z)|\lesssim_M 2^{(2n-2M)j}(1+2^j|z|)^{-M}$ for all $j\ge1$. Therefore whenever $M>4n$, we have for $j\ge1$,
\begin{align*}
    |\phi_j\ast g^\eps(z)|=|\Delta^{-M}\phi_j\ast\Delta^M g^\eps(z)|\lesssim_M\int_{K_\eps}\frac{2^{(2n-2M)j}e^{-\frac{\re w_1}2}d\Vol_w}{(1+2^j|z-w|)^M}\lesssim_M\frac{2^{-2jM}}{\big(1+|x_1|+\sqrt\eps(|y_1|+|z''|)+|z_n|\big)^{ M/2}}.
\end{align*}
Similarly the estimate $|\phi_0\ast g^\eps(z)|\lesssim_M\int_{K_\eps}e^{-\frac{\re w_1}2}(1+|z-w|)^{-M}d\Vol_w\lesssim_M\big(1+|x_1|+\sqrt\eps(|y_1|+|z''|)+|z_n|\big)^{- M/2}$ holds as well.
So by setting $M>\max(4n,2(2n+1)/p+|s|)$ we get
\begin{align*}
    \|g^\eps\|_{\Fs_{pr}^s(\C^n)}=\Big\|\Big(\sum_{j=0}^\infty 2^{jsr}|\phi_j\ast g^\eps|^r\Big)^{1/r}\Big\|_{L^p}\lesssim_p\Big\|z\mapsto\big(1+|x_1|+\sqrt\eps(|y_1|+|z''|)+|z_n|\big)^{-\frac{2n+1}p}\Big\|_{L^p}\lesssim_p\eps^{-\frac{2n-3}{2p}}.
\end{align*}
This proves the claim \eqref{Eqn::Hardy::ClaimGEps}.

Now by definition,  $f_\eps(z)=(-\eps)^L(D_{x_1}^Lf_\eps)(z)$ for every $L\in\N$. Take $L>2n(\min(p,1)^{-1}-1)-s$, by Lemma~\ref{Lem::Space::EqvNormDer} and Proposition~\ref{Prop::Hardy::Dila}, \eqref{Eqn::Hardy::ClaimGEps}, and the fact that $g^\eps(\cdot/\eps)|_{\Omega_m}=f_\eps|_{\Omega_m}$,
    \begin{equation*}
        \|f_\eps\|_{\Fs_{pr}^s(\Omega_m)}=\eps^L\|D_{x_1}^L f_\eps\|_{\Fs_{pr}^s(\Omega_m)}\lesssim\eps^L\|f_\eps\|_{\Fs_{pr}^{s+L}(\Omega_m)}\le\eps^L\|g^\eps(\tfrac\cdot\eps)\|_{\Fs_{pr}^{s+L}}\lesssim\eps^L\eps^{\frac{2n}p-s-L}\|g^\eps\|_{\Fs_{pr}^{s+L}}\lesssim\eps^{\frac{n+3/2}p-s}.
    \end{equation*}
    This proves the first inequality in \eqref{Eqn::Hardy::CountExEst}.

    Set $\displaystyle v_\eps(z):=\frac{2e^{-z_1/\eps}}{z_n}\int_0^{|z_n|}r\chi\big(\tfrac r\eps\big)dr$ for $z\in\Omega_m$, where $v_\eps(z)=0$ defined naturally for $z_n=0$. By direct computation $\dbar v_\eps=f_\eps$ in $\Omega_m$. Now let $u_\eps$ be such that $\dbar u_\eps=f_\eps$, $u_\eps-v_\eps$ is a holomorphic function in $\Omega_m$. Therefore by Cauchy's theorem, for every $z'=(z_1,z'')$ such that $|z'-(1,0,\dots,0)|<1-\eps$, and for every  $R>\eps$ such that $\{(z',w):|w|=R\}\subset\Omega_m$, and for every $L\in\N$,
    \begin{equation*}
        \int_0^{2\pi}
\frac{D_{x_1}^L u_\eps(z',Re^{i\theta})}{2\pi}e^{i\theta}d\theta=\int_{|w|=R}\frac{D_{x_1}^Lv_\eps(z',w)}{2\pi iR}dw=
\frac{e^{-z_1/\eps}}{(-\eps)^{L}i\pi R}
\int_0^Rr\chi\big(\tfrac r\eps\big)dr\int_{|w|=R}\frac{dw}w=\frac{2e^{-z_1/\eps}}{(-\eps)^{L-2}R}.
    \end{equation*}
    Here $\int_0^Rr\chi\big(\tfrac r\eps\big)dr=\eps^2\int_0^{R/\eps} t\chi(t)dt=\eps^2$ by the assumption of $\chi$.

    Set $A_\eps:=\{z=(x_1+iy_1,z'',z_n):4\eps<x_1<5\eps,\ y_1^2+|z''|^2<\eps,\ \eps^\frac1m<|z_n|<(2\eps)^\frac1m\}$, we have $|A_\eps|\approx\eps^{n-1/2+2/m}$ and $A_\eps+B(0,\frac{\eps}{m})\subset\Omega_m\cap\{|z_n|>\eps\}$ so $\dbar u_\eps|_{A_\eps+B(0,\frac{\eps}{m})}=0$. Since $u_\eps-v_\eps$ is holomorphic, $[w\mapsto wD_{x_1}^Lu_\eps(z',w)]$ on $|w|>\eps$ has a holomorphic continuation into the disk. By the mean-value formula and subharmonicity of its absolute value raised to $p$ (since $p>0$), we get
        \begin{align*}
        \|D_{x_1}^L u_\eps\|_{L^p(A_\eps)}^p&=\int_{\substack{4\eps<x_1<5\eps\\ y_1^2+|z''|^2<\eps}}\int_{\eps^{1/m}}^{(2\eps)^{1/m}}\Big(\int_0^{2\pi}|re^{i\theta}D_{x_1}^L u_\eps(z',re^{i\theta})|^pd\theta\Big)r^{1-p}drd\Vol_{(z_1,z'')}
        \\
        &\ge\int_{\substack{4\eps<x_1<5\eps\\ y_1^2+|z''|^2<\eps}}d\Vol_{(z_1,z'')}\int_{\eps^{1/m}}^{(2\eps)^{1/m}}(2\pi)^{1-p}\Big|\int_0^{2\pi}re^{i\theta}D_{x_1}^L u_\eps(z',re^{i\theta})d\theta\Big|^p r^{1-p}dr
        \\
        &=2^{p+1}\pi|\sqrt\eps\B^{2n-3}|\int_{4\eps}^{5\eps}e^{-px_1/\eps}dx_1\int_{\eps^{1/m}}^{(2\eps)^{1/m}}\frac{r^{1-p}dr}{\eps^{(L-2)p}}\approx\eps^{n-\frac12+(2-L)p+\frac{2-p}m}.
    \end{align*}

    Now take a radial function $\mu\in C_c^\infty(\frac1{m}\B^{2n})$ such that $\int\mu=1$, and take $\mu_\delta(z)=\delta^{-2n}\mu(z/\delta)$ for $0<\delta<1$. Note that $u_\eps$ is holomorphic in $A_\eps+B(0,\frac{\eps}{m})$, which is in particular harmonic, so by spherical average formula, $u_\eps=\mu_\eps\ast u_\eps$ on $A_\eps$. Take $L>t$, by Proposition~\ref{Prop::Hardy::BddBelow},
    \begin{align*}
        \|u_\eps\|_{\Bs_{p\infty}^t(\Omega_m)}&=\inf_{\tilde u_\eps|_{\Omega_m}=u_\eps}\|\tilde u_\eps\|_{\Bs_{p\infty}^t(\C^n)}\gtrsim\inf_{\tilde u_\eps|_{\Omega_m}=u_\eps}\|D_{x_1}^L\tilde u_\eps\|_{\Bs_{p\infty}^{t-L}(\C^n)}
        \\
        &\gtrsim \inf_{\tilde u_\eps|_{\Omega_m}=u_\eps}\eps^{L-t}\|\mu_\eps\ast D_{x_1}^L\tilde u_\eps\|_{L^p(\C^n)}\ge \eps^{L-t}\inf_{\tilde u_\eps|_{\Omega_m}=u_\eps}\|\mu_\eps\ast D_{x_1}^L\tilde u_\eps\|_{L^p(A_\eps)}
        \\
        &=\eps^{L-t}\|\mu_\eps\ast D_{x_1}^Lu_\eps\|_{L^p(A_\eps)}=\eps^{L-t}\|D_{x_1}^Lu_\eps\|_{L^p(A_\eps)}\gtrsim\eps^{2-\frac1m-t+\frac{n-1/2+2/m}{p}}.
    \end{align*}
    This proves the second inequality in \eqref{Eqn::Hardy::CountExEst} and completes the proof.
\end{proof}

\begin{cor}
    Let $\Omega_m$ be from Theorem~\ref{Thm::Hardy::CountEx} with $m\in[2,\infty)$. For every $0<p<1$ and $s\in\R$, there exists a $\dbar$ closed $(0,1)$ form $f\in\bigcap_{r>0}\Fs_{pr}^s(\Omega_m;\wedge^{0,1})$ such that there is no $u\in\Fs_{p\infty}^{s+\frac1m}(\Omega_m)$ for which $\dbar u=f$.
\end{cor}
\begin{proof}
    This follows from Theorem~\ref{Thm::Hardy::CountEx} along with Baire's category argument. Indeed, let $(f_\eps)_{0<\eps<1}$ be the smooth $\dbar$ closed $(0,1)$ forms in Theorem~\ref{Thm::Hardy::CountEx}. By Remark~\ref{Rmk::Hardy::Besov}~\ref{Item::Hardy::Besov::Embed} and \eqref{Eqn::Hardy::CountExEst}, for every $0<\eta<2(1-\frac1m)(\frac1p-1)$, 
    \begin{equation}\label{Eqn::Hardy::QuotTmp}
        \|f_\eps\|_{\Fs_{p\infty}^{s+\eta}(\Omega_m)}\big/\inf\big\{\|u_\eps\|_{\Fs_{p\infty}^{s+1/m}(\Omega_m)}:\dbar u_\eps=f_\eps\big\}\lesssim_{m,p,r,s}\eps^{2(1-\frac1m)(\frac1p-1)-\eta}\to0\quad\text{as}\ \eps\to0.
    \end{equation}
        Fix a $0<\eta<(1-\frac1m)(\frac1p-1)$, we have $\Fs_{p\infty}^{s+\eta}(\Omega_m)\subset\bigcap_{r>0}\Fs_{pr}^s(\Omega_m)$ by Lemma~\ref{Lem::Space::TLEmbed}. So it suffices to find a $\dbar $ closed $f\in \Fs_{p\infty}^{s+\eta}(\Omega_m;\wedge^{0,1})$ that has no $\dbar$ preimage in $\Fs_{p\infty}^{s+1/m}(\Omega_m)$.

    Now for $f_{2^{-j}}$ be the functions in \eqref{Eqn::Hardy::Exmpf} with $\eps=2^{-j}$, we define
    \begin{equation*}
        f^\sigma:=\sum_{j=0}^\infty \sigma_j 2^{j(\frac{n+3/2}p-s-2\eta)}f_{2^{-j}},\qquad\text{for }\sigma=(\sigma_j)_{j=0}^\infty\in\{0,1\}^\N.
    \end{equation*}
        We endow $\{0,1\}^\N$ with standard product topology which is completely metrizable by the metric $d(\sigma,\sigma'):=\sum_j2^{-j}|\sigma_j-\sigma'_j|$. Note that $[\sigma\mapsto f^\sigma]:\{0,1\}^\N\to\Fs_{p\infty}^{s+\eta}(\Omega_m;\wedge^{0,1})$ is continuous because for $C$ in \eqref{Eqn::Hardy::CountExEst} enlarged to include $j=0,1,2$,
        \begin{align*}
        \|f^\sigma-f^{\sigma'}\|_{\Fs_{p\infty}^{s+\eta}}^p&=\Big\|\sum_{j=0}^\infty(\sigma_j-\sigma'_j) 2^{j(\frac{n+3/2}p-s-2\eta)}f_{2^{-j}}\Big\|_{\Fs_{p\infty}^{s+\eta}}^p\le\sum_{\sigma_j\neq\sigma'_j}\|2^{j(\frac{n+3/2}p-s-2\eta)}f_{2^{-j}}\|_{\Fs_{p\infty}^{s+\eta}}^p
        \\
        &\le C^p\sum_{\sigma_j\neq\sigma_j'} 2^{-j\eta p}\le\frac{C^p2^{-\eta p\cdot\min\{j:\sigma_j\neq\sigma'_j\}}}{1-2^{-\eta p}}\le\frac{C^pd(\sigma,\sigma')^{\eta p}}{1-2^{-\eta p}}.
    \end{align*}
    In particular $\{f^\sigma:\sigma\in\{0,1\}^\N\}\subset\Fs_{p\infty}^{s+\eta}(\Omega_m;\wedge^{0,1})$ is a bounded set.
    
    Fix a choice of $\phi$ in Definition~\ref{Defn::Space::TLSpace} to define the norm. For each $N\in\R_+$, we set $$A_N:=\{\sigma\in\{0,1\}^\N:\text{there exists }u\in\Fs_{p\infty}^{s+1/m}(\Omega_m)\text{ such that }\dbar u=f^\sigma\text{ and }\|u\|_{\Fs_{p\infty}^{s+1/m}}\le N\}.$$
    Let $e^k\in\{0,1\}^\N$ denotes the sequence such that $e^k_j=1$ if $j=k$ and $e^k_j=0$ for $j\neq k$. From \eqref{Eqn::Hardy::QuotTmp}
    \begin{equation}\label{Eqn::Hardy::PreImageTmp}
        \inf\{\|u\|_{\Fs_{p\infty}^{s+1/m}(\Omega_m)}:\dbar u=f^{e^k}\}\gtrsim 2^{k(2(1-\frac1m)(\frac1p-1)-2\eta)}\to\infty\quad\text{as }k\to\infty.
    \end{equation}

    We claim that there is a $C'>0$ such that $\overline {A_N}\subset A_{C'N}$ and $A_N$ have empty interior for all $N$. Hence $A_N$ are nowhere dense.

    Let $(\sigma^l)_{l=1}^\infty\subset A_N$ be a convergent sequence and denote $\sigma^0=\lim_{l\to\infty}\sigma^l$. Let $u_l\in\Fs_{p\infty}^{s+1/m}(\Omega_m)$ be such that $\dbar u_l=f^{\sigma^l}$ and $\|u_l\|_{\Fs_{p\infty}^{s+1/m}}\le N$. So $(u_l)_{l=1}^\infty\subset\Fs_{p\infty}^{s+1/m}(\Omega_m)$ is a bounded sequence,  which admits a convergent subsequence in $\Fs_{p\infty}^s(\Omega_m)$ because $\Fs_{p\infty}^{s+1/m}(\Omega_m)\hookrightarrow \Fs_{p\infty}^{s}(\Omega_m)$ is compact embedding (see e.g. \cite[Remark~1.98]{TriebelTheoryOfFunctionSpacesIII}). Let $u_0=\lim_{k\to\infty}u_{l_k}\in \Fs_{p\infty}^s(\Omega_m)$ be a subsequential limit, so $\dbar u_0=(\Fs_{p\infty}^{s-1}\text{-})\lim_{k\to\infty}\dbar u_{l_k}=(\Fs_{p\infty}^{s-1}\text{-})\lim_{k\to\infty} f^{\sigma^{l_k}}=f^{\sigma^0}$.
    
    On the other hand by Fatou property of $\Fs_{p\infty}^{s+1/m}(\Omega_m)$, we have $u_0\in\Fs_{p\infty}^{s+1/m}(\Omega_m)$ and $\|u_0\|_{\Fs_{p\infty}^{s+1/m}}\le C'\sup_k\|u_{l_k}\|_{\Fs_{p\infty}^{s+1/m}}\le C'N$ for some $C'$ depending only on $p,s,m$ and the choice of $\Fs_{p\infty}^{s+1/m}$-norm. Indeed, take an extension operator $\Ec:\Ss'(\Omega_m)\to\Ss'(\C^n)$ such that $\Ec:\Fs_{p\infty}^t(\Omega_m)\to\Fs_{p\infty}^t(\C^n)$ is bounded for $t\in\{s,s+\frac1m\}$ (see e.g. \cite[Lemma~6.3]{ShiYaoExt}), then $\Ec u_0=(\Ss'\text{-})\lim_{k\to\infty}\Ec u_{l_k}$ and $\sup_k\|\Ec u_{l_k}\|_{\Fs_{p\infty}^{s+1/m}(\C^n)}\lesssim N$. By \cite[Page 360]{TriebelStructureFunctions} for the Fatou property of $\Fs_{p\infty}^{s+1/m}(\C^n)$, we have $\Ec u_0\in\Fs_{p\infty}^{s+1/m}(\C^n)$ and $$\|u_0\|_{\Fs_{p\infty}^{s+1/m}(\Omega_m)}\le\|\Ec u_0\|_{\Fs_{p\infty}^{s+1/m}(\C^n)}\lesssim\sup_k\|\Ec u_k\|_{\Fs_{p\infty}^{s+1/m}(\C^n)}\lesssim\sup_k\|u_k\|_{\Fs_{p\infty}^{s+1/m}(\Omega_m)}\le N,$$ where the implied constants do not depend on $(\sigma^l)_l$ or $N$.
    Hence $\sigma^0\in A_{C'N}$ where $C'$ does not depend on $(\sigma^l)_l$ or $N$. This proves the property $\overline{A_N}\subset A_{C'N}$.

    Suppose in contrast $A_N$ has nonempty interior, then there is a $\sigma\in A_N$ and $K\in\N$ such that $\{\sigma'\in\{0,1\}^\N:\sigma'_j=\sigma_j\text{ for }j=0,\dots,K\}\subseteq A_N$. For convenience we assume $\sigma_j=0$ for all $j>K$ since it does not change the open subset. By \eqref{Eqn::Hardy::PreImageTmp} there is an integer $L>K$ such that $e^L\notin A_{3^{1/p}N}$. So if $v\in\Fs_{p\infty}^{s+1/m}(\Omega_m)$ satisfies $\dbar v=f^{\sigma +e^L}$, we have 
        \begin{equation*}
        \|v\|_{\Fs_{p\infty}^{s+1/m}}^p\ge\inf_{\dbar u=f^{e^L}}\|u\|_{\Fs_{p\infty}^{s+1/m}}^p-\inf_{\dbar u=f^\sigma}\|u\|_{\Fs_{p\infty}^{s+1/m}}^p\ge (3^{1/p}N)^p-N^p=2N^p>N^p.
    \end{equation*}
    We conclude that $\sigma+e^L\notin A_N$ contradicts the assumption $\{\sigma'\in\{0,1\}^\N:\sigma'_j=\sigma_j\text{ for }j\le K\}\subseteq A_N$.

    So now $\overline{A_N}^\circ\subset A_{C'N}^\circ=\varnothing$ for all $N$. By Baire's theorem $\{0,1\}^\N\backslash\bigcup_{N=1}^\infty A_N\neq\varnothing$. Take a such $\sigma\in\{0,1\}^\N$, then the $(0,1)$ form $f^\sigma$ is as desired.
\end{proof}
\subsection*{Acknowledgement} The author would like to thank Xianghong Gong for the valuable support and comments. The author would also like to thank Chian Yeong Chuah and Jan Lang for the discussion of the spectral theory.

\subsection*{AI disclaimer} GPT 6-Astra and GPT 6.1-Sol were used to fix typos, and simplify the proofs of Propositions~\ref{Prop::PM::AbsPM}, \ref{Prop::PM::ConcretePM} and Lemma~\ref{Lem::Space::IllPosed}. GPT 6-Astra was also used to give inspiration for Theorem~\ref{Thm::Hardy::CountEx}.
The author has checked the resulting arguments and takes full responsibility for the mathematical content of this paper.

\bibliographystyle{amsalpha}
\bibliography{reference} 

@book {TriebelTheoryOfFunctionSpacesI,
    AUTHOR = {Triebel, Hans},
     TITLE = {Theory of function spaces},
    SERIES = {Modern Birkh\"{a}user Classics},
      NOTE = {Reprint of 1983 edition [MR0730762],
              Also published in 1983 by Birkh\"{a}user Verlag [MR0781540]},
 PUBLISHER = {Birkh\"{a}user/Springer Basel AG, Basel},
      YEAR = {2010},
     PAGES = {285},
      ISBN = {978-3-0346-0415-4; 978-3-0346-0416-1},
   MRCLASS = {01A75},
  MRNUMBER = {3024598},
}

@book {TriebelTheoryOfFunctionSpacesII,
    AUTHOR = {Triebel, Hans},
     TITLE = {Theory of function spaces. {II}},
    SERIES = {Monographs in Mathematics},
    VOLUME = {84},
 PUBLISHER = {Birkh\"{a}user Verlag, Basel},
      YEAR = {1992},
     PAGES = {viii+370},
      ISBN = {3-7643-2639-5},
   MRCLASS = {46Exx (46-02)},
  MRNUMBER = {1163193},
MRREVIEWER = {P. Szeptycki},
       DOI = {10.1007/978-3-0346-0419-2},
       URL = {https://doi-org.proxy.lib.ohio-state.edu/10.1007/978-3-0346-0419-2},
}

@book {TriebelTheoryOfFunctionSpacesIII,
    AUTHOR = {Triebel, Hans},
     TITLE = {Theory of function spaces. {III}},
    SERIES = {Monographs in Mathematics},
    VOLUME = {100},
 PUBLISHER = {Birkh\"{a}user Verlag, Basel},
      YEAR = {2006},
     PAGES = {xii+426},
      ISBN = {978-3-7643-7581-2; 3-7643-7581-7},
   MRCLASS = {46E35 (46-02)},
  MRNUMBER = {2250142},
MRREVIEWER = {Lubo\v{s} Pick},
}

@book {TriebelTheoryOfFunctionSpacesIV,
    AUTHOR = {Triebel, Hans},
     TITLE = {Theory of function spaces {IV}},
    SERIES = {Monographs in Mathematics},
    VOLUME = {107},
 PUBLISHER = {Birkh\"{a}user/Springer, Cham},
      YEAR = {2020},
     PAGES = {x+160},
      ISBN = {978-3-030-35890-7; 978-3-030-35891-4},
   MRCLASS = {46-02 (42B35 42C40 46E30 46E35)},
  MRNUMBER = {4298338},
       DOI = {10.1007/978-3-030-35891-4},
       URL = {https://doi-org.proxy.lib.ohio-state.edu/10.1007/978-3-030-35891-4},
}

@book {TriebelInterpolation,
    AUTHOR = {Triebel, Hans},
     TITLE = {Interpolation theory, function spaces, differential operators},
   EDITION = {Second},
 PUBLISHER = {Johann Ambrosius Barth, Heidelberg},
      YEAR = {1995},
     PAGES = {532},
      ISBN = {3-335-00420-5},
   MRCLASS = {46-02 (35Jxx 46E10 46E35 46M35)},
  MRNUMBER = {1328645},
}

@book {TriebelStructureFunctions,
    AUTHOR = {Triebel, Hans},
     TITLE = {The structure of functions},
    SERIES = {Monographs in Mathematics},
    VOLUME = {97},
 PUBLISHER = {Birkh\"{a}user Verlag, Basel},
      YEAR = {2001},
     PAGES = {xii+425},
      ISBN = {3-7643-6546-3},
   MRCLASS = {46E35 (35J60 41A17 42C40)},
  MRNUMBER = {1851996},
MRREVIEWER = {Leszek Skrzypczak},
       DOI = {10.1007/978-3-0348-8257-6},
       URL = {https://doi-org.proxy.lib.ohio-state.edu/10.1007/978-3-0348-8257-6},
}

@book {GrandpaRudin,
    AUTHOR = {Rudin, Walter},
     TITLE = {Functional analysis},
    SERIES = {International Series in Pure and Applied Mathematics},
   EDITION = {Second},
 PUBLISHER = {McGraw-Hill, Inc., New York},
      YEAR = {1991},
     PAGES = {xviii+424},
      ISBN = {0-07-054236-8},
   MRCLASS = {46-01 (47-01)},
  MRNUMBER = {1157815},
}

@article {RychkovExtension,
    AUTHOR = {Rychkov, Vyacheslav S.},
     TITLE = {On restrictions and extensions of the {B}esov and
              {T}riebel-{L}izorkin spaces with respect to {L}ipschitz
              domains},
   JOURNAL = {J. London Math. Soc. (2)},
  FJOURNAL = {Journal of the London Mathematical Society. Second Series},
    VOLUME = {60},
      YEAR = {1999},
    NUMBER = {1},
     PAGES = {237--257},
      ISSN = {0024-6107},
   MRCLASS = {46E35 (46F05)},
  MRNUMBER = {1721827},
MRREVIEWER = {J. Horv\'{a}th},
       DOI = {10.1112/S0024610799007723},
       URL = {https://doi-org.ezproxy.library.wisc.edu/10.1112/S0024610799007723},
}

@article {Seeley,
    AUTHOR = {Seeley, Robert Thomas},
     TITLE = {Extension of {$C^{\infty }$} functions defined in a half
              space},
   JOURNAL = {Proc. Amer. Math. Soc.},
  FJOURNAL = {Proceedings of the American Mathematical Society},
    VOLUME = {15},
      YEAR = {1964},
     PAGES = {625--626},
      ISSN = {0002-9939},
   MRCLASS = {46.38},
  MRNUMBER = {165392},
MRREVIEWER = {J. L. Lions},
       DOI = {10.2307/2034761},
       URL = {https://doi-org.proxy.lib.ohio-state.edu/10.2307/2034761},
}

@article {Bui,
    AUTHOR = {Bui, Huy-Qui and Paluszy\'{n}ski, Maciej and Taibleson, Mitchell H.},
     TITLE = {A maximal function characterization of weighted
              {B}esov-{L}ipschitz and {T}riebel-{L}izorkin spaces},
   JOURNAL = {Studia Math.},
  FJOURNAL = {Studia Mathematica},
    VOLUME = {119},
      YEAR = {1996},
    NUMBER = {3},
     PAGES = {219--246},
      ISSN = {0039-3223},
   MRCLASS = {46E35 (42B25)},
  MRNUMBER = {1397492},
MRREVIEWER = {Vladimir D. Stepanov},
}

@article {GongHolderSPsiCXC2,
    AUTHOR = {Gong, Xianghong},
     TITLE = {H\"{o}lder estimates for homotopy operators on strictly
              pseudoconvex domains with {$C^2$} boundary},
   JOURNAL = {Math. Ann.},
  FJOURNAL = {Mathematische Annalen},
    VOLUME = {374},
      YEAR = {2019},
    NUMBER = {1-2},
     PAGES = {841--880},
      ISSN = {0025-5831},
   MRCLASS = {32T15 (32W05)},
  MRNUMBER = {3961327},
MRREVIEWER = {Aaron J. Peterson},
       DOI = {10.1007/s00208-018-1693-9},
       URL = {https://doi-org.ezproxy.library.wisc.edu/10.1007/s00208-018-1693-9},
}

@book {RangeSCVBook,
    AUTHOR = {Range, R. Michael},
     TITLE = {Holomorphic functions and integral representations in several
              complex variables},
    SERIES = {Graduate Texts in Mathematics},
    VOLUME = {108},
 PUBLISHER = {Springer-Verlag, New York},
      YEAR = {1986},
     PAGES = {xx+386},
      ISBN = {0-387-96259-X},
   MRCLASS = {32-01},
  MRNUMBER = {847923},
MRREVIEWER = {Harold P. Boas},
       DOI = {10.1007/978-1-4757-1918-5},
       URL = {https://doi-org.proxy.lib.ohio-state.edu/10.1007/978-1-4757-1918-5},
}

@book {ChenShawBook,
    AUTHOR = {Chen, So-Chin and Shaw, Mei-Chi},
     TITLE = {Partial differential equations in several complex variables},
    SERIES = {AMS/IP Studies in Advanced Mathematics},
    VOLUME = {19},
 PUBLISHER = {American Mathematical Society, Providence, RI; International
              Press, Boston, MA},
      YEAR = {2001},
     PAGES = {xii+380},
      ISBN = {0-8218-1062-6},
   MRCLASS = {32Wxx (32-02 35A20 35N15 35S05 47N20)},
  MRNUMBER = {1800297},
MRREVIEWER = {Harold P. Boas},
       DOI = {10.1090/amsip/019},
}

@article{ShiYaoExt,
    AUTHOR = {Shi, Ziming and Yao, Liding},
     TITLE = {New estimates of {R}ychkov's universal extension operator for
              {L}ipschitz domains and some applications},
   JOURNAL = {Math. Nachr.},
  FJOURNAL = {Mathematische Nachrichten},
    VOLUME = {297},
      YEAR = {2024},
    NUMBER = {4},
     PAGES = {1407--1443},
      ISSN = {0025-584X,1522-2616},
   MRCLASS = {42 (47)},
  MRNUMBER = {4734977},
}

@article {ShiYaoCk,
    AUTHOR = {Shi, Ziming and Yao, Liding},
     TITLE = {A solution operator for the {$\overline \partial$} equation in
              {S}obolev spaces of negative index},
   JOURNAL = {Trans. Amer. Math. Soc.},
  FJOURNAL = {Transactions of the American Mathematical Society},
    VOLUME = {377},
      YEAR = {2024},
    NUMBER = {2},
     PAGES = {1111--1139},
      ISSN = {0002-9947,1088-6850},
   MRCLASS = {32A26 (32T15 46E35)},
  MRNUMBER = {4688544},
       DOI = {10.1090/tran/9066},
       URL = {https://doi.org/10.1090/tran/9066},
}

@article {GongHolderAq,
    AUTHOR = {Gong, Xianghong},
     TITLE = {On regularity of {$\overline\partial$}-solutions on {$a_q$}
              domains with {$C^2$} boundary in complex manifolds},
   JOURNAL = {Trans. Amer. Math. Soc.},
  FJOURNAL = {Transactions of the American Mathematical Society},
    VOLUME = {378},
      YEAR = {2025},
    NUMBER = {3},
     PAGES = {1771--1829},
      ISSN = {0002-9947,1088-6850},
   MRCLASS = {32W05 (32A26 32F10)},
  MRNUMBER = {4866351},
MRREVIEWER = {Mehmet\ \c Celik},
       DOI = {10.1090/tran/9315},
       URL = {https://doi-org.ezproxy.library.pfw.edu/10.1090/tran/9315},
}

@article {McNealStein,
    AUTHOR = {McNeal, Jeffery D. and Stein, Elias M.},
     TITLE = {Mapping properties of the {B}ergman projection on convex
              domains of finite type},
   JOURNAL = {Duke Math. J.},
  FJOURNAL = {Duke Mathematical Journal},
    VOLUME = {73},
      YEAR = {1994},
    NUMBER = {1},
     PAGES = {177--199},
      ISSN = {0012-7094},
   MRCLASS = {32H10 (32F15)},
  MRNUMBER = {1257282},
MRREVIEWER = {Der-Chen E. Chang},
       DOI = {10.1215/S0012-7094-94-07307-9},
       URL = {https://doi-org.proxy.lib.ohio-state.edu/10.1215/S0012-7094-94-07307-9},
}

@article {ChangHardy,
    AUTHOR = {Chang, Der-Chen},
     TITLE = {The dual of {H}ardy spaces on a bounded domain in {${\bf
              R}^n$}},
   JOURNAL = {Forum Math.},
  FJOURNAL = {Forum Mathematicum},
    VOLUME = {6},
      YEAR = {1994},
    NUMBER = {1},
     PAGES = {65--81},
      ISSN = {0933-7741},
   MRCLASS = {42B30 (42B25)},
  MRNUMBER = {1253178},
MRREVIEWER = {Michael J. Wilson},
       DOI = {10.1515/form.1994.6.65},
       URL = {https://doi-org.proxy.lib.ohio-state.edu/10.1515/form.1994.6.65},
}

@article {GarabedianSpencer1952,
    AUTHOR = {Garabedian, Paul R. and Spencer, Donald C.},
     TITLE = {Complex boundary value problems},
   JOURNAL = {Trans. Amer. Math. Soc.},
  FJOURNAL = {Transactions of the American Mathematical Society},
    VOLUME = {73},
      YEAR = {1952},
     PAGES = {223--242},
      ISSN = {0002-9947},
   MRCLASS = {30.0X},
  MRNUMBER = {51326},
MRREVIEWER = {P. Dolbeault},
       DOI = {10.2307/1990667},
       URL = {https://doi-org.proxy.lib.ohio-state.edu/10.2307/1990667},
}

@article {Hormander1965,
    AUTHOR = {H{\"{o}}rmander, Lars},
     TITLE = {{$L^{2}$} estimates and existence theorems for the {$\bar
              \partial $} operator},
   JOURNAL = {Acta Math.},
  FJOURNAL = {Acta Mathematica},
    VOLUME = {113},
      YEAR = {1965},
     PAGES = {89--152},
      ISSN = {0001-5962},
   MRCLASS = {35.01 (32.22)},
  MRNUMBER = {179443},
MRREVIEWER = {M. Schechter},
       DOI = {10.1007/BF02391775},
       URL = {https://doi-org.proxy.libraries.rutgers.edu/10.1007/BF02391775},
}

@article {Henkin1969,
    AUTHOR = {Henkin, Gennadi},
     TITLE = {Integral representation of functions which are holomorphic in
              strictly pseudoconvex regions, and some applications},
   JOURNAL = {Mat. Sb. (N.S.)},
    VOLUME = {78 (120)},
      YEAR = {1969},
     PAGES = {611--632},
   MRCLASS = {32.22},
  MRNUMBER = {0249660},
MRREVIEWER = {H. Tornehave},
}

@article {GrauertLieb1971,
    AUTHOR = {Grauert, Hans and Lieb, Ingo},
     TITLE = {Das {R}amirezsche {I}ntegral und die {L}\"{o}sung der {G}leichung
              {$\bar \partial f=\alpha $} im {B}ereich der beschr\"{a}nkten
              {F}ormen},
   JOURNAL = {Rice Univ. Stud.},
  FJOURNAL = {Rice University Studies},
    VOLUME = {56},
      YEAR = {1970},
    NUMBER = {2},
     PAGES = {29--50 (1971)},
      ISSN = {0035-4996},
   MRCLASS = {32.22},
  MRNUMBER = {273057},
MRREVIEWER = {J. Kajiwara},
}

@article {LiebRange1980,
    AUTHOR = {Lieb, Ingo and Range, R. Michael},
     TITLE = {L\"{o}sungsoperatoren f\"{u}r den {C}auchy-{R}iemann-{K}omplex mit
              {${\mathscr C}^{k}$}-{A}bsch\"{a}tzungen},
   JOURNAL = {Math. Ann.},
  FJOURNAL = {Mathematische Annalen},
    VOLUME = {253},
      YEAR = {1980},
    NUMBER = {2},
     PAGES = {145--164},
      ISSN = {0025-5831},
   MRCLASS = {32F15 (32F10)},
  MRNUMBER = {597825},
MRREVIEWER = {G. Berg},
       DOI = {10.1007/BF01578911},
       URL = {https://doi-org.proxy.lib.ohio-state.edu/10.1007/BF01578911},
}

@article {Michel1991,
    AUTHOR = {Michel, Joachim},
     TITLE = {Integral representations on weakly pseudoconvex domains},
   JOURNAL = {Math. Z.},
  FJOURNAL = {Mathematische Zeitschrift},
    VOLUME = {208},
      YEAR = {1991},
    NUMBER = {3},
     PAGES = {437--462},
      ISSN = {0025-5874},
   MRCLASS = {32A25 (32F20 35N15)},
  MRNUMBER = {1134587},
MRREVIEWER = {Harold P. Boas},
       DOI = {10.1007/BF02571538},
       URL = {https://doi-org.proxy.lib.ohio-state.edu/10.1007/BF02571538},
}

@book {SteinBook,
    AUTHOR = {Stein, Elias M.},
     TITLE = {Singular integrals and differentiability properties of
              functions},
    SERIES = {Princeton Mathematical Series, No. 30},
 PUBLISHER = {Princeton University Press, Princeton, N.J.},
      YEAR = {1970},
     PAGES = {xiv+290},
   MRCLASS = {46.38 (26.00)},
  MRNUMBER = {0290095},
MRREVIEWER = {R. E. Edwards},
}

@book {LiebMichelBook,
    AUTHOR = {Lieb, Ingo and Michel, Joachim},
     TITLE = {The {C}auchy-{R}iemann complex},
    SERIES = {Aspects of Mathematics, E34},
      NOTE = {Integral formulae and Neumann problem},
 PUBLISHER = {Friedr. Vieweg \& Sohn, Braunschweig},
      YEAR = {2002},
     PAGES = {x+362},
      ISBN = {3-528-06954-6},
   MRCLASS = {32A26 (32-02 32W05 35N15)},
  MRNUMBER = {1900133},
MRREVIEWER = {Bert Fischer},
       DOI = {10.1007/978-3-322-91608-2},
       URL = {https://doi-org.proxy.lib.ohio-state.edu/10.1007/978-3-322-91608-2},
}

@article {Peters1991,
    AUTHOR = {Peters, Klaus},
     TITLE = {Solution operators for the {$\overline\partial$}-equation on
              nontransversal intersections of strictly pseudoconvex domains},
   JOURNAL = {Math. Ann.},
  FJOURNAL = {Mathematische Annalen},
    VOLUME = {291},
      YEAR = {1991},
    NUMBER = {4},
     PAGES = {617--641},
      ISSN = {0025-5831},
   MRCLASS = {32A25 (32F20 35N15)},
  MRNUMBER = {1135535},
MRREVIEWER = {Harold P. Boas},
       DOI = {10.1007/BF01445231},
       URL = {https://doi-org.proxy.lib.ohio-state.edu/10.1007/BF01445231},
}

@book {NegativeBVPBook,
    AUTHOR = {Roitberg, Yakov},
     TITLE = {Elliptic boundary value problems in the spaces of
              distributions},
    SERIES = {Mathematics and its Applications},
    VOLUME = {384},
      NOTE = {Translated from the Russian by Peter Malyshev and Dmitry
              Malyshev},
 PUBLISHER = {Kluwer Academic Publishers Group, Dordrecht},
      YEAR = {1996},
     PAGES = {xii+415},
      ISBN = {0-7923-4303-4},
   MRCLASS = {35J40 (35D05 46F10 47N20)},
  MRNUMBER = {1423135},
MRREVIEWER = {H. Triebel},
       DOI = {10.1007/978-94-011-5410-9},
       URL = {https://doi-org.proxy.lib.ohio-state.edu/10.1007/978-94-011-5410-9},
}

@article {ZhuoTriebelType,
    AUTHOR = {Zhuo, Ciqiang and Hovemann, Marc and Sickel, Winfried},
     TITLE = {Complex interpolation of {L}izorkin-{T}riebel-{M}orrey spaces
              on domains},
   JOURNAL = {Anal. Geom. Metr. Spaces},
  FJOURNAL = {Analysis and Geometry in Metric Spaces},
    VOLUME = {8},
      YEAR = {2020},
    NUMBER = {1},
     PAGES = {268--304},
   MRCLASS = {46B70 (46E35)},
  MRNUMBER = {4178742},
MRREVIEWER = {Morteza Koozehgar Kalleji},
       DOI = {10.1515/agms-2020-0114},
       URL = {https://doi-org.ezproxy.library.wisc.edu/10.1515/agms-2020-0114},
}

@article{YaoExtensionMorrey,
    AUTHOR = {Yao, Liding},
     TITLE = {Some intrinsic characterizations of
              {B}esov-{T}riebel-{L}izorkin-{M}orrey-type spaces on
              {L}ipschitz domains},
   JOURNAL = {J. Fourier Anal. Appl.},
  FJOURNAL = {The Journal of Fourier Analysis and Applications},
    VOLUME = {29},
      YEAR = {2023},
    NUMBER = {2},
     PAGES = {Paper No. 24, 21},
      ISSN = {1069-5869,1531-5851},
   MRCLASS = {46E35 (42B25 42B35)},
  MRNUMBER = {4575470},
       DOI = {10.1007/s00041-023-10001-x},
       URL = {https://doi.org/10.1007/s00041-023-10001-x},
}

@article{ShiYaoC2,
    AUTHOR = {Shi, Ziming and Yao, Liding},
     TITLE = {Sobolev {$\frac 12$} estimate for the {$\overline\partial$}
              equation on strictly pseudoconvex domains with {$C^2$}
              boundary},
   JOURNAL = {Amer. J. Math.},
  FJOURNAL = {American Journal of Mathematics},
    VOLUME = {147},
      YEAR = {2025},
    NUMBER = {1},
     PAGES = {235--277},
      ISSN = {0002-9327,1080-6377},
   MRCLASS = {32W05},
  MRNUMBER = {4861589},
MRREVIEWER = {S\"onmez\ \c Sahuto\u glu},
       DOI = {10.1353/ajm.2025.a950277},
       URL = {https://doi-org.ezproxy.library.pfw.edu/10.1353/ajm.2025.a950277},
}

@article {YaoCXFinite,
    AUTHOR = {Yao, Liding},
     TITLE = {Sobolev and {H}\"{o}lder estimates for homotopy operators of
              the {$\overline{\partial}$}-equation on convex domains of
              finite multitype},
   JOURNAL = {J. Math. Anal. Appl.},
  FJOURNAL = {Journal of Mathematical Analysis and Applications},
    VOLUME = {538},
      YEAR = {2024},
    NUMBER = {2},
     PAGES = {Paper No. 128238},
      ISSN = {0022-247X,1096-0813},
   MRCLASS = {99-06},
  MRNUMBER = {4739361},
       DOI = {10.1016/j.jmaa.2024.128238},
       URL = {https://doi.org/10.1016/j.jmaa.2024.128238},
}

@book {KressSpectralCompact,
    AUTHOR = {Kress, Rainer},
     TITLE = {Linear integral equations},
    SERIES = {Applied Mathematical Sciences},
    VOLUME = {82},
   EDITION = {Third},
 PUBLISHER = {Springer, New York},
      YEAR = {2014},
     PAGES = {xvi+412},
      ISBN = {978-1-4614-9592-5; 978-1-4614-9593-2},
   MRCLASS = {45A05 (45-02 45L05 47G10 65R20)},
  MRNUMBER = {3184286},
       DOI = {10.1007/978-1-4614-9593-2},
       URL = {https://doi.org/10.1007/978-1-4614-9593-2},
}

@article {GongLanzaniCConvex,
    AUTHOR = {Gong, Xianghong and Lanzani, Loredana},
     TITLE = {Regularity of a {$\overline\partial$}-solution operator for
              strongly {$\bf C$}-linearly convex domains with minimal
              smoothness},
   JOURNAL = {J. Geom. Anal.},
  FJOURNAL = {Journal of Geometric Analysis},
    VOLUME = {31},
      YEAR = {2021},
    NUMBER = {7},
     PAGES = {6796--6818},
      ISSN = {1050-6926,1559-002X},
   MRCLASS = {32W05 (32T15)},
  MRNUMBER = {4289246},
MRREVIEWER = {Harold\ P.\ Boas},
       DOI = {10.1007/s12220-020-00443-w},
       URL = {https://doi.org/10.1007/s12220-020-00443-w},
}

@article {JonesBMO,
    AUTHOR = {Jones, Peter W.},
     TITLE = {Extension theorems for {BMO}},
   JOURNAL = {Indiana Univ. Math. J.},
  FJOURNAL = {Indiana University Mathematics Journal},
    VOLUME = {29},
      YEAR = {1980},
    NUMBER = {1},
     PAGES = {41--66},
      ISSN = {0022-2518,1943-5258},
   MRCLASS = {42A50 (32A35)},
  MRNUMBER = {554817},
MRREVIEWER = {Sun\ Yung A. Chang},
       DOI = {10.1512/iumj.1980.29.29005},
       URL = {https://doi.org/10.1512/iumj.1980.29.29005},
}

@article {Kohn1963,
    AUTHOR = {Kohn, Joseph J.},
     TITLE = {Harmonic integrals on strongly pseudo-convex manifolds. {I}},
   JOURNAL = {Ann. of Math. (2)},
  FJOURNAL = {Annals of Mathematics. Second Series},
    VOLUME = {78},
      YEAR = {1963},
     PAGES = {112--148},
      ISSN = {0003-486X},
   MRCLASS = {32.22 (35.79)},
  MRNUMBER = {153030},
MRREVIEWER = {F.\ Norguet},
       DOI = {10.2307/1970506},
       URL = {https://doi.org/10.2307/1970506},
}

@book {GreinerStein1977,
    AUTHOR = {Greiner, Peter C. and Stein, Elias M.},
     TITLE = {Estimates for the {$\overline \partial $}-{N}eumann problem},
    SERIES = {Mathematical Notes},
    VOLUME = {No. 19},
 PUBLISHER = {Princeton University Press, Princeton, NJ},
      YEAR = {1977},
     PAGES = {iv+195},
      ISBN = {0-691-08013-5},
   MRCLASS = {32F15 (35N15)},
  MRNUMBER = {499319},
MRREVIEWER = {G.\ M.\ Khenkin},
}

@article {Chang1989,
	AUTHOR = {Chang, Der-Chen E.},
	TITLE = {Optimal {$L^p$} and {H}\"{o}lder estimates for the {K}ohn solution
	of the {$\overline\partial$}-equation on strongly pseudoconvex
	domains},
	JOURNAL = {Trans. Amer. Math. Soc.},
	FJOURNAL = {Transactions of the American Mathematical Society},
	VOLUME = {315},
	YEAR = {1989},
	NUMBER = {1},
	PAGES = {273--304},
	ISSN = {0002-9947},
	MRCLASS = {32F20 (35N15)},
	MRNUMBER = {937241},
	MRREVIEWER = {David S. Tartakoff},
	DOI = {10.2307/2001384},
	URL = {https://doi.org/10.2307/2001384},
}

@article {Ovrelid1971,
    AUTHOR = {\O vrelid, Nils},
     TITLE = {Integral representation formulas and {$L\sp{p}$}-estimates for
              the {$\bar \partial $}-equation},
   JOURNAL = {Math. Scand.},
  FJOURNAL = {Mathematica Scandinavica},
    VOLUME = {29},
      YEAR = {1971},
     PAGES = {137--160},
      ISSN = {0025-5521,1903-1807},
   MRCLASS = {32F15},
  MRNUMBER = {324073},
MRREVIEWER = {B.\ M.\ Weinstock},
       DOI = {10.7146/math.scand.a-11041},
       URL = {https://doi.org/10.7146/math.scand.a-11041},
}

@article {Kerzman1971,
    AUTHOR = {Kerzman, Norberto},
     TITLE = {H\"older and {$L\sp{p}$} estimates for solutions of {$\bar
              \partial u=f$} in strongly pseudoconvex domains},
   JOURNAL = {Comm. Pure Appl. Math.},
  FJOURNAL = {Communications on Pure and Applied Mathematics},
    VOLUME = {24},
      YEAR = {1971},
     PAGES = {301--379},
      ISSN = {0010-3640,1097-0312},
   MRCLASS = {32.22 (47.00)},
  MRNUMBER = {281944},
MRREVIEWER = {B.\ M.\ Weinstock},
       DOI = {10.1002/cpa.3160240303},
       URL = {https://doi.org/10.1002/cpa.3160240303},
}

@article {HenkinRomanov1971,
    AUTHOR = {Romanov, Aleksandr and Henkin, Gennadi},
     TITLE = {Exact {H}\"older estimates of the solutions of the {$\bar
              \partial $}-equation},
   JOURNAL = {Izv. Akad. Nauk SSSR Ser. Mat.},
  FJOURNAL = {Izvestiya Akademii Nauk SSSR. Seriya Matematicheskaya},
    VOLUME = {35},
      YEAR = {1971},
     PAGES = {1171--1183},
      ISSN = {0373-2436},
   MRCLASS = {35N15 (32F15)},
  MRNUMBER = {293121},
MRREVIEWER = {J.-M.\ Kantor},
}

@article {RangeSiu1973,
    AUTHOR = {Range, R. Michael and Siu, Yum-Tong},
     TITLE = {Uniform estimates for the {$\bar \partial $}-equation on
              domains with piecewise smooth strictly pseudoconvex
              boundaries},
   JOURNAL = {Math. Ann.},
  FJOURNAL = {Mathematische Annalen},
    VOLUME = {206},
      YEAR = {1973},
     PAGES = {325--354},
      ISSN = {0025-5831,1432-1807},
   MRCLASS = {32F15 (35N15)},
  MRNUMBER = {338450},
MRREVIEWER = {S.\ G.\ Gindikin},
       DOI = {10.1007/BF01355986},
       URL = {https://doi.org/10.1007/BF01355986},
}

@article {Siu1974,
    AUTHOR = {Siu, Yum Tong},
     TITLE = {The {$\bar \partial $} problem with uniform bounds on
              derivatives},
   JOURNAL = {Math. Ann.},
  FJOURNAL = {Mathematische Annalen},
    VOLUME = {207},
      YEAR = {1974},
     PAGES = {163--176},
      ISSN = {0025-5831,1432-1807},
   MRCLASS = {32F15 (35N15)},
  MRNUMBER = {330515},
MRREVIEWER = {S.\ G.\ Gindikin},
       DOI = {10.1007/BF01362154},
       URL = {https://doi.org/10.1007/BF01362154},
}

@book {FollandKohn1972,
    AUTHOR = {Folland, Gerald B. and Kohn, Joseph J.},
     TITLE = {The {N}eumann problem for the {C}auchy-{R}iemann complex},
    SERIES = {Annals of Mathematics Studies},
    VOLUME = {No. 75},
 PUBLISHER = {Princeton University Press, Princeton, NJ; University of Tokyo
              Press, Tokyo},
      YEAR = {1972},
     PAGES = {viii+146},
   MRCLASS = {58G05 (32F15 35N15)},
  MRNUMBER = {461588},
MRREVIEWER = {V.\ G.\ Mazja},
}

@article {BGS1987,
    AUTHOR = {Beals, Richard and Greiner, Peter and Stanton, Nancy},
     TITLE = {{$L^p$} and {L}ipschitz estimates for the
              {$\overline\partial$}-equation and the
              {$\overline\partial$}-{N}eumann problem},
   JOURNAL = {Math. Ann.},
  FJOURNAL = {Mathematische Annalen},
    VOLUME = {277},
      YEAR = {1987},
    NUMBER = {2},
     PAGES = {185--196},
      ISSN = {0025-5831,1432-1807},
   MRCLASS = {58G05 (32F20 35N15)},
  MRNUMBER = {886418},
MRREVIEWER = {David\ S.\ Tartakoff},
       DOI = {10.1007/BF01457358},
       URL = {https://doi.org/10.1007/BF01457358},
}

@article {DispaBesovNorm,
    AUTHOR = {Dispa, Sophie},
     TITLE = {Intrinsic characterizations of {B}esov spaces on {L}ipschitz
              domains},
   JOURNAL = {Math. Nachr.},
  FJOURNAL = {Mathematische Nachrichten},
    VOLUME = {260},
      YEAR = {2003},
     PAGES = {21--33},
      ISSN = {0025-584X,1522-2616},
   MRCLASS = {46E35 (42B35 46F05)},
  MRNUMBER = {2017700},
       DOI = {10.1002/mana.200310101},
       URL = {https://doi.org/10.1002/mana.200310101},
}

@article {Peters1989,
    AUTHOR = {Peters, Klaus},
     TITLE = {Uniform estimates for {$\overline\partial$} on the
              intersection of two strictly pseudoconvex {$C^2$}-domains
              without transversality condition},
   JOURNAL = {Math. Ann.},
  FJOURNAL = {Mathematische Annalen},
    VOLUME = {284},
      YEAR = {1989},
    NUMBER = {3},
     PAGES = {409--421},
      ISSN = {0025-5831,1432-1807},
   MRCLASS = {32F20 (32F15)},
  MRNUMBER = {1001710},
MRREVIEWER = {Alexander\ Nagel},
       DOI = {10.1007/BF01442493},
       URL = {https://doi.org/10.1007/BF01442493},
}

@article {MichelPerotti1990,
    AUTHOR = {Michel, Joachim and Perotti, Alessandro},
     TITLE = {{$C^k$}-regularity for the {$\overline\partial$}-equation on
              strictly pseudoconvex domains with piecewise smooth
              boundaries},
   JOURNAL = {Math. Z.},
  FJOURNAL = {Mathematische Zeitschrift},
    VOLUME = {203},
      YEAR = {1990},
    NUMBER = {3},
     PAGES = {415--427},
      ISSN = {0025-5874,1432-1823},
   MRCLASS = {32F20 (32F15 35N15)},
  MRNUMBER = {1038709},
MRREVIEWER = {Steven\ George\ Krantz},
       DOI = {10.1007/BF02570747},
       URL = {https://doi.org/10.1007/BF02570747},
}

@book {BesovBookVol1,
    AUTHOR = {Besov, Oleg V. and Il'in, Valentin P. and Nikol'ski{\u i}, Sergey M.},
     TITLE = {Integral representations of functions and imbedding theorems.
              {V}ol. {I}},
    SERIES = {Scripta Series in Mathematics},
      NOTE = {Translated from the Russian; edited by Mitchell H. Taibleson},
 PUBLISHER = {V. H. Winston \& Sons, Washington, DC; Halsted Press [John
              Wiley \& Sons], New York-Toronto-London},
      YEAR = {1978},
     PAGES = {viii+345},
      ISBN = {0-470-26540-*},
   MRCLASS = {46E35},
  MRNUMBER = {519341},
}

@book {BesovBookVol2,
    AUTHOR = {Besov, Oleg V. and Il'in, Valentin P. and Nikol'ski{\u i}, Sergey M.},
     TITLE = {Integral representations of functions and imbedding theorems.
              {V}ol. {II}},
    SERIES = {Scripta Series in Mathematics},
      NOTE = {Translated from the Russian; edited by Mitchell H. Taibleson},
 PUBLISHER = {V. H. Winston \& Sons, Washington, DC; Halsted Press [John
              Wiley \& Sons], New York-Toronto-London},
      YEAR = {1979},
     PAGES = {viii+311},
      ISBN = {0-470-26593-0},
   MRCLASS = {46E35},
  MRNUMBER = {521808},
}

@article {Morrey1958,
    AUTHOR = {Morrey, Jr., Charles B.},
     TITLE = {The analytic embedding of abstract real-analytic manifolds},
   JOURNAL = {Ann. of Math. (2)},
  FJOURNAL = {Annals of Mathematics. Second Series},
    VOLUME = {68},
      YEAR = {1958},
     PAGES = {159--201},
      ISSN = {0003-486X},
   MRCLASS = {53.00 (32.00)},
  MRNUMBER = {99060},
MRREVIEWER = {P.\ R.\ Garabedian},
       DOI = {10.2307/1970048},
       URL = {https://doi.org/10.2307/1970048},
}

@incollection {LeitererParametrix,
    AUTHOR = {Leiterer, J\"urgen},
     TITLE = {From local to global homotopy formulas for {$\overline
              \partial$} and {$\overline \partial_b$}},
 BOOKTITLE = {Geometric complex analysis ({H}ayama, 1995)},
     PAGES = {385--391},
 PUBLISHER = {World Sci. Publ., River Edge, NJ},
      YEAR = {1996},
      ISBN = {981-02-2439-7},
   MRCLASS = {32F20},
  MRNUMBER = {1453619},
MRREVIEWER = {Der-Chen\ E.\ Chang},
}

@article {Laurent-ThiebautParametrix,
    AUTHOR = {Laurent-Thi\'ebaut, Christine and Leiterer, J\"urgen},
     TITLE = {Global homotopy formulas on {$q$}-concave {CR} manifolds for
              small degrees},
   JOURNAL = {J. Geom. Anal.},
  FJOURNAL = {Journal of Geometric Analysis},
    VOLUME = {18},
      YEAR = {2008},
    NUMBER = {2},
     PAGES = {511--536},
      ISSN = {1050-6926,1559-002X},
   MRCLASS = {32V20},
  MRNUMBER = {2393269},
MRREVIEWER = {Phillip\ S.\ Harrington},
       DOI = {10.1007/s12220-008-9020-4},
       URL = {https://doi.org/10.1007/s12220-008-9020-4},
}

@article {TriebelLipschitzDomain,
    AUTHOR = {Triebel, Hans},
     TITLE = {Function spaces in {L}ipschitz domains and on {L}ipschitz
              manifolds. {C}haracteristic functions as pointwise
              multipliers},
   JOURNAL = {Rev. Mat. Complut.},
  FJOURNAL = {Revista Matem\'atica Complutense},
    VOLUME = {15},
      YEAR = {2002},
    NUMBER = {2},
     PAGES = {475--524},
      ISSN = {1139-1138,1988-2807},
   MRCLASS = {46E35 (58D15)},
  MRNUMBER = {1951822},
MRREVIEWER = {Gilles\ Carron},
       DOI = {10.5209/rev\_REMA.2002.v15.n2.16910},
       URL = {https://doi.org/10.5209/rev_REMA.2002.v15.n2.16910},
}

@article {Komatsu67,
    AUTHOR = {Komatsu, Hikosaburo},
     TITLE = {Projective and injective limits of weakly compact sequences of
              locally convex spaces},
   JOURNAL = {J. Math. Soc. Japan},
  FJOURNAL = {Journal of the Mathematical Society of Japan},
    VOLUME = {19},
      YEAR = {1967},
     PAGES = {366--383},
      ISSN = {0025-5645,1881-1167},
   MRCLASS = {46.01},
  MRNUMBER = {217557},
MRREVIEWER = {L.\ Bungart},
       DOI = {10.2969/jmsj/01930366},
       URL = {https://doi.org/10.2969/jmsj/01930366},
}

@article {LanzaniStein2013,
    AUTHOR = {Lanzani, Loredana and Stein, Elias M.},
     TITLE = {Cauchy-type integrals in several complex variables},
   JOURNAL = {Bull. Math. Sci.},
  FJOURNAL = {Bulletin of Mathematical Sciences},
    VOLUME = {3},
      YEAR = {2013},
    NUMBER = {2},
     PAGES = {241--285},
      ISSN = {1664-3607,1664-3615},
   MRCLASS = {32A50 (31B10 32A36 42B08)},
  MRNUMBER = {3084008},
MRREVIEWER = {Pamela\ Ryan},
       DOI = {10.1007/s13373-013-0038-y},
       URL = {https://doi-org.ezproxy.library.pfw.edu/10.1007/s13373-013-0038-y},
}

@book {EdmundsTriebelFunctionSpaces,
    AUTHOR = {Edmunds, D. E. and Triebel, H.},
     TITLE = {Function spaces, entropy numbers, differential operators},
    SERIES = {Cambridge Tracts in Mathematics},
    VOLUME = {120},
 PUBLISHER = {Cambridge University Press, Cambridge},
      YEAR = {1996},
     PAGES = {xii+252},
      ISBN = {0-521-56036-5},
   MRCLASS = {46E35 (35J70 35P15 35P20 47B06 47G30)},
  MRNUMBER = {1410258},
MRREVIEWER = {Georgi\ E.\ Karadzhov},
       DOI = {10.1017/CBO9780511662201},
       URL = {https://doi.org/10.1017/CBO9780511662201},
}

@article {GoncalvesHaroskeSkrzypczak,
    AUTHOR = {Gon{\c c}alves, Helena F. and Haroske, Dorothee D. and
              Skrzypczak, Leszek},
     TITLE = {Limiting embeddings of {B}esov-type and
              {T}riebel-{L}izorkin-type spaces on domains and an extension
              operator},
   JOURNAL = {Ann. Mat. Pura Appl. (4)},
  FJOURNAL = {Annali di Matematica Pura ed Applicata. Series IV},
    VOLUME = {202},
      YEAR = {2023},
    NUMBER = {5},
     PAGES = {2481--2516},
      ISSN = {0373-3114,1618-1891},
   MRCLASS = {46E35 (42B35)},
  MRNUMBER = {4634273},
MRREVIEWER = {Jingshi\ Xu},
       DOI = {10.1007/s10231-023-01327-w},
       URL = {https://doi.org/10.1007/s10231-023-01327-w},
}
\end{document}